\documentclass[10pt]{article}
\usepackage{amsfonts}
\usepackage{latexsym,amssymb}
\usepackage{amsmath, amsbsy}
\usepackage{amsopn, amstext}
\usepackage{graphicx, color, epstopdf}
\usepackage{threeparttable}
\usepackage{hyperref}

\usepackage{array}
\usepackage{bm}
\usepackage{multirow}
\usepackage{booktabs}
\usepackage{diagbox}
\usepackage{subcaption}

\def\R{\mathbb{R}}

\newtheorem{Theo}{Theorem}

\newtheorem{Rem}{Remark}

\numberwithin{equation}{section}
\numberwithin{figure}{section}
\numberwithin{table}{section}
\numberwithin{Lem}{section}
\numberwithin{Defi}{section}
\numberwithin{Theo}{section}
\numberwithin{Rem}{section}
 \numberwithin{Coro}{section}

\newcommand{\ds}{\,\mathrm{d}s}
\newcommand{\dx}{\,\mathrm{d}x}
\renewcommand{\i}{\mathrm{i}}

\def\dy{\,\mathrm{d}y}
\def\m{\mathbf{m}}
\def\n{\mathbf{n}}
\def\L{\mathcal{L}}

\title{\bf Perfectly Matched Layers for nonlocal Helmholtz equations II: multi-dimensional cases \thanks{This work is supported in NSFC under grants No. 12071401, 11771035, Natural Science Foundation of Hunan Province No. 2019JJ50572, Natural Science Foundation of Hubei Province No. 2019CFA007 and Xiangtan University 2018ICIP01. }}

\author{
 Yu Du\thanks{Department of Mathematics, Xiangtan University, Hunan, 411105, China({\tt duyu@xtu.edu.cn})}
\and Jiwei Zhang\thanks{School of Mathematics and Statistics, and Hubei Key Laboratory of Computational Science, Wuhan University, Wuhan 430072, China.
({\tt jiweizhang@whu.edu.cn})}
}
\date{}

\begin{document}

\maketitle

\begin{abstract}
Perfectly matched layers (PMLs) are formulated and  applied to numerically solve nonlocal Helmholtz equations in one and two dimensions. In one dimension, we present the PML modifications for the nonlocal Helmholtz equation with general kernels and theoretically show its effectiveness in some sense. In two dimensions, we give the PML modifications in both Cartesian coordinates and polar coordinates. Based on the PML modifications, nonlocal Helmholtz equations are truncated in one and two dimensional spaces, and asymptotic compatibility schemes are introduced to discretize the resulting truncated problems. Finally, numerical examples are provided to study the ``numerical reflections'' by PMLs and demonstrate the effectiveness and validation of our nonlocal PML strategy.



\end{abstract}

\section{Introduction}\label{sec_intro}
In this paper, we consider the computation of the following nonlocal Helmholtz equation
 \begin{align}
\mathcal{L} u(x) - k^2u(x) = f(x), \quad x\in\R^d, \ d=1,2,\label{eq:nonlocalHelmholtz}
\end{align}
where $f(x)\in L^2(\R)$ has a bounded support $\Omega_f$, and $k$ is a constant related to the traditional \emph{wavenumber} for the local Helmholtz equation and the nonlocal operator $\L$ is defined by
\begin{align}
	\mathcal{L} u(x) = \int_{\R^d} \big( u(x)-u(y) \big) \gamma(x,y) \mathrm{d}y.	\label{eq:nonlocalOperator}
\end{align}
Here the kernel function $\gamma(x,y)$ satisfies
\begin{align}
\gamma(x,y)\geq0\quad \mathrm{and}\quad \gamma(x,y)=\gamma(y,x). \label{eq:symkernel}
\end{align}


Note that the nonlocal Helmholtz equation shown above is defined in the whole space $\R^d$, but the far field boundary condition is not provided. The original PML technique is proposed for the standard PDEs \cite{Berenger1994} with some far field boundary conditions, such as the local Helmholtz equation and the Maxwell's equation. A typical example is the Sommerfeld radiation condition for the local Helmholtz equation in a homogeneous medium. However, the definition of the far field boundary condition for the nonlocal Helmholtz equation still remains open. Therefore, in this paper we simply assume a suitable boundary condition at infinity is imposed to exclude energy incoming from infinity and only to allow energy outgoing to infinity. 

While most of works are carried out for the simulation of nonlocal problems with free or fixed boundary conditions on bounded domains \cite{DUGunLeZhou,TianDu,TianDu2,ZhouDu,macek2007peridynamics,chen2011continuous,emmrich2007analysis}, there are applications where the simulation in an infinite medium might be more reasonable, such as wave or crack propagation in the whole space. 
However, the unboundedness of the spatial domain presents a new computational challenge since the traditional grid-based methods such as finite element and finite difference methods generally lead to an infinite number of degrees of freedom, which cannot be solved numerically. To overcome the difficulty of unboundedness of spatial domains, one successful approach is the artificial boundary method (ABM), see monograph \cite{han2013artificial}. The main idea of ABM is to introduce an artificial boundary to limit a finite domain of interest, and then design suitable absorbing boundary conditions (ABCs) to absorb the impinging wave at artificial boundaries. Recently, much progress has been made for nonlocal problems \cite{DuHanZhangZheng, ZYZD16, ZhengHuDuZhang, wildman2012a}, including the nonlocal PMLs for peridynamics. PML was originally proposed by Ber\' enger \cite{Berenger1994} and has been a popular and effective method for solving wave scattering problems \cite{antoine2019towards,berenger1996three,collino1998the,ls98,cw03,Chen2005,bw05,li2018fem,bramble2006analysis,hohage2003solving,jiang2017an,jiang2018convergence,duan2020exponential,zhou2018an,jiang2019an}. 

There have been some PMLs for nonlocal problems \cite{antoine2019towards,DuZhangNonlocal1,wildman2012a}.  Wildman and Gazonas \cite{wildman2012a} presented PMLs for peridynamics by treating the nonlocal kernel as the convolution of the displacement with the second derivative of a nascent Dirac delta distribution and making the substitution for the nonlocal divergence and its adjoint. Antoine and Lorin \cite{antoine2019towards} proposed PMLs for time-dependent space fractional PDEs by replacing the fractional derivative operators with the corresponding new complex-valued ones. Du and Zhang \cite{DuZhangNonlocal1} proposed a new PML for the one-dimensional nonlocal Helmholtz equation with radial kernel and some special \emph{wavenumber} $k$. The main idea of deriving PML in \cite{DuZhangNonlocal1} is to rewrite the nonlocal equation into its weak form, and analytically continue the nonlocal solution and equation of integral form to complex $\tilde x,\tilde y$ contours. After that, coordinate transformations are performed to express the complexes $\tilde x, \tilde y$ as functions of the corresponding real coordinates. In the new coordinates, the nonlocal PML equation are presented with real coordinates and complex materials. After performing the PML transformation of the nonlocal Helmholtz equation, the solution is unchanged in the region of interest (small $x$) and exponentially decaying in the outer region (large $x$).

The goal of this paper is to construct PMLs as ABCs for nonlocal Helmholtz equations with more general kernels and \emph{wavenumber} $k$ in multi-dimensions. To do so, we extend the strategy of constructing PML developed in \cite{DuZhangNonlocal1} to obtain a nonlocal PML equation directly from the weak form of the nonlocal Helmholtz equation. Comparing with the results in \cite{DuZhangNonlocal1}, we here consider the construction of PMLs for more general kernels and \emph{wavenumber} $k$, and present the general form of nonlocal PML modifications for 2D case. Specially, two kinds of nonlocal PML modifications for 2D nonlocal equations are introduced: one is to stretch the Cartesian coordinate and another is to stretch the radial coordinate in polar coordinates. In addition, the properties of the nonlocal PML solution in 1D are studied for a more general analytic continuation. To do so, we introduce a weighted average value $u^a$ of the nonlocal Helmholtz solution defined by \eqref{eq:barqdef} for radial kernels. Our theorem \ref{theo31} shows that the analytic continuation of $u^a$ in complex plane decays exponentially in PML layers, which depends on the \emph{wavenumber} $k$, the kernel $\gamma$ and the PML coefficients. Particularly, a refined estimate of the PML solution is further established based on the exact formula of the corresponding Green's function for a typical kernel $\gamma(x,y) = \gamma_r(x-y)$ defined as
\begin{equation} \label{k1}
\gamma_r(x-y)=\frac{1}{2c_\gamma^3} e^{-\frac{|x-y|}{c_\gamma}},
\end{equation}
where $c_\gamma$ is a positive constant. These estimates theoretically show that our PMLs are efficient for the one-dimensional equations all \emph{wavenumber} $k$ in some sense, i.e., these PMLs can be applied directly for all $k$, which is an improvement of our previous work \cite{DuZhangNonlocal1}.


The outline of this paper is organized as follows. In Section~\ref{sec:NPML}, nonlocal PML equations are derived by using complex coordinate transforms. In Section~\ref{sec:EOPT1}, the efficiency of our PMLs in 1D is theoretically studied for some radical kernels. In Section~\ref{sec:NSPE}, the nonlocal computational region is truncated by putting homogeneous Dirichlet boundary conditions, and AC schemes are introduced for numerically solving the resulting truncated problems. In Section~\ref{sec:ne}, numerical examples in 1D and 2D are provided to study the ``numerical reflections'' caused by PML modifications and demonstrate the effectiveness of our PMLs.

\section{Nonlocal Perfectly Matched Layers} \label{sec:NPML}
The nonlocal PML has been studied in \cite{DuZhangNonlocal1} for one-dimensional case with the radial kernels, i.e., $\gamma(x,y)=\gamma_r(x-y)$. Here we extend the idea to derive nonlocal PMLs for both 1D and 2D cases with the general kernel $\gamma(x,y)$ and a more general analytical continuation.

\subsection{The nonlocal PML in one dimension}

Let $\Omega=[-l,l]\ (l>0)$ be a bounded domain including $\Omega_f$, i.e., $\Omega_f\subset\Omega$ (see Figure~\ref{fig:domainPMLa}). Define the following analytic continuation to the complex plane by
\begin{align}
	\tilde x:=\int_0^x \alpha(t)\mathrm{d}t= \int_0^x \Big(1 + \frac{z}{k}\sigma(t)\Big) \mathrm{d}t,\qquad \tilde y:=\int_0^y \alpha(t)\mathrm{d}t= \int_0^y \Big(1 + \frac{z}{k}\sigma(t)\Big) \mathrm{d}t. \label{eq:complexStreching} 
\end{align}
Here the absorption function $\sigma(t)=0$ in $\Omega$, and is positive inside the PMLs, i.e., $\sigma(t)>0$ in $\R\setminus\Omega$. The PML coefficient $z$ is complex such as taking $\i$ or $1+\i$. This mapping has the effect of transforming traveling waves of the form $e^{\i kx}$ into evanescent waves of the form $e^{ik\Re\tilde x}e^{-k\Im \tilde x}$ when the imaginary part of $z$ is positive. In \cite{DuZhangNonlocal1}, we chose $z$ as a pure imaginary number and have shown the efficiency of the variable changes for the nonlocal Helmholtz equation with certain \emph{wavenumber} $k$. Here we extend the range of $z$ to a complex number in the plane $\Im z>0$, and will show that this change is efficient for all \emph{wavenumber} $k$ in the Theorem \eqref{theo31} below. The following derivation of nonlocal PML modifications is using the same methodology as that in \cite{DuZhangNonlocal1}, and can be extended to the two dimensional PMLs naturally. To under the philosophy of the method in \cite{DuZhangNonlocal1}, let us rewrite Eq.~\eqref{eq:nonlocalHelmholtz} in the weak form
\begin{align}
	(\mathcal{L} u,v) - k^2(u,v) = (f,v),\quad \forall v\in C_0^\infty(\R), \label{eq:weakform}
\end{align}
where $(\cdot,\cdot)$ denotes the inner product in the complex valued $L^2$-space, and 
\begin{align}
	(\mathcal{L} u,v) = \frac12 \int_\R \int_\R \big( u(x)-u(y) \big) \big( \bar v(x)-\bar v(y) \big) \gamma(x,y) \mathrm{d}y\mathrm{d}x.\label{eq:weakcall}
\end{align}
Here $\bar{v}(x)$ represents the complex conjugate of $v(x)$. By replacing
\begin{align*}
	&x\to\tilde x,\ y\to\tilde y,\ \mathrm{and}\ \mathrm{d}x \rightarrow \frac{\partial \tilde x}{\partial x}\dx=\alpha(x)\mathrm{d}x,\ \mathrm{d}y\rightarrow \frac{\partial \tilde y}{\partial y}\dy=\alpha(y)\mathrm{d}y,
\end{align*}
the weak form \eqref{eq:weakcall} can be extended to the following weak form with PML modifications
\begin{align}
	 & \frac12 \int_\R \int_\R \big( \tilde u(x)-\tilde u(y) \big) \big( \tilde{\bar v}(x) - \tilde{\bar v}(y) \big) \gamma(\tilde x,\tilde y) \alpha(x)\alpha(y)\mathrm{d}y\mathrm{d}x \notag\\
	 &\quad\quad\quad\quad\quad\quad\quad\quad\quad\quad\quad\quad- k^2\int_\R \tilde u(x)\tilde{\bar v}(x)\alpha(x)\mathrm{d}x = \int_\R f(x)\tilde{\bar v}(x)\alpha(x)\mathrm{d}x, \label{eq:PML} 
\end{align}
where $\tilde u(x):=u(\tilde x)$, $\tilde u(y):=u(\tilde y)$, $\tilde{\bar v}(x):=\bar v(\tilde x)$ and $\tilde{\bar v}(y):=\bar v(\tilde y)$.
Using the facts that $\alpha(x)=1,\ \tilde x=x,\ \forall x\in \Omega$ and $\mathrm{supp} f(x)\subset \Omega$, we have the strong form of \eqref{eq:PML}
\begin{align} 
	\tilde{\mathcal{L}}\tilde u(x) - k^2\alpha(x)\tilde u(x) = f(x), \label{eq:PMLeq}
\end{align}
where the nonlocal operator with PML modifications is given by
\begin{align}
	\tilde{\mathcal{L}}\tilde u(x) = \int_\R \big( \tilde u(x)-\tilde u(y) \big) \gamma(\tilde x,\tilde y) \alpha(x)\alpha(y)\mathrm{d}y. \label{eq:nonlocalPMLoperator}
\end{align}
We point out that the new kernel $\gamma(\tilde x,\tilde y) \alpha(x)\alpha(y)$ in above PML modifications is still symmetric but complex comparing with \ref{eq:symkernel}.

\begin{figure}[htbp]
\centering
\begin{subfigure}{.32\textwidth}
 \centering
 \includegraphics[width=\textwidth]{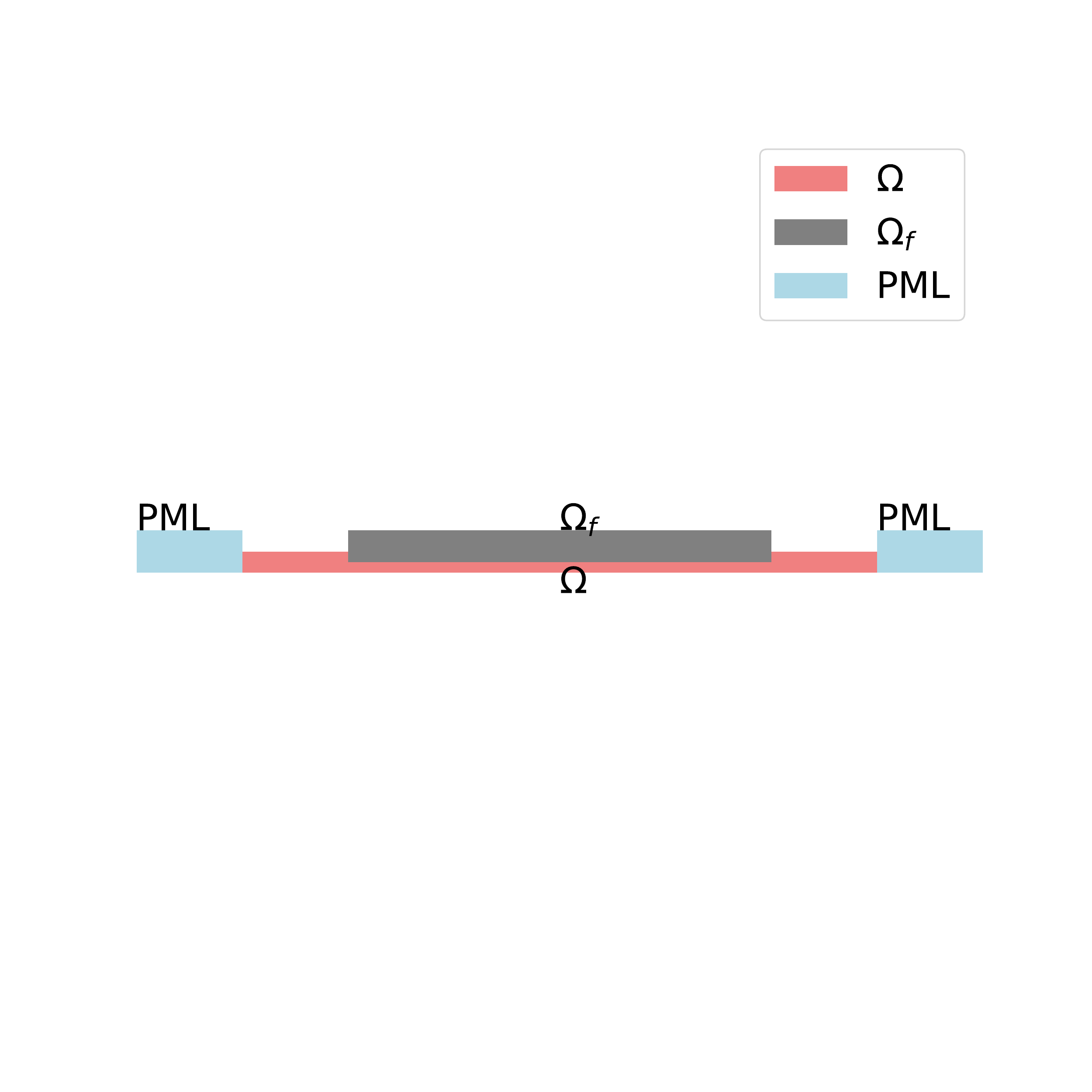} 
 \caption{1D}
 \label{fig:domainPMLa}
\end{subfigure}
\begin{subfigure}{.32\textwidth}
 \centering
 \includegraphics[width=\textwidth]{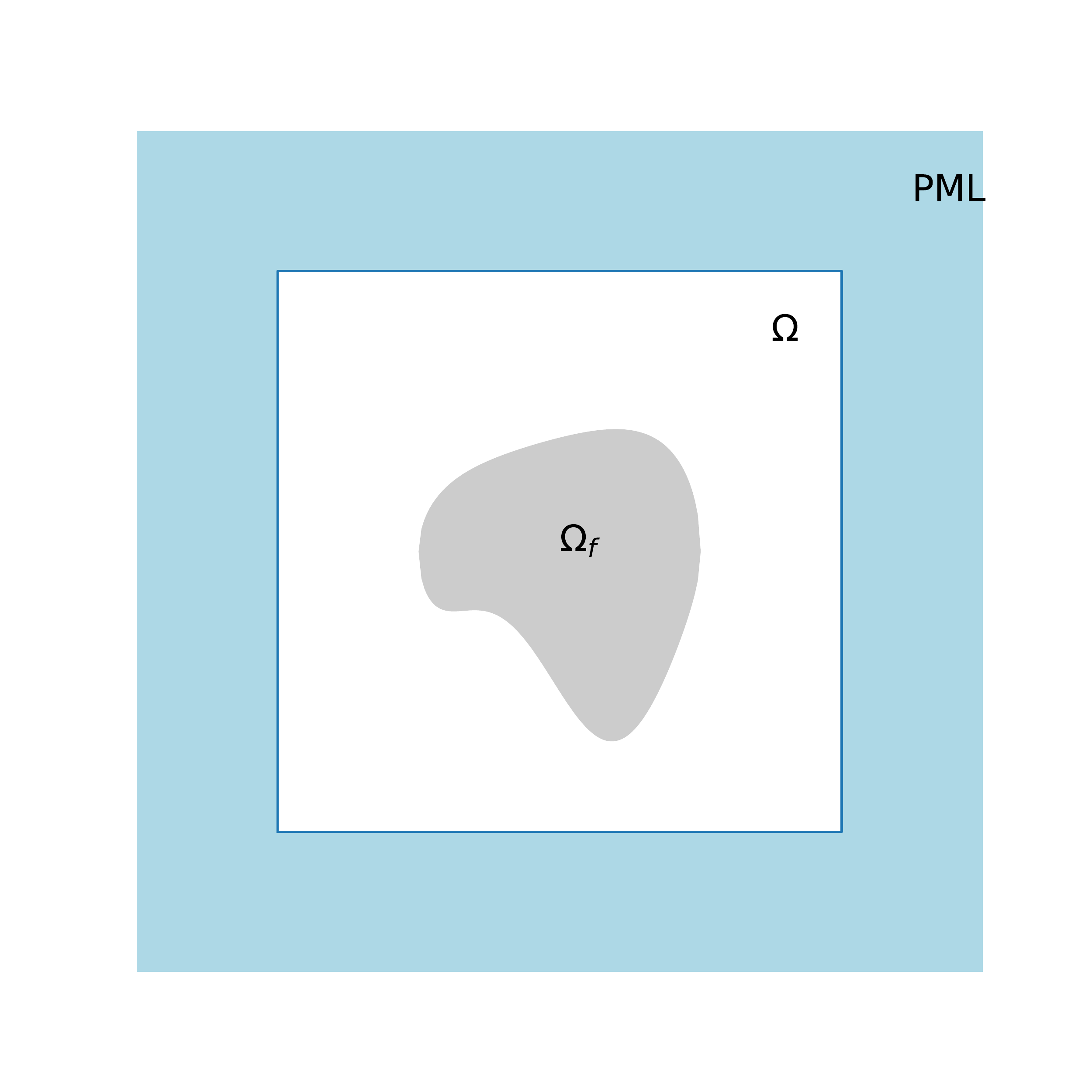} 
 \caption{2D}
 \label{fig:domainPMLb}
\end{subfigure}
\begin{subfigure}{.32\textwidth}
 \centering
 \includegraphics[width=\textwidth]{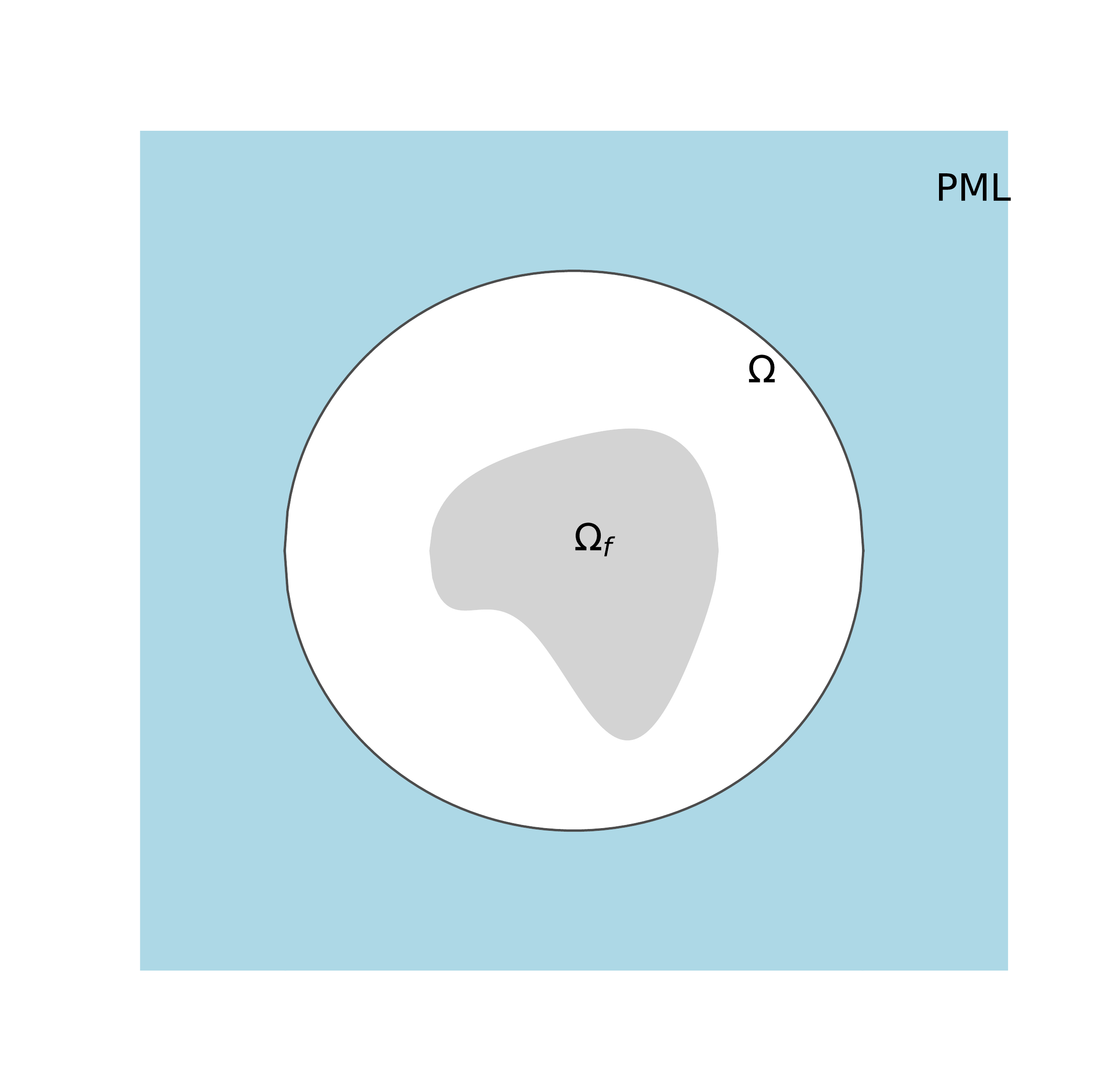} 
 \caption{2D}
 \label{fig:domainPMLc}
\end{subfigure}
	\caption{$\Omega_f$, $\Omega$ and the PML ``absorbing'' regions. In the ``absorbing'' regions, the PML absorption coefficients $\sigma$ for 1D case, $\sigma_j \;(j=1,2)$ for Cartesian coordinates and $\sigma_r$ for polar coordinates in 2D) are positive. In $\Omega$, these PML absorption coefficients are zero.} \label{fig:domainPML}
\end{figure}

\subsection{Nonlocal PMLs in two dimensions}
We now consider the PMLs for two-dimensional nonlocal equations. Let $\Omega$ be a bounded domain such as the square $\Omega=[-l,l]^2$ and the circle $\Omega=\{x\in\R^2:|x|\leq l_r\}$ (Figures~\ref{fig:domainPMLb} and \ref{fig:domainPMLc}).

Denote $\tilde x=(\tilde x_1(x),\tilde x_2(x))$ and $\tilde y=(\tilde y_1(y),\tilde y_2(y))$ by the variable changes of $x=(x_1,x_2)$ and $y=(y_1,y_2)$, respectively, which satisfy $\tilde x=x,\tilde y=y\ \forall x, y\in \Omega$ and are complex in the PML layers $\R^2\setminus \Omega$. Here we consider two kinds of variable changes: one is obtained by stretching Cartesian coordinates and the other is by stretching the radial coordinate in polar coordinates.

The weak form of Eq.~\eqref{eq:nonlocalHelmholtz} in two dimensions is given as 
\begin{align} \label{wf2D}
\frac12 \int_{\R^2}\int_{\R^2} \big[u(x)-u(y)\big]\big[\bar v(x)-\bar v(y)\big]\gamma(x, y) \dy\dx - k^2 \int_{\R^2} u(x)\bar v(x) \dx = \int_{\R^2} f(x)\bar v(x) \dx.
\end{align}
By replacing
\begin{align*}
x\to \tilde x,\ 
y\to\tilde y,\ 
\dx\to\frac{\partial \tilde x}{\partial x}\dx=\begin{vmatrix}
\frac{\partial \tilde x_1}{\partial x_1} & \frac{\partial \tilde x_2}{\partial x_1}\\
\frac{\partial \tilde x_1}{\partial x_2} & \frac{\partial \tilde x_2}{\partial x_2}
\end{vmatrix}\dx,\ 
\dy\to\frac{\partial \tilde y}{\partial y}\dy=\begin{vmatrix}
\frac{\partial \tilde y_1}{\partial y_1} & \frac{\partial \tilde y_2}{\partial y_1}\\
\frac{\partial \tilde y_1}{\partial y_2} & \frac{\partial \tilde y_2}{\partial y_2}
\end{vmatrix}\dy,
\end{align*}
the weak form of \eqref{wf2D} is transformed to the following nonlocal weak form with PML modifications 
\begin{align}
\frac12 \int_{\R^2}\int_{\R^2} \big[u(\tilde x)-u(\tilde y)\big]\big[\bar v(\tilde x) &- \bar v(\tilde y)\big]\gamma(\tilde x,\tilde y) \frac{\partial\tilde x}{\partial x}\frac{\partial\tilde y}{\partial y}\dy\dx\notag \\
&- \int_{\R^2} k^2 u(\tilde x)\bar v(\tilde x) \frac{\partial\tilde x}{\partial x}\dx = \int_{\R^2} f(\tilde x)\bar v(\tilde x) \frac{\partial\tilde x}{\partial x}\dx.
\end{align}
The corresponding strong form of the nonlocal PML equation is given by
\begin{align}
\tilde{\mathcal{L}} \tilde u(x)- k^2\frac{\partial\tilde x}{\partial x} \tilde u(x)= f(x), \label{eq:nonlocalPML2Deq}
\end{align}
where
\begin{align}
\tilde{\mathcal{L}} \tilde u(x)=\int_{\R^2} \big[\tilde u(x)-\tilde u(y)\big]\gamma(\tilde x,\tilde y) \frac{\partial\tilde x}{\partial x}\frac{\partial\tilde y}{\partial y}\dy\label{eq:nonlocalPML2Dop}
\end{align}
with the fact that the right hand side is still $f(x)$ since $\tilde x=x$ in $\Omega$ and $\mathrm{supp}f\subset\Omega$.

Next we introduce the PMLs in Cartesian coordinates and polar coordinates, rescpectively.
\subsubsection{The PML in Cartesian coordinates}
We first introduce the PML obtained by stretching Cartesian coordinates. For simplicity, we choose $\Omega$ as a square centered at $(0,0)$, i.e., $\Omega=[-l,l]^2$, where $l$ is chosen such that $\Omega_f\subset\Omega$. The change of variables is given by
$$\tilde x=(\tilde x_1(x_1), \tilde x_2(x_2)),\quad\tilde y=(\tilde y_1(y_1),\tilde y_2(y_2)),$$
where
\begin{align}
\tilde x_j(x_j) =& \int_0^{x_j} \alpha_j(t)\mathrm dt = x_j + \frac{z_j}{k}\int_0^{x_j} \sigma_j(t)\mathrm dt,\\
\tilde y_j(y_j) =& \int_0^{y_j} \alpha_j(t)\mathrm dt = y_j + \frac{z_j}{k}\int_0^{y_j} \sigma_j(t)\mathrm dt.
\end{align}
Here $\sigma_j(t)=0 \; (j=1,2)$ in $[-l,l]$ and are positive in $(-\infty,-l)\cup(l,+\infty)$, and $z_j\; (j=1,2)$ are complex parameters.

By simple calculations we get
\begin{align}
\begin{pmatrix}
\frac{\partial \tilde x_1}{\partial x_1} & \frac{\partial \tilde x_2}{\partial x_1}\\
\frac{\partial \tilde x_1}{\partial x_2} & \frac{\partial \tilde x_2}{\partial x_2}
\end{pmatrix}=
\begin{pmatrix}
\alpha_1(x_1) & 0\\
0 & \alpha_2(x_2)
\end{pmatrix}
\Longrightarrow
\frac{\partial \tilde x}{\partial x} = \alpha_1(x_1)\alpha_2(x_2),
\end{align}
\begin{align}
\begin{pmatrix}
\frac{\partial \tilde y_1}{\partial y_1} & \frac{\partial \tilde y_2}{\partial y_1}\\
\frac{\partial \tilde y_1}{\partial y_2} & \frac{\partial \tilde y_2}{\partial y_2}
\end{pmatrix}=
\begin{pmatrix}
\alpha_1(y_1) & 0\\
0 & \alpha_2(y_2)
\end{pmatrix}
\Longrightarrow
\frac{\partial \tilde y}{\partial y} = \alpha_1(y_1)\alpha_2(y_2).
\end{align}
By \eqref{eq:nonlocalPML2Deq}--\eqref{eq:nonlocalPML2Dop} the nonlocal PML equation is given by
\begin{align}
\int_{\R^2} \big[\tilde u( x)-\tilde u(y)\big]\gamma(\tilde x,\tilde y)\alpha_1(x_1)\alpha_2(x_2)\alpha_1(y_1)\alpha_2(y_2)\dy - k^2 \alpha_1(x_1)\alpha_2(x_2) \tilde u( x) = f( x).\label{eq:carPML}
\end{align}

\subsubsection{The PML in polar coordinates}
Here we consider the nonlocal PML by stretching the radial coordinate with 
$$\Omega=\{(x,y):\sqrt{x^2+y^2}\leq l_r\},$$ where $l_r$ is a constant such that $\Omega_f\subset\Omega$ (see Figure~\ref{fig:domainPMLc}). Denote by $(r_x,\theta_x)$ and $(r_y,\theta_y)$ the polar coordinates of $x,y\in\R^2$, respectively. We define the change of variables $r_x$ and $r_y$ by
\begin{align}
\tilde r_x =& \int_0^{r_x} \alpha_r(t)\mathrm dt=r_x+\frac{z}{k}\int_0^{r_x} \sigma_r(t)\mathrm dt,\\
\tilde r_y =& \int_0^{r_ y} \alpha_r(t)\mathrm dt=r_y+\frac{z}{k}\int_0^{r_ y} \sigma_r(t)\mathrm dt.
\end{align}
Here $\sigma_r(t)=0$ in $[0,l_r]$ and is positive in $(l_r,+\infty)$, and $z$ is a complex parameter. Then the complex coordinate stretching is given by
\begin{align}
x_1\to\tilde x_1(x)=&\tilde r_x\cos\theta_x,\ x_2\to\tilde x_2(x)=\tilde r_x\sin\theta_x,\\
y_1\to\tilde y_1(y)=&\tilde r_y\cos\theta_y,\ y_2\to\tilde y_2(y)=\tilde r_y\sin\theta_y.
\end{align}
By a simple calculation we have 
\begin{align}
\begin{pmatrix}
\frac{\partial \tilde x_1}{\partial x_1} & \frac{\partial \tilde x_2}{\partial x_1}\\
\frac{\partial \tilde x_1}{\partial x_2} & \frac{\partial \tilde x_2}{\partial x_2}
\end{pmatrix} = \begin{pmatrix}
\cos\theta_x & -\sin\theta_x\\
\sin\theta_x & \cos\theta_x
\end{pmatrix}
\begin{pmatrix}
\alpha_r(r_x) & 0\\
0 &\beta(r_x)
\end{pmatrix}
\begin{pmatrix}
\cos\theta_x & \sin\theta_x\\
-\sin\theta_x & \cos\theta_x
\end{pmatrix},
\end{align}
where $\beta(r_x)=\tilde r_x/r_x$, which implies
\begin{align}
\frac{\partial \tilde x}{\partial x} = \alpha_r(r_x)\beta(r_x).
\end{align}
Similarly, we have
\begin{align}
\frac{\partial \tilde y}{\partial y} = \alpha_r(r_ y)\beta(r_ y).
\end{align}

By \eqref{eq:nonlocalPML2Deq}--\eqref{eq:nonlocalPML2Dop}, we get the nonlocal PML equation
\begin{align}
\int_{\R^2} \big[\tilde u( x)-\tilde u(y)\big]\gamma(\tilde x,\tilde y) \alpha_r(r_x)\beta(r_x) \alpha_r(r_ y)\beta(r_ y)\dy - k^2 \alpha_r(r_x)\beta(r_x)\tilde u( x) = f(x). \label{eq:polPML}
\end{align}


We point out the derivation of PMLs for higher dimensional nonlocal Helmholtz equations can be directly obtained by using the same idea as above, which is omitted here.

\section{The exponentially decaying waves} \label{sec:EOPT1}
We here show the efficiency of the PML technique for radial kernels in one dimension, i.e., $\gamma(x,y)=\gamma_r(x-y)$, satisfies
\begin{itemize}
\item nonnegativeness: $\gamma_r(s)\geq0$;
\item symmetry in $s$: $\gamma_r(s)=\gamma_r(-s)$;
\item finite horizon: $\gamma_r(s)=0$ for $|s|>\delta$;
\item the second moment condition: $\frac12\int_\R s^2\gamma_r(s)\ds=1$.
\end{itemize}
To the end, we introduce the following function, see \cite{DuZhangNonlocal1} 
\begin{align}
	G_{x_0}(x) = \begin{cases}
		C_1(x_0) e^{-\i\tilde kx}, & \quad x\leq x_0 \quad \text{with} \quad C_1(x_0)=- \frac{e^{\i \tilde k x_0}}{2\i \tilde k}, \\
		C_2(x_0) e^{\i\tilde kx}, & \quad x>x_0\quad \text{with} \quad C_2(x_0)=- \frac{e^{-\i \tilde k x_0}}{2\i \tilde k}, 
	\end{cases}
	\label{eq:green}
\end{align}
where $\tilde k$ which has the nonnegative imaginary part is the solution of the following identity 
\begin{align} 
	\int_\R \left(1-e^{\i\tilde k s}\right)\gamma_r(s)\ds = k^2.\label{eq:relation}
\end{align}
Set an average of $u$ as 
\begin{align}
	u^a(x) := \int_{\R} u(t+x)\kappa(t)\mathrm dt, \label{eq:barqdef}
\end{align}
where the weight $\kappa(t)$ is an even function given by 
\begin{align}
	\kappa(t) = -\frac{1}{\tilde k} \int_{t}^{+\infty} \sin \tilde k (t- s) \gamma_r(s) \ds,\quad \mathrm{for}\ t>0.
\end{align}
In addition, we set $u^{e}$ is the solution to \eqref{eq:nonlocalHelmholtz} with a typical kernel \eqref{k1}. 

In analogy to \cite{DuZhangNonlocal1}, the average $u^a$ and the particular solution $u^e$ are given respectively by
\begin{align}
&u^a(x) = \int_{\R} G_{x}(y)f(y)\mathrm dy, \label{eq:baruExp}\\
&u^e(x) = \int_\R \left(\frac{1}{\left(1-(c_\gamma k)^2\right)^2} G_{x}(y) + \frac{c_\gamma^2}{1-(c_\gamma k)^2} \mathcal{D}(y-x)\right) f(y)\mathrm{d}y.\label{eq:GF}
\end{align}
Here $\mathcal{D}(x-x_0)$ represents the Dirac delta function. In~\eqref{eq:GF}, we can directly calculate the relationship between $\tilde k$ and $k$ by $\tilde k=k\sqrt{1/(1-(c_\gamma k)^2)}$ from ~\eqref{eq:relation}.

Note that Eq.~\eqref{eq:GF} actually gives the Green's function of Eq. \eqref{eq:nonlocalHelmholtz} with the kernel \eqref{k1}, namely,
\begin{align*}
G_x^e(y)= \frac{1}{\left(1-(c_\gamma k)^2\right)^2} G_{x}(y) + \frac{c_\gamma^2}{1-(c_\gamma k)^2} \mathcal{D}(y-x).
\end{align*}
This nonlocal Green's function has the asymptotic property to the Green's function of the local Helmholtz equation with Sommerfeld radiation boundary condition \cite{ms10}. This is to say, the nonlocal Green's function will converge to the local Green's function as $c_\gamma\to0$.


We have the following results to show the PML is an efficient technique in 1D for some $z$.
\begin{Theo} \label{theo31}
Denote $\Omega=[-l,l]$ and $\eta(x)=\left|\int_0^x\sigma(t)\mathrm{d}t\right|$. Let $z=z_1+\i z_2$ $(z_1\geq0,z_2>0)$. When ${z_1}/{z_2}$ is large enough, the PML solutions $u^a(\tilde x)$ and $u^e(\tilde x)$ decay exponentially as $|x|\to\infty$:
\begin{enumerate}
\item If $\tilde k\in \R$ is positive, it holds for any $z_1\geq0,\ z_2>0$ that
\begin{align}
|u^a(\tilde x)| &\leq \frac{\sqrt{l}}{\sqrt{2}\tilde k} e^{-\frac{\tilde k}{k}z_2|\eta|} \|f\|_{L^2(\Omega)},\qquad\qquad\qquad \mathrm{for}\ |x|>l,\label{eq:case1eq1}\\
|u^e(\tilde x)| &\leq \frac{\sqrt{l}}{\sqrt{2}\tilde k\left(1-(\delta k)^2\right)^2} e^{-\frac{\tilde k}{k}z_2|\eta|} \|f\|_{L^2(\Omega)},\quad \mathrm{for}\ |x|>l.\label{eq:case1eq2}
\end{align}
\item If $\tilde k$ is a complex constant with positive imaginary part, i.e., $\Im(\tilde k)>0$, for any $\lambda\in(0,1)$, when $\frac{z_1}{z_2}\geq - \frac{1}{1-\lambda}\frac{\Re(\tilde k)}{\Im(\tilde k)}$, it holds 
\begin{align}
|u^a(\tilde x)| &\leq \frac{\sqrt{l}}{\sqrt{2}|\tilde k|} e^{-\lambda \Im(\tilde k)(|x+\frac{z_1}{k}\eta|-l)} \|f\|_{L^2(\Omega)},\qquad\qquad\qquad \mathrm{for}\ |x|>l,\label{eq:case2eq1}\\
|u^e(\tilde x)| &\leq \frac{\sqrt{l}}{\sqrt{2}|\tilde k|\left(1-(\delta k)^2\right)^2} e^{-\lambda \Im(\tilde k)(|x+\frac{z_1}{k}\eta|-l)} \|f\|_{L^2(\Omega)},\quad \mathrm{for}\ |x|>l.\label{eq:case2eq2}
\end{align}
\end{enumerate}
\end{Theo}
\begin{proof}
Inequalities~\eqref{eq:case1eq1} and \eqref{eq:case1eq2} can be obtained by the same arguments in \cite[Theorem 1]{DuZhangNonlocal1}, and we here omit the details. We mainly prove \eqref{eq:case2eq1}. For $x>l$,
\begin{align}
	|u^a(\tilde x)| = & \left|\int_\Omega G_{\tilde x}(y)f(y) \mathrm{d}y\right| = \left|\int_{\Omega} C_0 e^{\i\tilde k\big((\Re(\tilde x)-y)+\i \Im(\tilde x)\big)} f(y) \mathrm{d}y\right| \notag\\
	\leq         & \frac{1}{2|\tilde k|} \left(\int_\Omega e^{-2\big(\Im(\tilde k)(\Re(\tilde x)-y)+\Re(\tilde k) \Im(\tilde x)\big)} \mathrm{d}y \right)^\frac12 \|f\|_{L^2(\Omega)}\notag\\
	\leq& \frac{\sqrt{2l}}{2|\tilde k|} e^{-\big(\Im(\tilde k)(\Re(\tilde x)-l)+\Re(\tilde k)\Im(\tilde x)\big)} \|f\|_{L^2(\Omega)}.\label{eq:prfeq1}
\end{align}
Note that $\Re(\tilde x)=x+\frac{z_1}{k}\eta$, $\Im(\tilde x)=\frac{z_1}{k}\eta$ and $\eta(x)>0$ when $x>l$. When 
$$\frac{x-l}{\eta}\frac{k}{z_2}+\frac{z_1}{z_2}\geq - \frac{1}{1-\lambda}\frac{\Re(\tilde k)}{\Im(\tilde k)},$$
we have by a simple calculation that 
\begin{align}
\Im(\tilde k)(\Re(\tilde x)-l)+\Re(\tilde k)\Im(\tilde x) = \Im(\tilde k)(x+\frac{z_1}{k}\eta-l)+\Re(\tilde k)\frac{z_2}{k}\eta \geq \lambda \Im(\tilde k)(x+\frac{z_1}{k}\eta-l).\label{eq:prfeq2}
\end{align}
Combining \eqref{eq:prfeq1} and \eqref{eq:prfeq2}, we get \eqref{eq:case2eq2} for $x>l$. 

Similarly, we also can get \eqref{eq:case2eq1} for $x<-l$ in analogy to \eqref{eq:prfeq1}--\eqref{eq:prfeq2}.
\end{proof}

\begin{Rem} We point out that 
\begin{enumerate}
\item For the case of $\tilde k\in\R^+$, the behaviors of $u^a(\tilde x)$ and $u^e(\tilde x)$ have been studied for the PML coefficient $z$ with $z_1=0$ in \cite[Theorems 1,3]{DuZhangNonlocal1}, which have the same estimates as \eqref{eq:case1eq1}-\eqref{eq:case1eq2}. This implies that the real part of $z$ has no influence on the wave decaying in theory in this situation.
\item For the case of the complex $\tilde k$ with $\Im(\tilde k)>0$, the theoretical analysis in \cite{DuZhangNonlocal1} and the above theorem \eqref{theo31} shows that the solutions $u^a(x)$ and $u^e(x)$ themself decay exponentially as $|x|\to\infty$. In this situation one may make the computational region large enough and directly impose zero boundary boundary conditions on the artificial boundary layers since an evanescent wave is decaying anyway. It is a nice idea, but one need to judge whether $\tilde k$ is complex, which is nontrivial by Eq.~\eqref{eq:relation} for various kernels. The above theorem \eqref{theo31} shows that the PML modifications still hold valid, which can accelerate the process of wave decaying by making $|z_1\eta|$ large where $\Im(\tilde k)$ may be relatively small and an evanescent wave would need a large computational domain to decay sufficiently.
\item On the other hand, a large $z_1$ also brings the side-effects that the propagating waves oscillate violently in the PML, and exacerbates the numerical reflections, which can be observed in our numerical tests (Examples 1.3 and 1.4 in Section~\ref{sec:ne}).
\end{enumerate}
\end{Rem}

\section{Numerical implementation and scheme for PML equations}\label{sec:NSPE}
In this section, we focus on the numerical implementations and schemes to solve the nonlocal PML equations by truncating the computational domain and presenting the finite difference scheme to discretize the nonlocal operator. To achieve it, we make the finite horizon assumption for the kernel:
\begin{align}
\gamma(x,y)=0\quad \mbox{when}\ |x-y|_2>\delta,\ x,y\in\R^d,\ d=1,2.\label{eq:fhass}
\end{align}
Here $\delta$ stands for the horizon of nonlocal effects. For more discussions on nonlocal constraints defined on a domain with a nonzero volume we refer to \cite{TTD17,TianDu,du2012analysis,chen2011continuous}, etc.

\subsection{The selections of computational regions and boundary conditions}
Note that nonlocal Helmholtz equations with PML modifications are still defined on unbounded domains, thus we need to truncate the computational domains of interest. Once we have performed PML transformations of our wave equations, solutions are unchanged in the region of interest $\Omega$ and exponentially decaying in the exterior regions $\R^d\setminus\Omega$. This implies the computational region can be truncated at some sufficiently large $x$ by putting a hard wall, such as the zero volume constrained boundary conditions.

Denote by $\Omega_p$ PML layers, such as a square
\begin{align}
\Omega_p= \big\{ x\in \R^d\setminus\Omega\big|\ \inf_{y\in\Omega} |x-y|_\infty<d_{pml}\big\},\ \mathrm{ ( see \;Figures~\ref{fig:truncated_domaina}\ and\ \ref{fig:truncated_domainb})}, \label{eq:Omegaps}
\end{align}
and a circular region
\begin{align}
\Omega_p= \big\{ x\in \R^d\setminus\Omega\big|\ \inf_{y\in\Omega} |x-y|_2<d_{pml}\big\},\ \mathrm{(see \;Figure~\ref{fig:truncated_domainc}),} \label{eq:Omegapc}
\end{align}
where $d_{pml}$ is a positive constant by denoting the width of PML layers.

By the finite horizon assumption, we define the boundary layer
\begin{align}
\Omega_b := \big\{ x\in \R^d\setminus(\Omega\cup\Omega_p)\big|\ \mathrm{dist}(x,\Omega_p)\leq\delta\big\}.
\end{align}
The definition of the distance between $x$ and $\Omega_p$ depends on the choice of PML style. For example, we can choose $\mathrm{dist}(x,\Omega_p):=\inf_{y\in\Omega_p}|x-y|_\infty$ (Figure~\ref{fig:truncated_domainb}) for the PML~\eqref{eq:carPML} in Cartesian coordinates, and choose $\mathrm{dist}(x,\Omega_p):=\inf_{y\in\Omega_p}|x-y|_2$ (Figure~\ref{fig:truncated_domainc}) for the PML~\eqref{eq:polPML} in polar coordinates.

We get the following truncated nonlocal PML problem subject to a nonlocal constraint of the Dirichlet type
\begin{align}
	\tilde{\mathcal{L}}\hat{\tilde u}(x) - k^2\frac{\partial \tilde x}{\partial x}\hat{\tilde u}(x) = & f(x),\quad & & x\in\Omega\cup\Omega_p,  \label{eq:truPML1}  \\
	\hat{\tilde u}(x) =       & 0,  & & x\in \Omega_b.	\label{eq:truPML2}
\end{align}
\begin{figure}[htbp]
\centering
\begin{subfigure}{.32\textwidth}
 \centering
 \includegraphics[width=\textwidth]{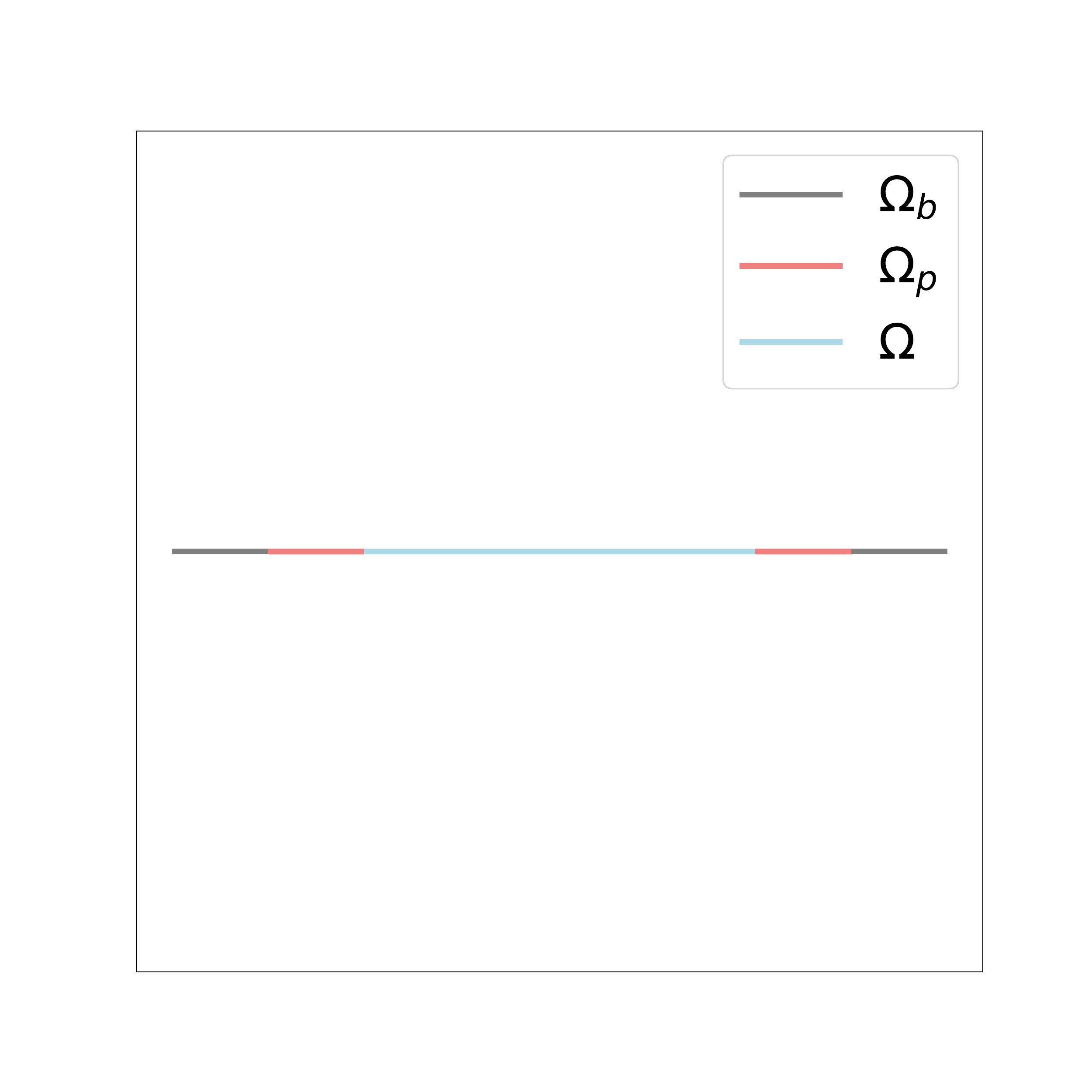} 
 \caption{1D}
 \label{fig:truncated_domaina}
\end{subfigure}
\begin{subfigure}{.32\textwidth}
 \centering
 \includegraphics[width=\textwidth]{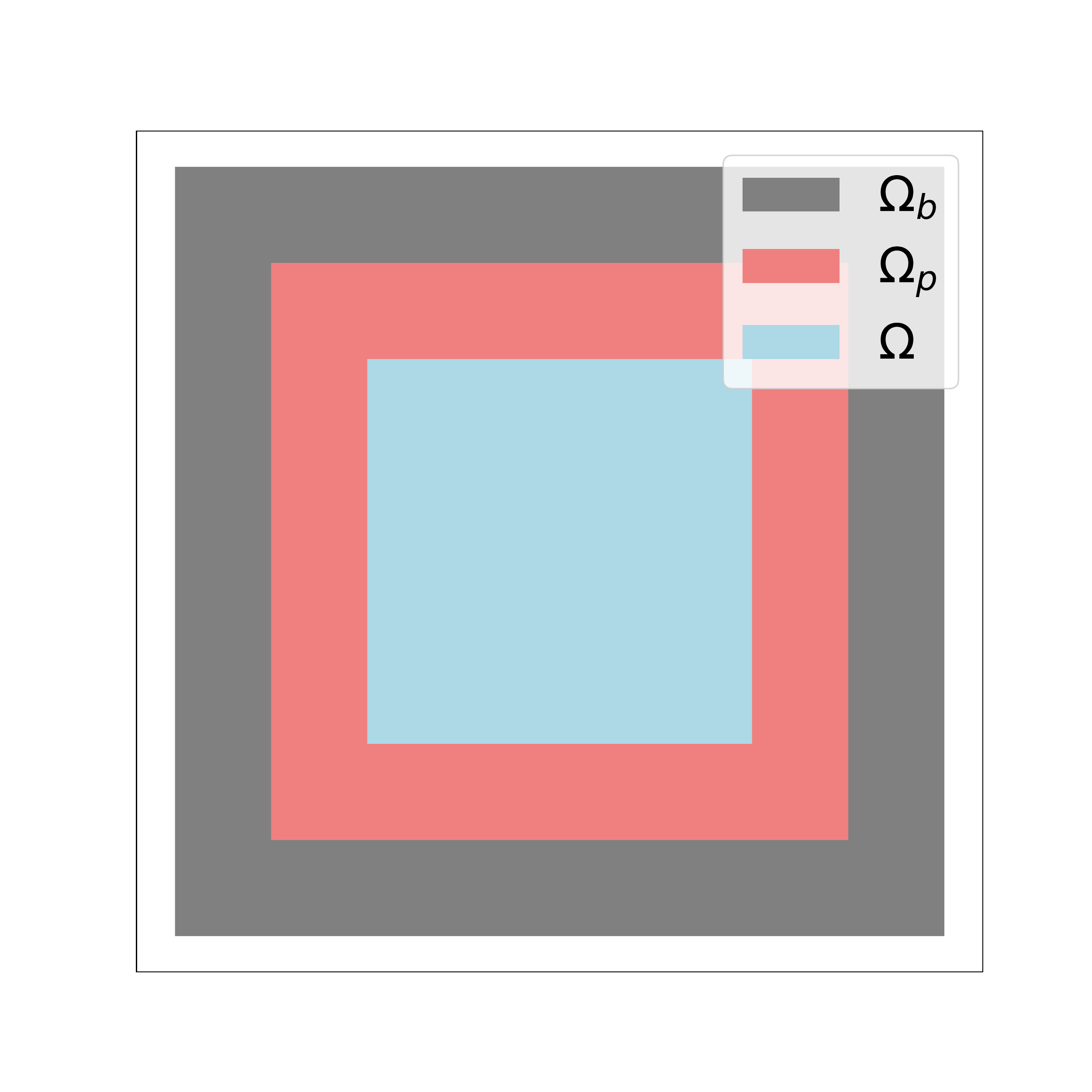} 
 \caption{2D}
 \label{fig:truncated_domainb}
\end{subfigure}
\begin{subfigure}{.32\textwidth}
 \centering
 \includegraphics[width=\textwidth]{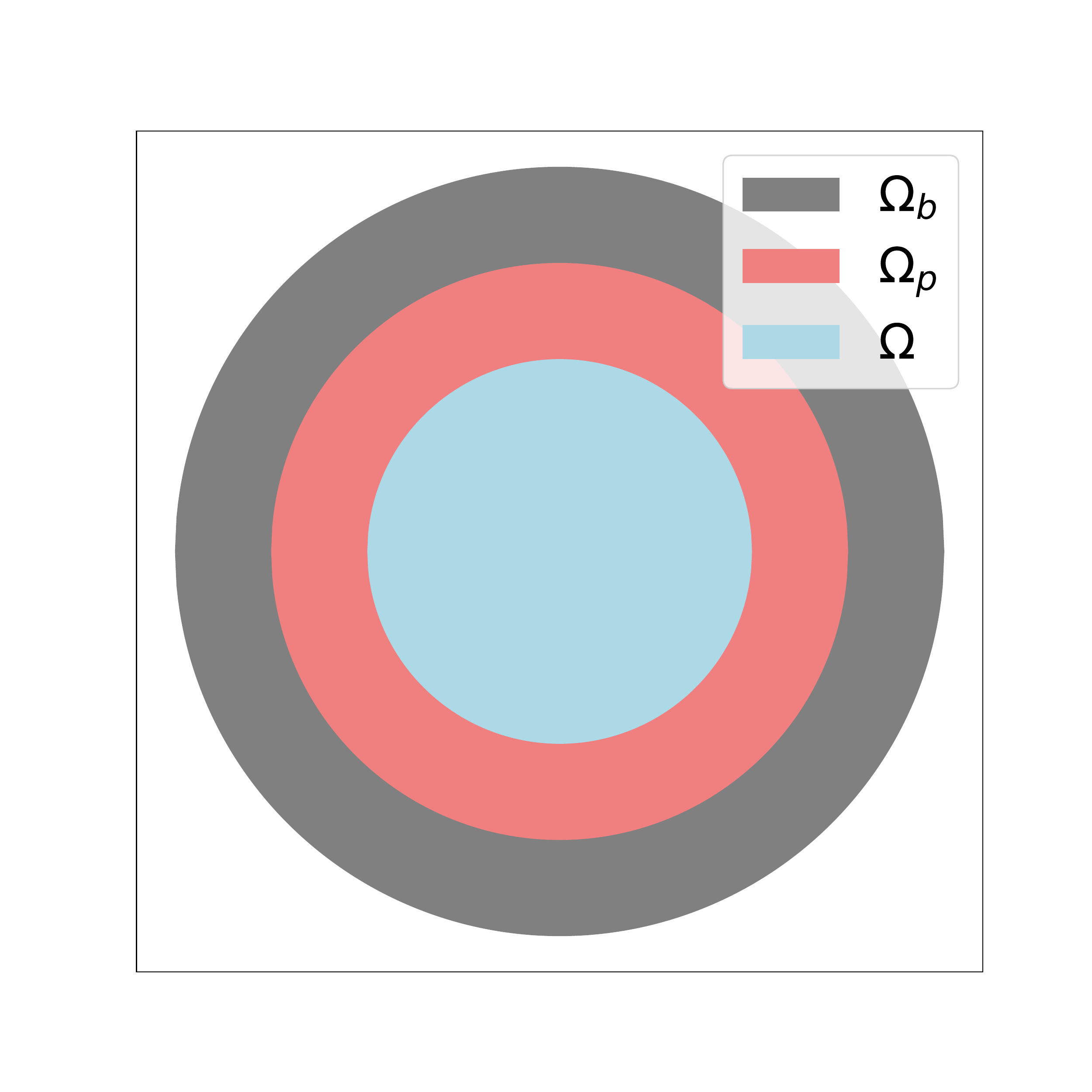} 
 \caption{2D}
 \label{fig:truncated_domainc}
\end{subfigure}
	\caption{Truncated domains in one and two dimensions. Left: 1D. Middle: $\Omega$ and $\Omega_p$ squares in 2D. Right: $\Omega$ and $\Omega_p$ is circular regions in 2D.} \label{fig:domain}
\end{figure}

\subsection{The discretization schemes} 
We now consider numerical schemes of truncated PML problem \eqref{eq:truPML1}-\eqref{eq:truPML2} by using an asymptotic compatibility scheme to discretize the nonlocal operator $\tilde{\mathcal{L}}$ to based on quadrature-based finite difference proposed in \cite{du2018nonlocal}. One can refer to more discretizations of $\tilde{\mathcal{L}}$ in \cite{du2019a,DuCbms,DUGunLeZhou,DZZ18,TTD17,TianDu,TianDu2}.

\subsubsection{The numerical scheme for one-dimensional case}
Let $[-l-d_{pml}-\delta,l+d_{pml}+\delta]$ be discretized by a uniform grid $x_{-N},x_{-N+1},\cdots,x_N $ with mesh size $h$. For simplicity, assume there exist integers $M_1$ and $M_2$ such that $l=M_1h$ and $d_{pml}=M_2h$. Denote by $\tilde \gamma(x,y) = \gamma(\tilde x,\tilde y)\alpha(x)\alpha(y)$ and
\begin{align*}
\tilde F(x,y,s) := \frac{\hat{\tilde u}(x)-\hat{\tilde u}(y)}{y-x} s \tilde \gamma\left(\frac{x+y}{2}-\frac{s}{2},\frac{x+y}{2}+\frac{s}{2}\right).
\end{align*}
It's clear that
\begin{align*}
\tilde{\mathcal{L}}\hat{\tilde u}(x) = \int_{\R} \tilde F(x,y,y-x) \mathrm{d}y.
\end{align*}
Taking $\phi_m(y)$ by the hat function of width $h$ centered at $y_m=mh$, and expanding $\tilde F(x,y,s)$ with respect to $y$ by 
\begin{align*}
\tilde F_h(x,y,s) =& \sum_m \phi_m(y) \tilde F(x,x_m,s),
\end{align*}
we then have the discretization of $\tilde{\mathcal{L}}$ as
\begin{align}
\tilde{\mathcal{L}}_h\hat{\tilde u}(x_n) =& \int_\R \tilde F_h(x_n,y,y-x_n)\dy\notag\\
=& \int_\R \sum_m \phi_m(y) \tilde F(x_n,x_m,y-x_n) \dy\notag\\
=& \sum_{m\neq n} \int_\R \phi_m(y)\frac{\hat{\tilde u}(x_n)-\hat{\tilde u}(x_m)}{x_m-x_n} (y-x_n) \tilde \gamma\left(\frac{x_n+x_m}{2}-\frac{y-x_n}{2},\frac{x_n+x_m}{2}+\frac{y-x_n}{2}\right)\dy\notag\\
=& \sum_{m\in\mathbb{Z}} \tilde a_{n,m} \hat{\tilde u}(x_m), \label{eq:diskercal}
\end{align}
where the coefficients can be calculated by 
\begin{align*}
	\tilde a_{n,m} = \left\{
	\begin{aligned}
		&- \frac{1}{(m-n)h} \int_{\R} \Big[\phi_{m}(y) (y-x_n) \tilde \gamma\Big(\frac{x_m+y}{2},x_n+\frac{x_m-y}{2}\Big) \Big]\dy, & & m\neq n, \\
		 & -\sum_{m\neq n} \tilde a_{n,m}\quad             & & m=n.
	\end{aligned}\right.
\end{align*}
By the finite horizon assumption~\eqref{eq:fhass}, in general we have $\tilde a_{n,m}=0$ when $|n-m|>\lceil \delta/h \rceil$.

Denote by $\mathcal{I}$ the index set with $n$ such that $x_n\in\Omega$, by $\mathcal{I}_p$ the index set with $n$ such that $x_n\in\Omega_p$ and by $\mathcal{I}_b$ the index set with $n$ such that $x_n\in\Omega_b$. Replacing the continuous nonlocal operator $\tilde{\mathcal{L}}$~\eqref{eq:nonlocalPMLoperator} with the above discrete counterpart, we have the following approximate discrete nonlocal Helmholtz system
\begin{align}
\sum_{m\in\mathbb{Z}} \tilde a_{n,m} \hat{\tilde u}_m - k^2 \alpha(x_n) \hat{\tilde u}_n =&  f(x_n),\quad n\in\mathcal{I}\cup\mathcal{I}_p,\\
\hat{\tilde u}_n =& 0,\qquad\quad n\in \mathcal{I}_b,
\end{align}
where $\hat{\tilde u}_n$ stands for the numerical approximation of $\hat{\tilde u}(x_n)$.

\subsubsection{The numerical scheme for two-dimensional case}
Set $\{x_{\n}\in\Omega\cup\Omega_p\cup\Omega_b\}$ to be the set of nodes of a uniform Cartesian mesh with size $h$, where $\n=(n_1,n_2)$ denotes a multi-index with $x_\n=h\n$. An illustration of the mesh grid is given in Figure \ref{fig:2Dmesh} for the two kinds of PMLs. 

Set $\tilde\gamma( x, y)= \gamma(\tilde x,\tilde y) \frac{\partial \tilde x}{\partial x}(x) \frac{\partial \tilde y}{\partial y}(y)$ and 
\begin{align*}
\tilde F(x,y,s) = \frac{\tilde u(x)-\tilde u(y)}{w(x- y)} w(s)\tilde\gamma(\frac{x+y}{2}+\frac{s}{2}, \frac{x+y}{2}-\frac{s}{2}),\quad x,y,s\in\R^2,
\end{align*}
where $w(s) = {|s|_2^2}/{|s|_1}$ represents a weight function introduced in \cite{tian2017a,du2019asymptotically}. Set $\phi_\m$ to be the piecewise bilinear basis function with $\phi_\m(x_\n)=0$ for $\m\neq\n$, and $\phi_\m(x_\m)=1$ for others. We expand $\tilde F (x, y, s)$ with respect to $y$ by 
\begin{align*}
\tilde F_h(x,y,s) = \sum_{\m} \tilde F(x,x_\m,s)\phi_\m(y). 
\end{align*}
In analogy to \eqref{eq:diskercal}, we get the approximation of $\tilde{\mathcal{L}}$ as
\begin{align*}
	\tilde{\mathcal{L}}_h\hat{\tilde u}(x_\n) = \sum_{\m\neq\n} \tilde a_{\n,\m} \hat{\tilde u}(x_\m),
\end{align*}
where
\begin{align*}
	\tilde a_{\n,\m} = \left\{
	\begin{aligned}
		&- \frac{1}{w(x_\m-x_\n)} \int_{\R^2} \Big[\phi_{\m}(y) w(y-x_\n) \tilde \gamma\Big(\frac{x_\m+y}{2},x_\n+\frac{x_\m-y}{2}\Big) \Big]\dy, & & \m\neq \n, \\
		 & -\sum_{\m\neq \n} \tilde a_{\n,\m}\quad             & & \m=\n.
	\end{aligned}\right.
\end{align*}
By the finite horizon assumption~\eqref{eq:fhass}, in general we have $\tilde a_{\n,\m}=0$ when $|\n-\m|>\lceil \delta/h \rceil$.

We still denote by $\mathcal{I}$ the index set with $\n$ such that $x_\n\in\Omega$, by $\mathcal{I}_p$ the index set with $\n$ such that $x_\n\in\Omega_p$ and by $\mathcal{I}_b$ the index set with $\n$ such that $x_\n\in\Omega_b$ (see Figure~\ref{fig:2Dmesh}). Replacing the continuous nonlocal operator $\tilde{\mathcal{L}}$ in \eqref{eq:carPML} or \eqref{eq:polPML} with the above discrete counterpart, we have the following approximate system
\begin{align}
\sum_{\m\in\mathbb{Z}^2} \tilde a_{\n,\m} \hat{\tilde u}_\m - k^2 \alpha(x_\n) \hat{\tilde u}_\n =&  f(x_\n),\quad \n\in\mathcal{I}\cup\mathcal{I}_p,\\
\hat{\tilde u}_\n =& 0,\qquad\quad \n\in \mathcal{I}_b,
\end{align}
where $\hat{\tilde u}_\n$ stands for the numerical approximation of $\hat{\tilde u}(x_\n)$.


\section{Numerical Examples} \label{sec:ne}
We now provide numerical examples to illustrate the effectiveness of PML techniques. Subsection~\ref{subsec:ne_1D} focuses on one dimensional case, and presents the ``numerical reflections'' generated by large PML coefficient $z$. Subsection~\ref{subsec:ne_2D} focuses on two dimensional case. 

In the following examples, we consider the errors between $u$ and $\hat{\tilde u}_h$ in $L^2$ norm over $\Omega$, i.e., $\|u-\hat{\tilde u}_h\|_{L^2(\Omega)}$, where $u$ represents the ``exact'' solution and $\hat{\tilde u}_h$ represents the numerical solution. For 1D, the exact solution for the radical kernel $ \gamma(x,y) = \gamma_r(x-y)$ \eqref{k1} is obtained by the explicit formula \eqref{eq:GF}, which is taken as a benchmark solution to study ``numerical reflections'' for different PML coefficients $z$. The exact solutions for other kernels are obtained by the quadrature-based finite difference discretization with a sufficiently fine mesh over a domain large enough, and with sufficiently large $d_{pml}$.

\subsection{Numerical examples in one dimension}\label{subsec:ne_1D}
In the following examples, the absorption coefficient is chosen as
\begin{align*}
\sigma(t) = \begin{cases}
0, & |t|<l,\\
\frac{1}{d_{pml}}(|t|-l), & l\leq|t|.
\end{cases}
\end{align*}

\textbf{Example 1.1.} Here we set $l=d_{pml}=1$ and consider the typical radical kernel $\gamma(x,y) = \gamma_1(|x-y|)$ given in \eqref{k1}. The source is chosen as a Gaussian function 
\begin{equation}\label{GSF}
f(x) = \frac{k}{\sqrt{\pi}} e^{-k^2x^2}. 
\end{equation}

Figure~\ref{fig:ex1solutions1} plots the solutions for $k=2\pi$ and different $c_\gamma$. The numerical solutions are obtained with mesh size $h=2^{-10}$ and PML coefficient $z=40\i$. The exact solution is obtained by using \eqref{eq:GF}. Figure~\ref{fig:ex1solutions1a} shows that the numerical solution for $k=2\pi$ and $c_\gamma=0.9/k$ decays exponentially in PML layers. For $k=2\pi$ and $c_\gamma=1.1/k$, since the exact solution $u$ itself decays outside the support of $f(x)$, one can observe in Figure~\ref{fig:ex1solutions1b} that the real part of $\hat{\tilde u}_h$ is the same as that of $u$, but the PML makes the imaginary part of the solution oscillate with an amplitude of 1e-7 in $\Omega$.

\begin{figure}[htbp]
\centering
\begin{subfigure}{.48\textwidth}
 \centering
 \includegraphics[width=0.49\textwidth]{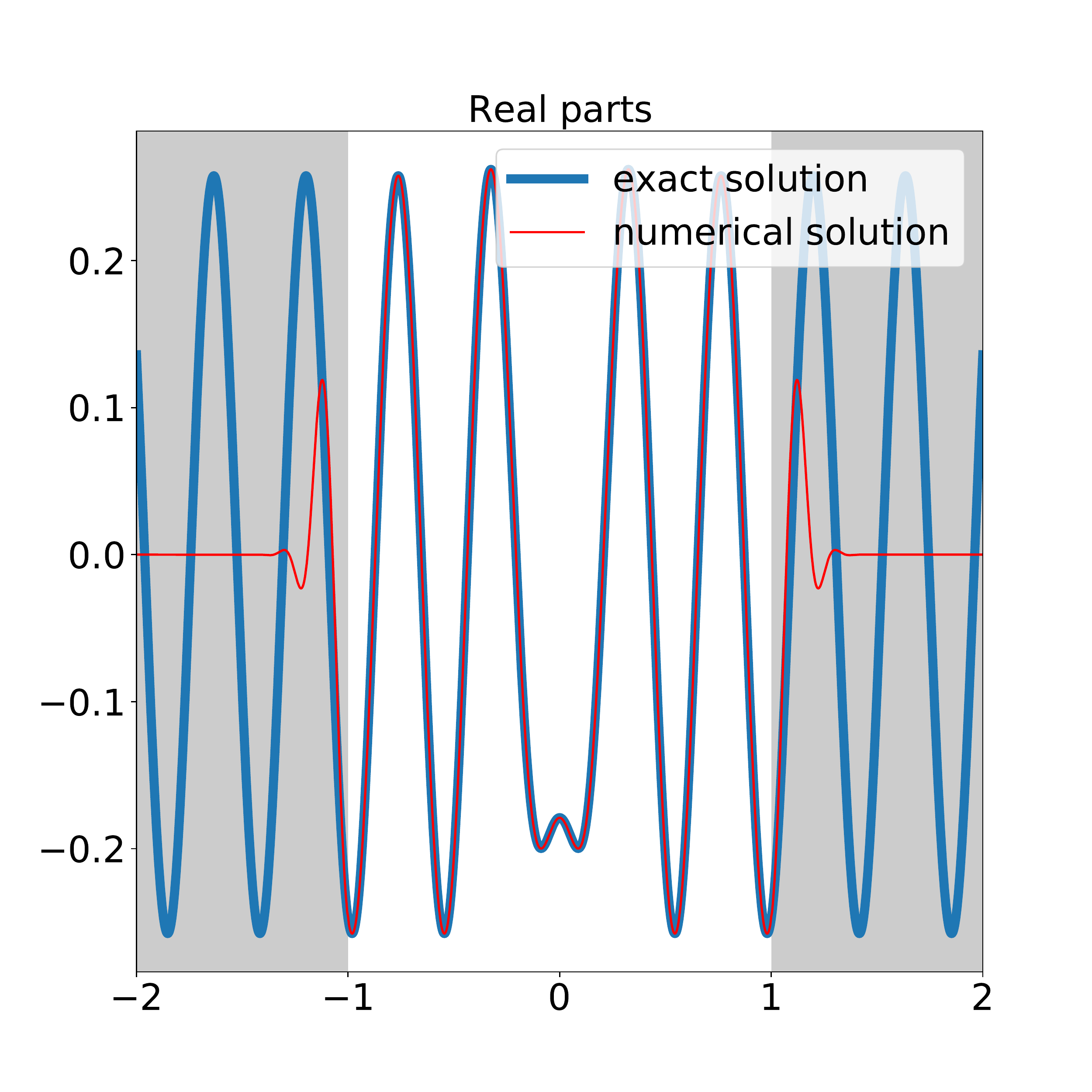}
 \includegraphics[width=0.49\textwidth]{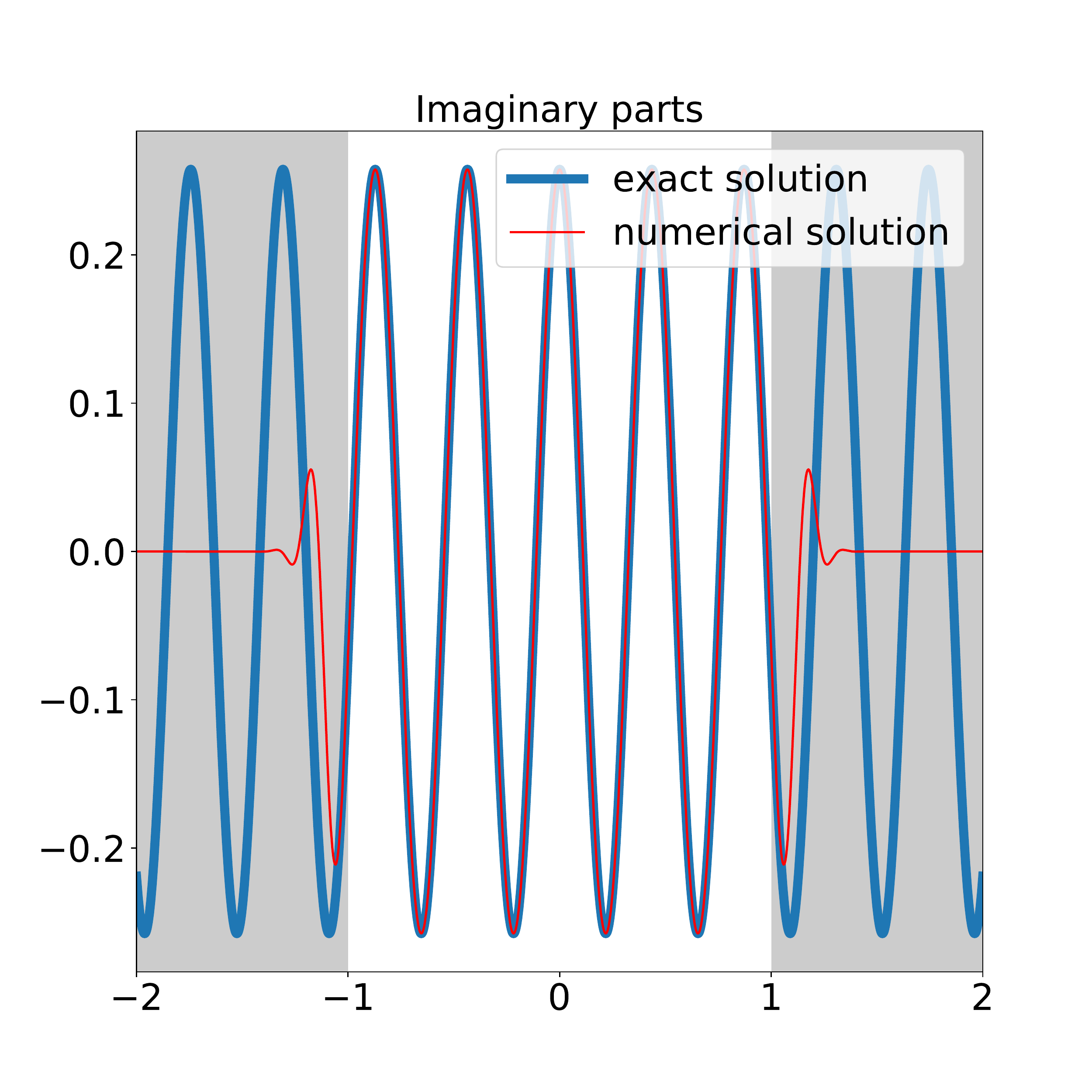}
 \caption{$k=2\pi,c_\gamma=0.9/k$}
 \label{fig:ex1solutions1a}
\end{subfigure}
\begin{subfigure}{.48\textwidth}
 \centering
 \includegraphics[width=0.49\textwidth]{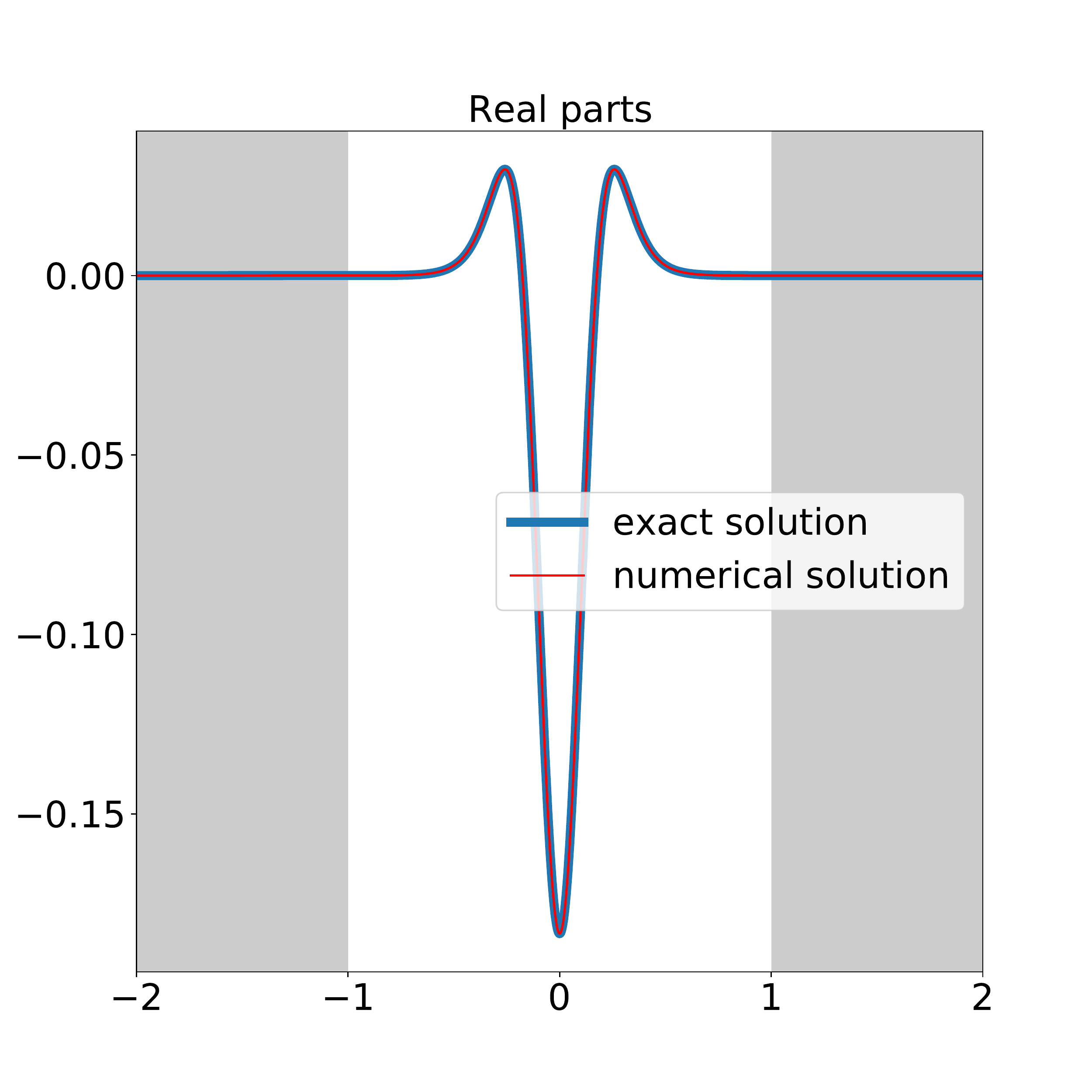}
 \includegraphics[width=0.49\textwidth]{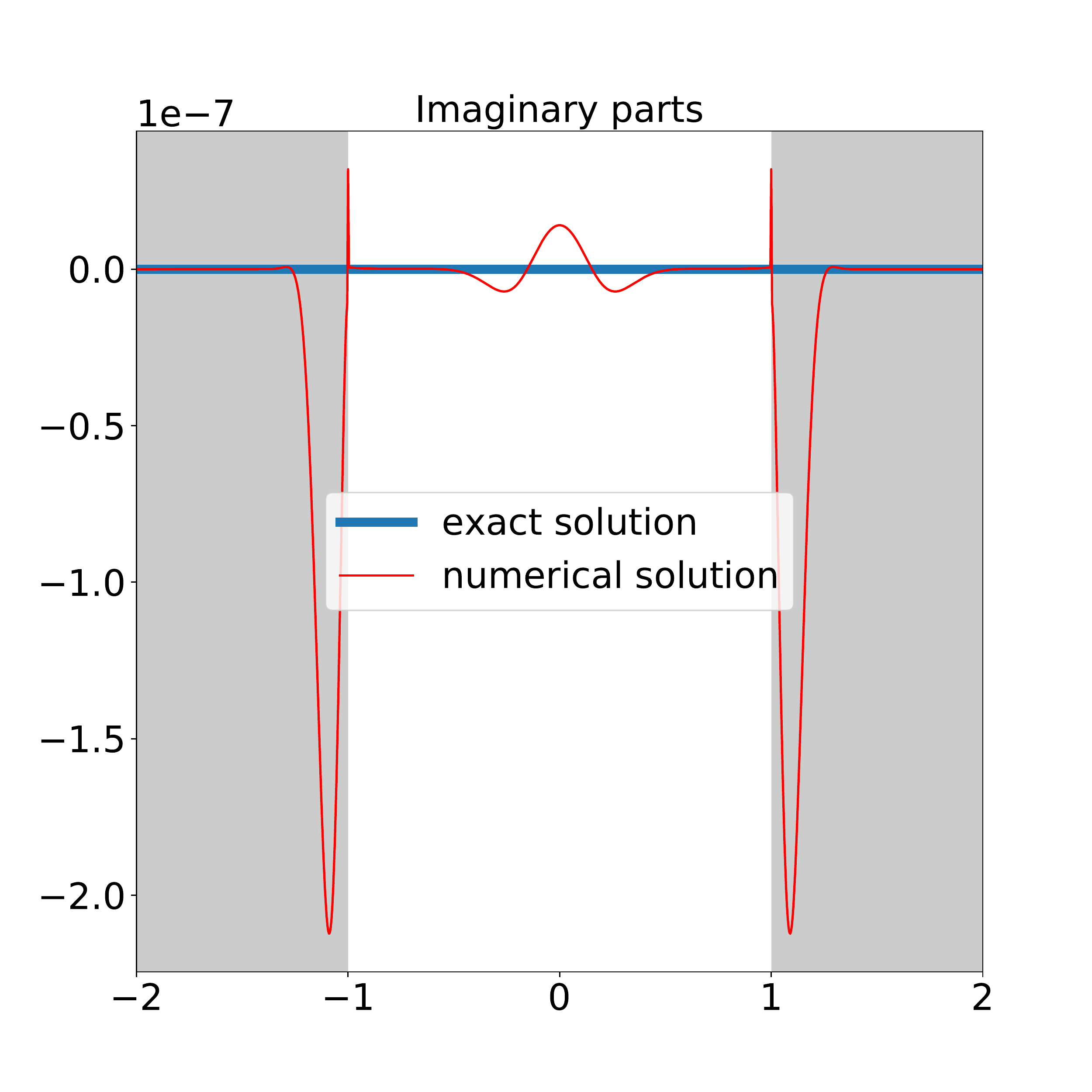}
 \caption{$k=2\pi,c_\gamma=1.1/k$}
 \label{fig:ex1solutions1b}
\end{subfigure}
	\caption{(Example 1.1) the exact solution $u$ without any modification and numerical solution $\hat{\tilde u}_h$ for $k=2\pi$, $c_\gamma=0.9/k$ and $k=2\pi$, $c_\gamma=1.1/k$, respectively. $\hat{\tilde u}_h$ is obtained on the mesh with $h=2^{-10}$. The PML region for numerical solutions are shaded in light grey.} \label{fig:ex1solutions1}
\end{figure}

 To investigate the influence of $z$ on numerical PML solutions for $k=2\pi$, we present the errors by taking various $z=10\i,20\i,40\i,80\i,160\i,10(1+\i),20(1+\i),40(1+\i),80(1+\i),40(2+\i),40(4+\i)$ and mesh sizes $h=2^{-n},(n=4,5,\cdots,8)$. Bar graphs in Figure~\ref{fig:ex1errors} show the errors for $c_\gamma=0.9/k$, where the exact solution is plotted in Figure~\ref{fig:ex1solutions1a}. It can be seen that large $z$ causes ``numerical reflections''. For $c_\gamma=1.1/k$, we also compute the errors for the above $z$ and $z=0$ since the exact solution itself decays exponentially (Theorem~\cite[Theorem 3]{DuZhangNonlocal1} and Figure~\ref{fig:ex1solutions1b}). Different from those for $c_\gamma=0.9/k$, large $z$ causes few ``numerical reflections'' for $c_\gamma=1.1/k$. These results suggest that one can use PML whatever $\tilde k$ is real or complex in practical implementations. This is, one can use the PML no matter the solution itself decays exponentially or itself does not decay exponentially.
 
\begin{figure}[htbp]
\centering
\begin{subfigure}{.19\textwidth}
 \centering
 \includegraphics[width=\textwidth]{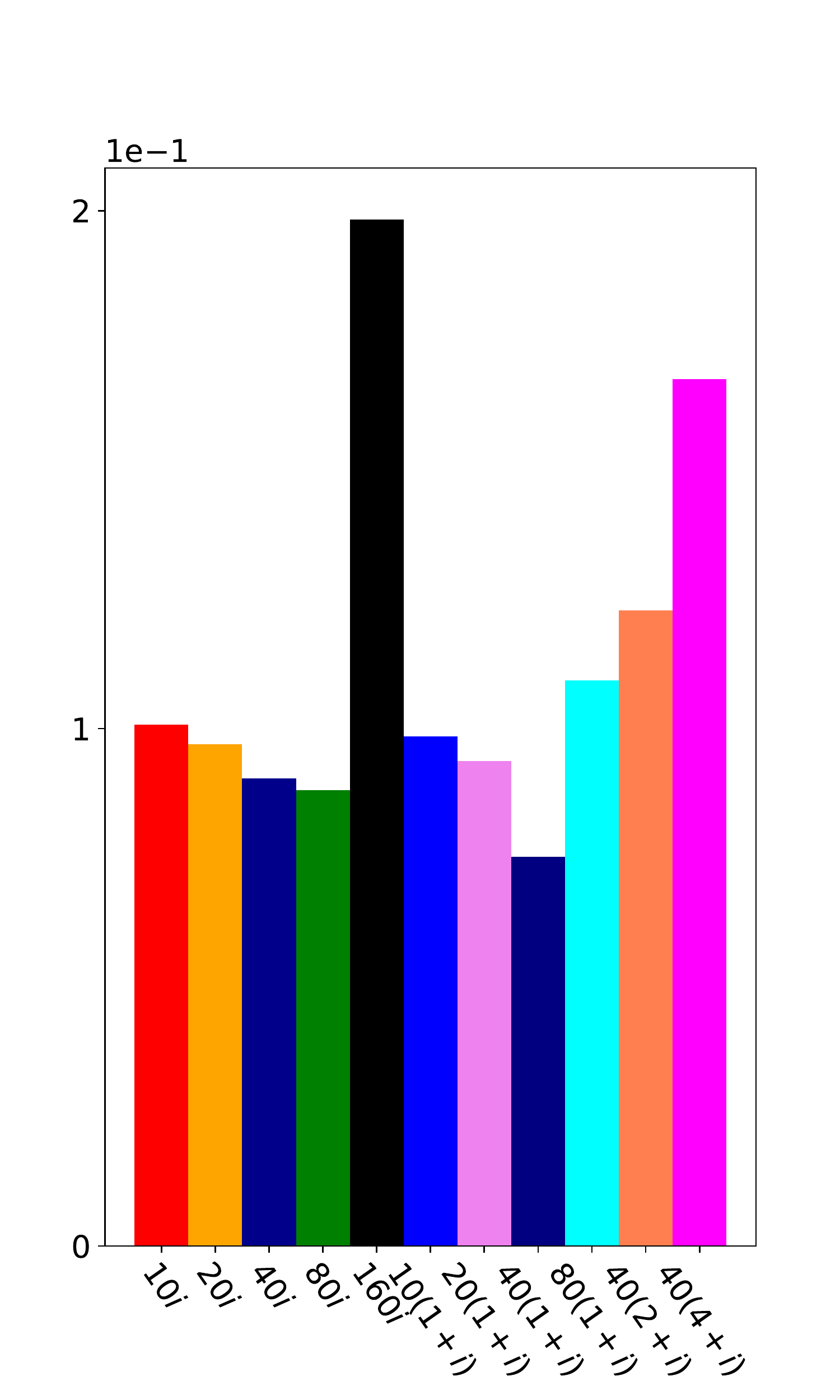} 
 \caption{$h=2^{-4}$}
 \label{fig:errors_sigpmlni_1}
\end{subfigure}
\begin{subfigure}{.19\textwidth}
 \centering
 \includegraphics[width=\textwidth]{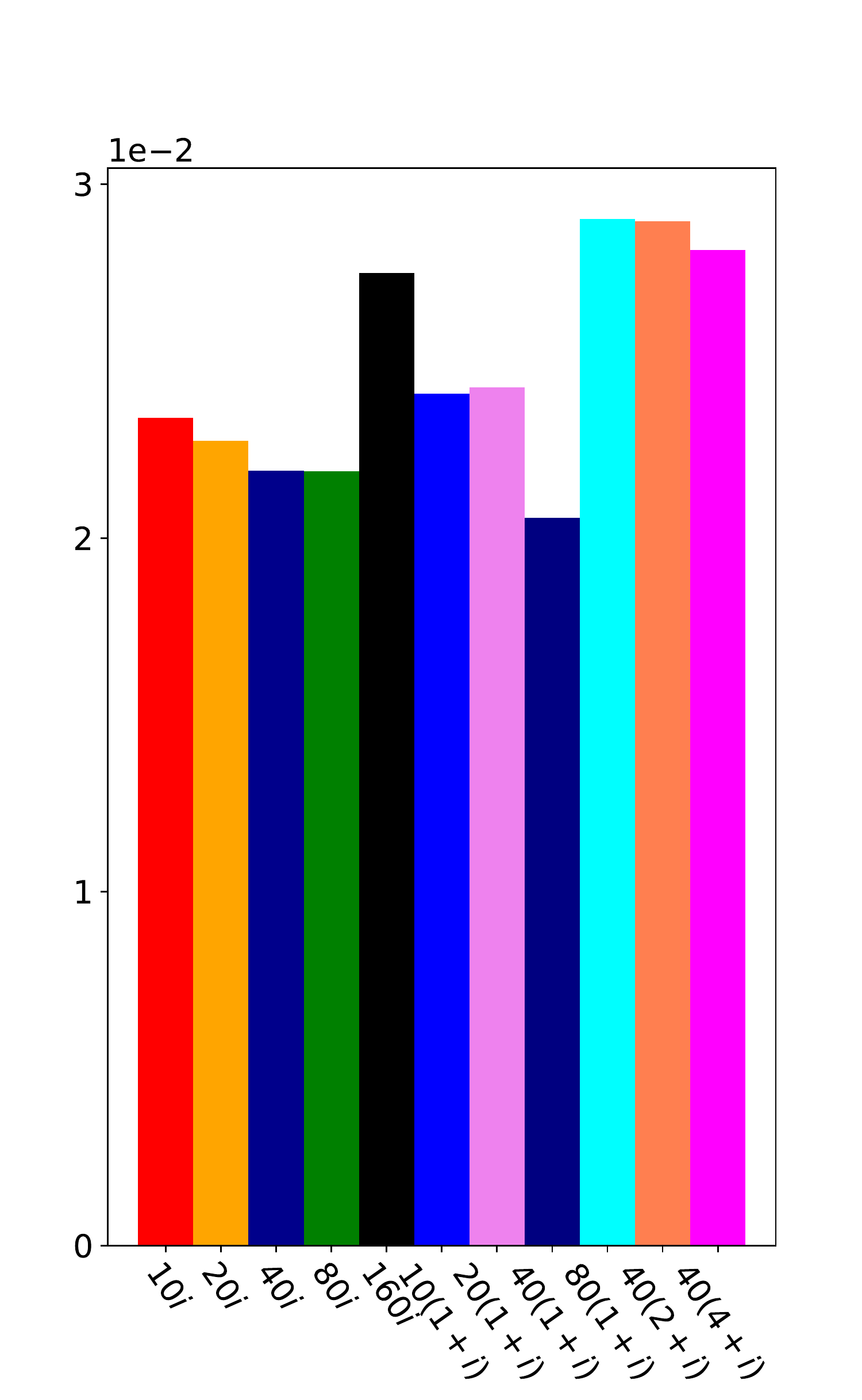} 
 \caption{$h=2^{-5}$}
 \label{fig:errors_sigpmlni_2}
\end{subfigure}
\begin{subfigure}{.19\textwidth}
 \centering
 \includegraphics[width=\textwidth]{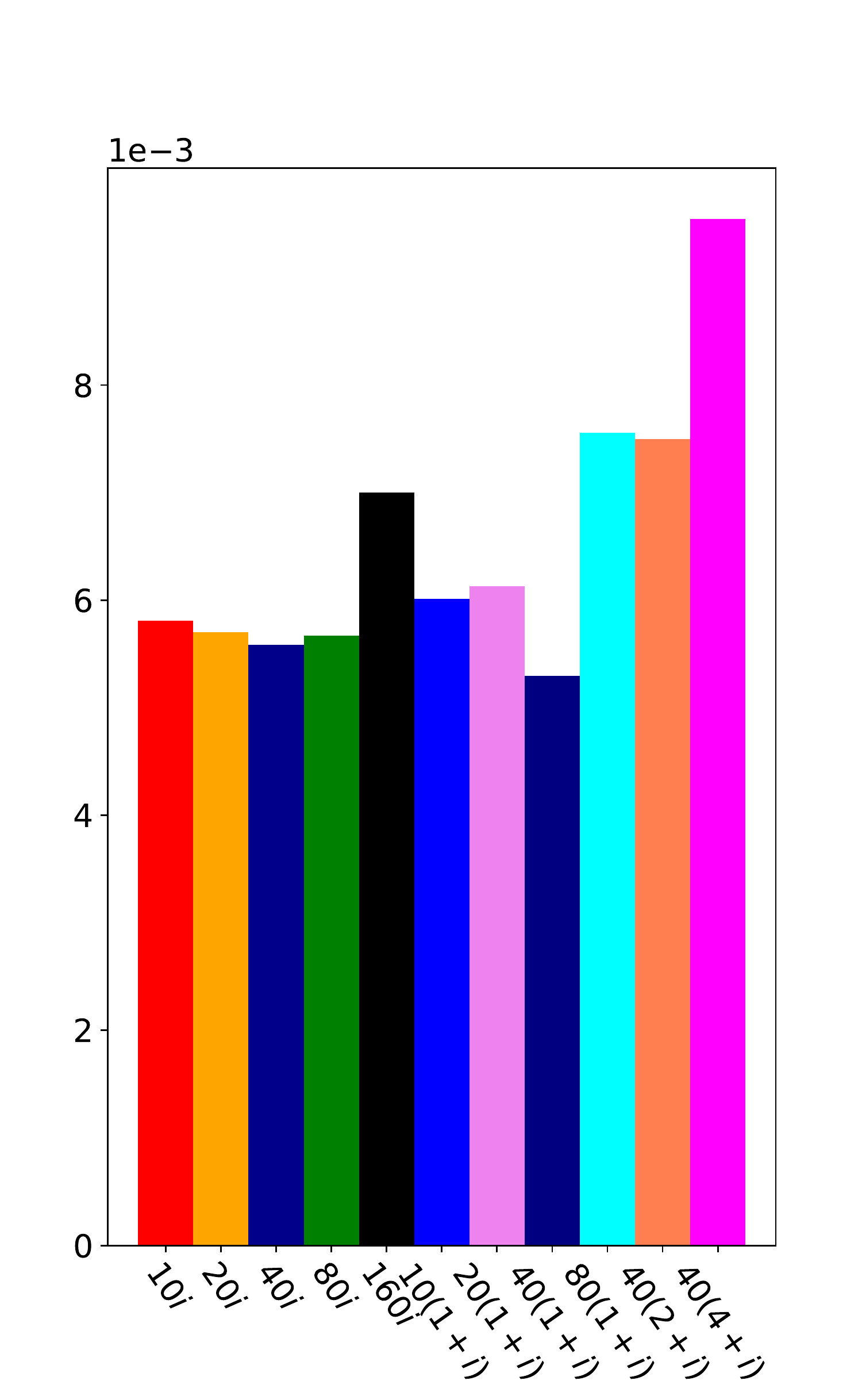} 
 \caption{$h=2^{-6}$}
 \label{fig:errors_sigpmlni_3}
\end{subfigure}
\begin{subfigure}{.19\textwidth}
 \centering
 \includegraphics[width=\textwidth]{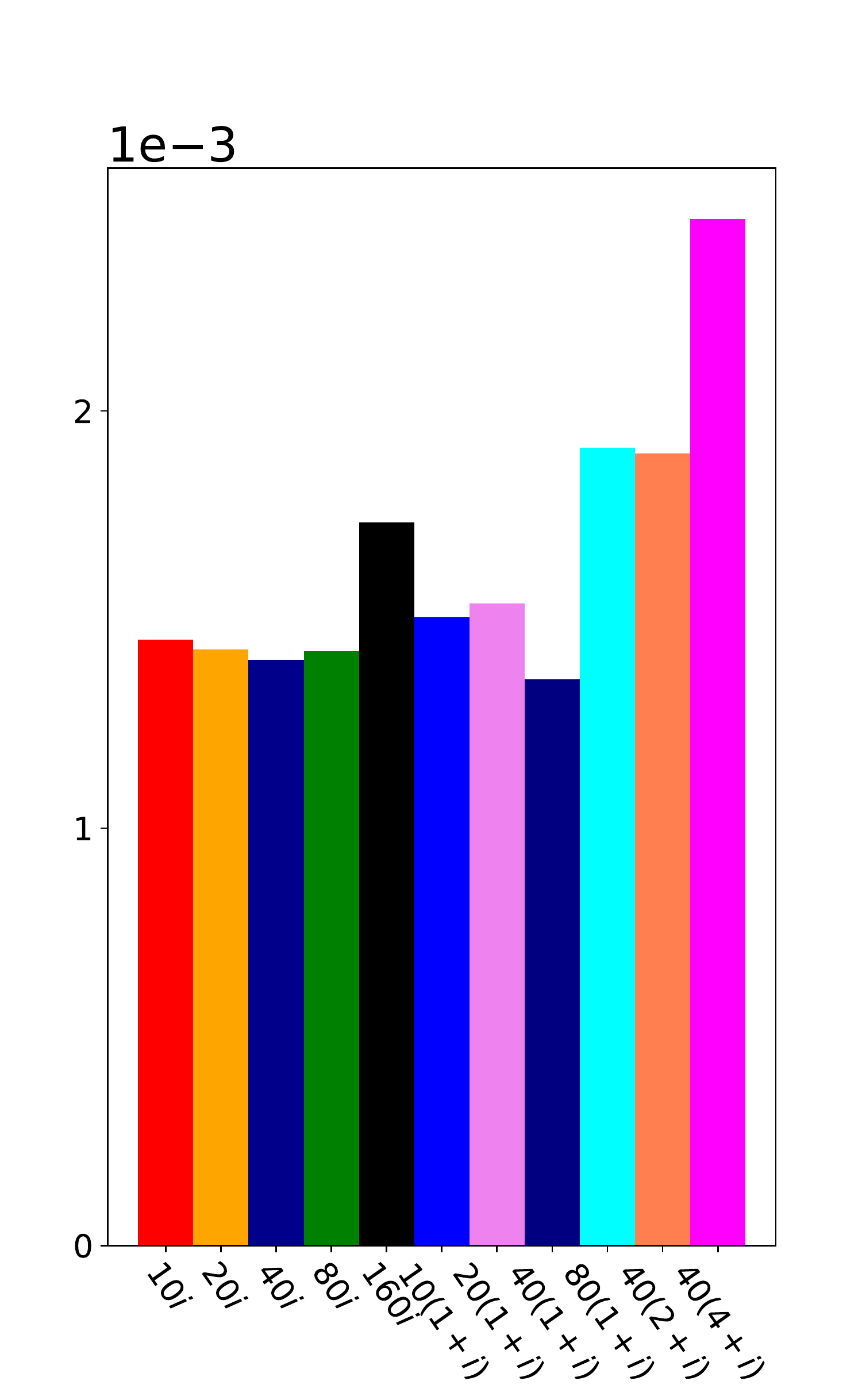} 
 \caption{$h=2^{-7}$}
 \label{fig:errors_sigpmlni_4}
\end{subfigure}
\begin{subfigure}{.19\textwidth}
 \centering
 \includegraphics[width=\textwidth]{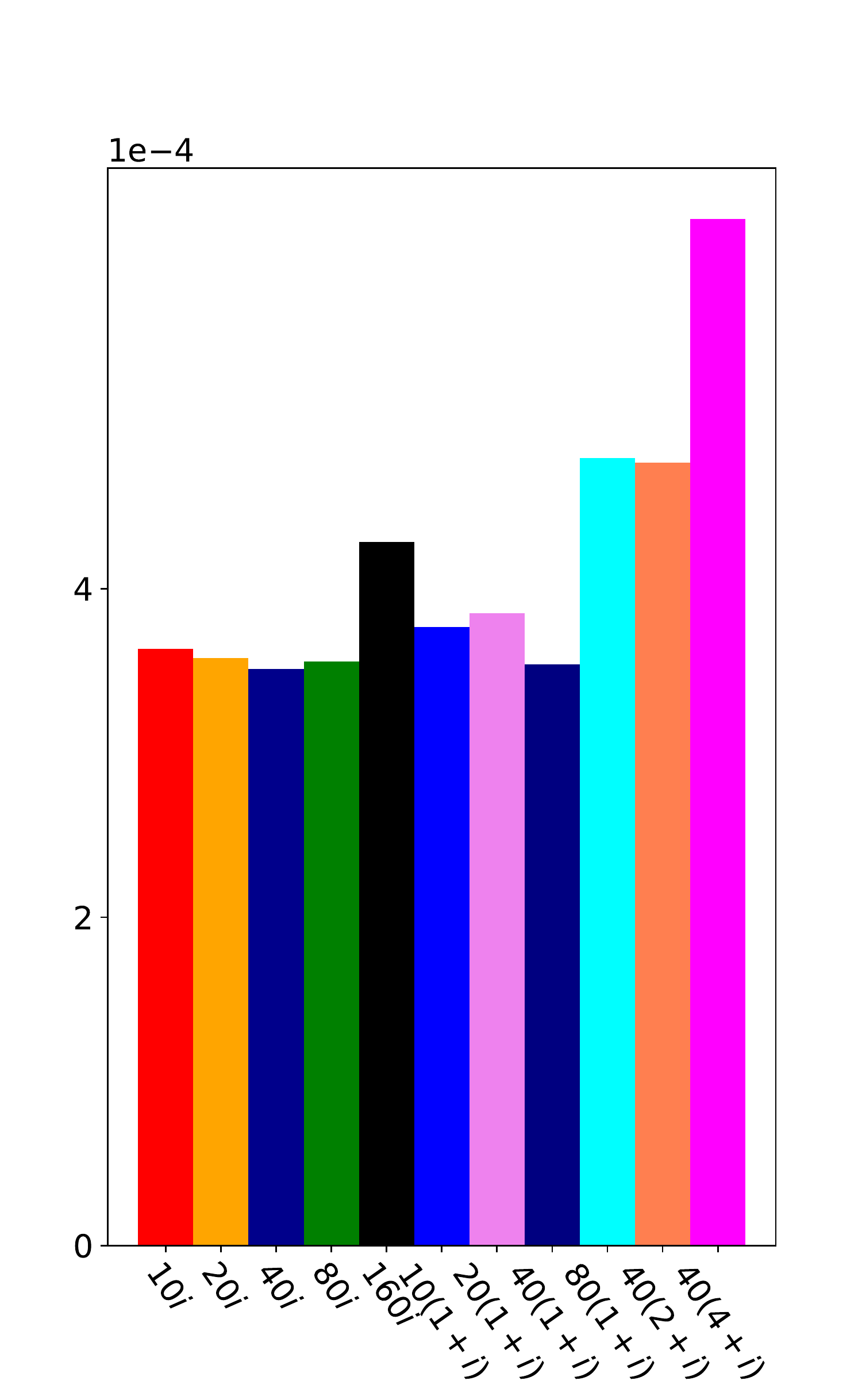} 
 \caption{$h=2^{-8}$}
 \label{fig:errors_sigpmlni_5}
\end{subfigure}
	\caption{(Example 1.1) errors for $k=2\pi$, $c_\gamma=0.9/k$ and different PML coefficients $z=10\i,20\i,40\i,80\i,160\i,10(1+\i),20(1+\i),40(1+\i),80(1+\i),40(2+\i),40(4+\i)$.} \label{fig:ex1errors}
\end{figure}

\begin{figure}[htbp]
\centering
\begin{subfigure}{.19\textwidth}
 \centering
 \includegraphics[width=\textwidth]{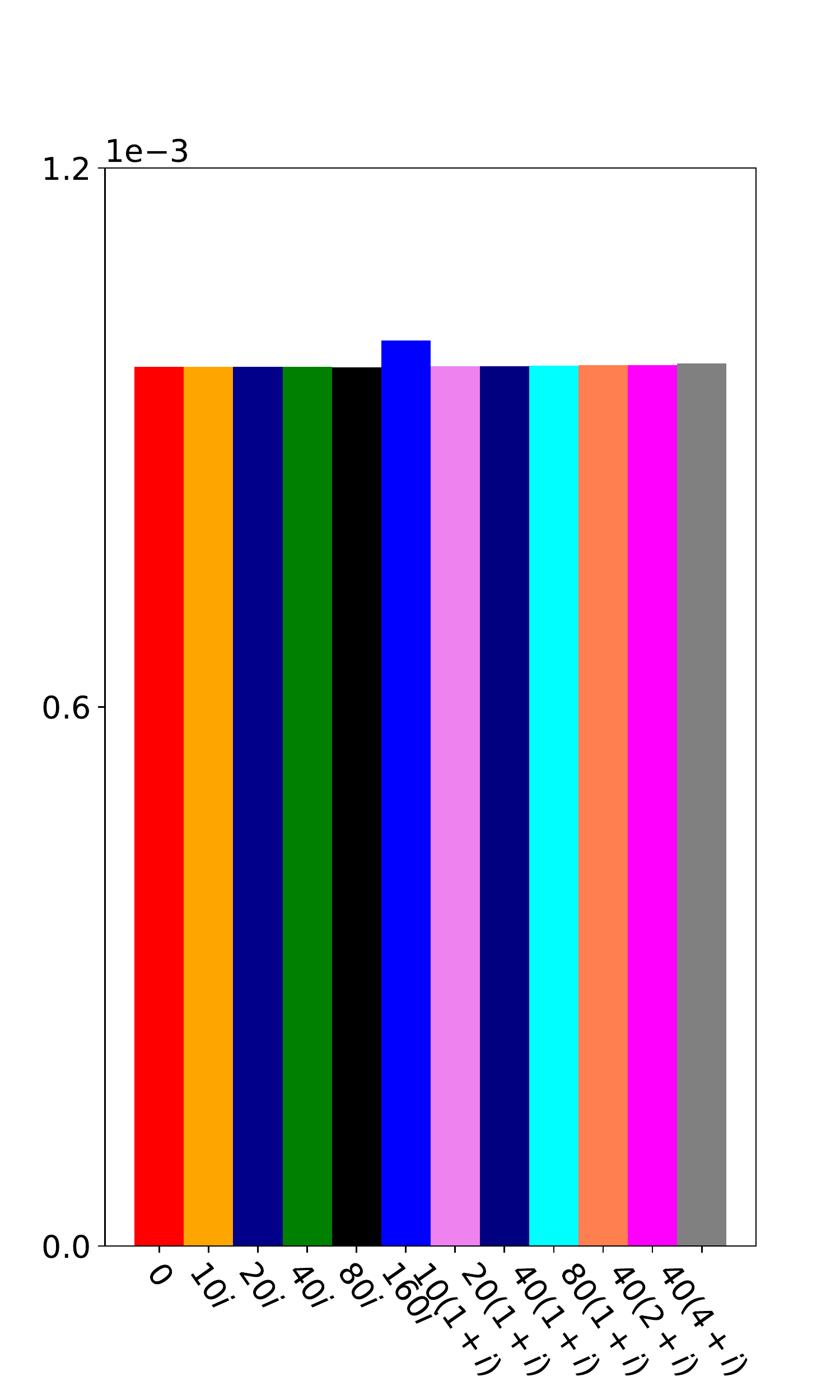} 
 \caption{$h=2^{-4}$}
 \label{fig:errors_sigpmlni_1}
\end{subfigure}
\begin{subfigure}{.19\textwidth}
 \centering
 \includegraphics[width=\textwidth]{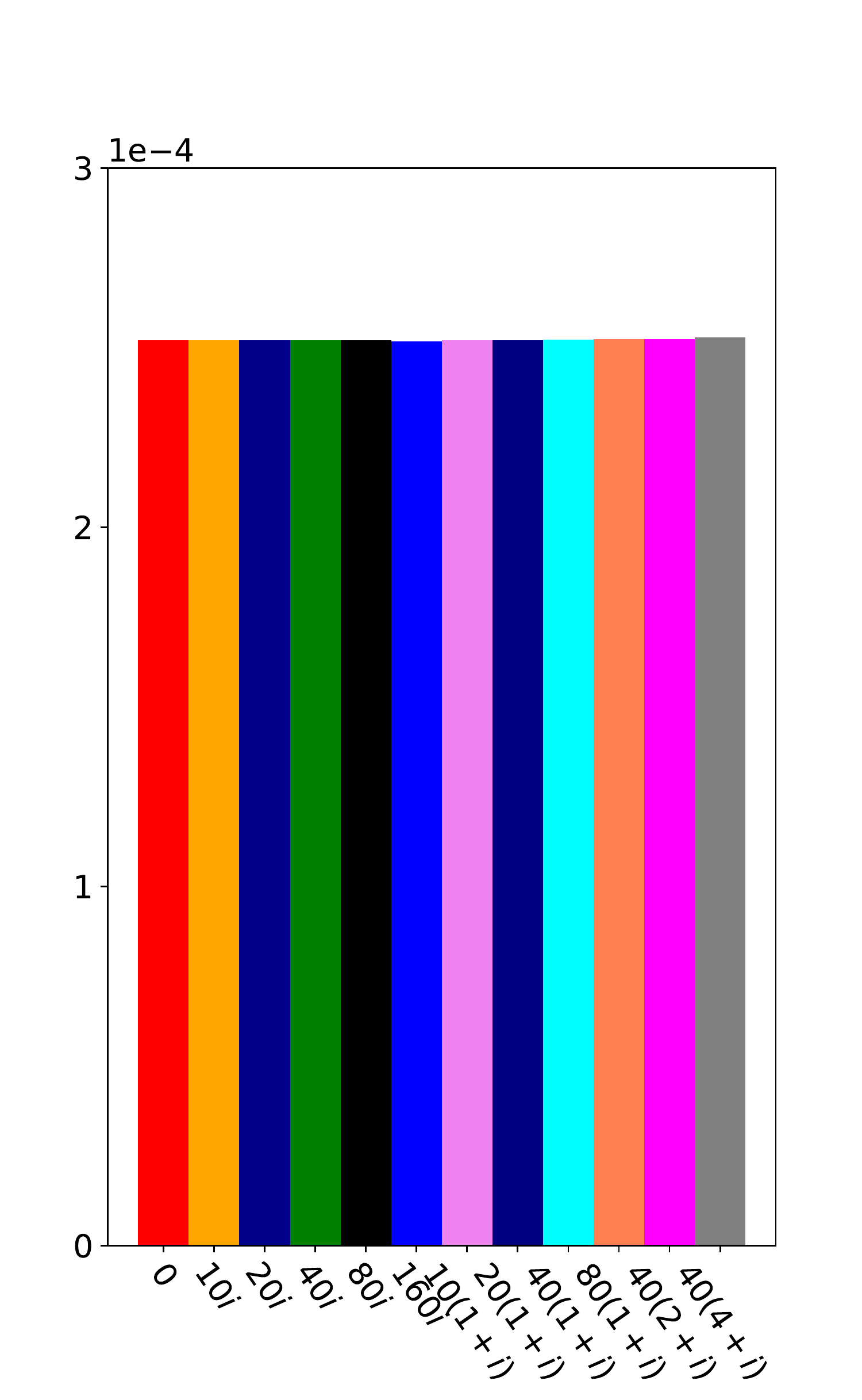} 
 \caption{$h=2^{-5}$}
 \label{fig:errors_sigpmlni_2}
\end{subfigure}
\begin{subfigure}{.19\textwidth}
 \centering
 \includegraphics[width=\textwidth]{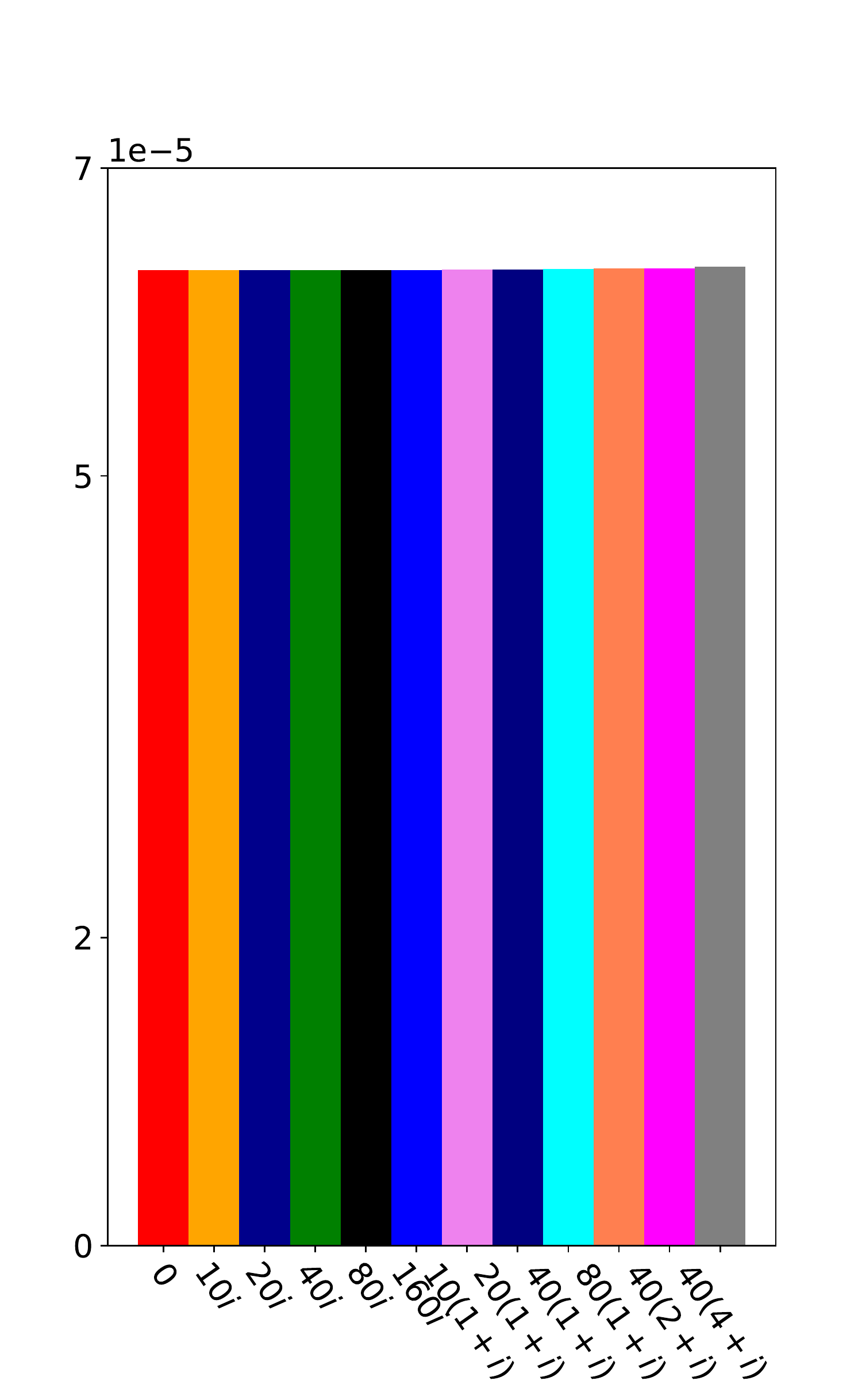} 
 \caption{$h=2^{-6}$}
 \label{fig:errors_sigpmlni_3}
\end{subfigure}
\begin{subfigure}{.19\textwidth}
 \centering
 \includegraphics[width=\textwidth]{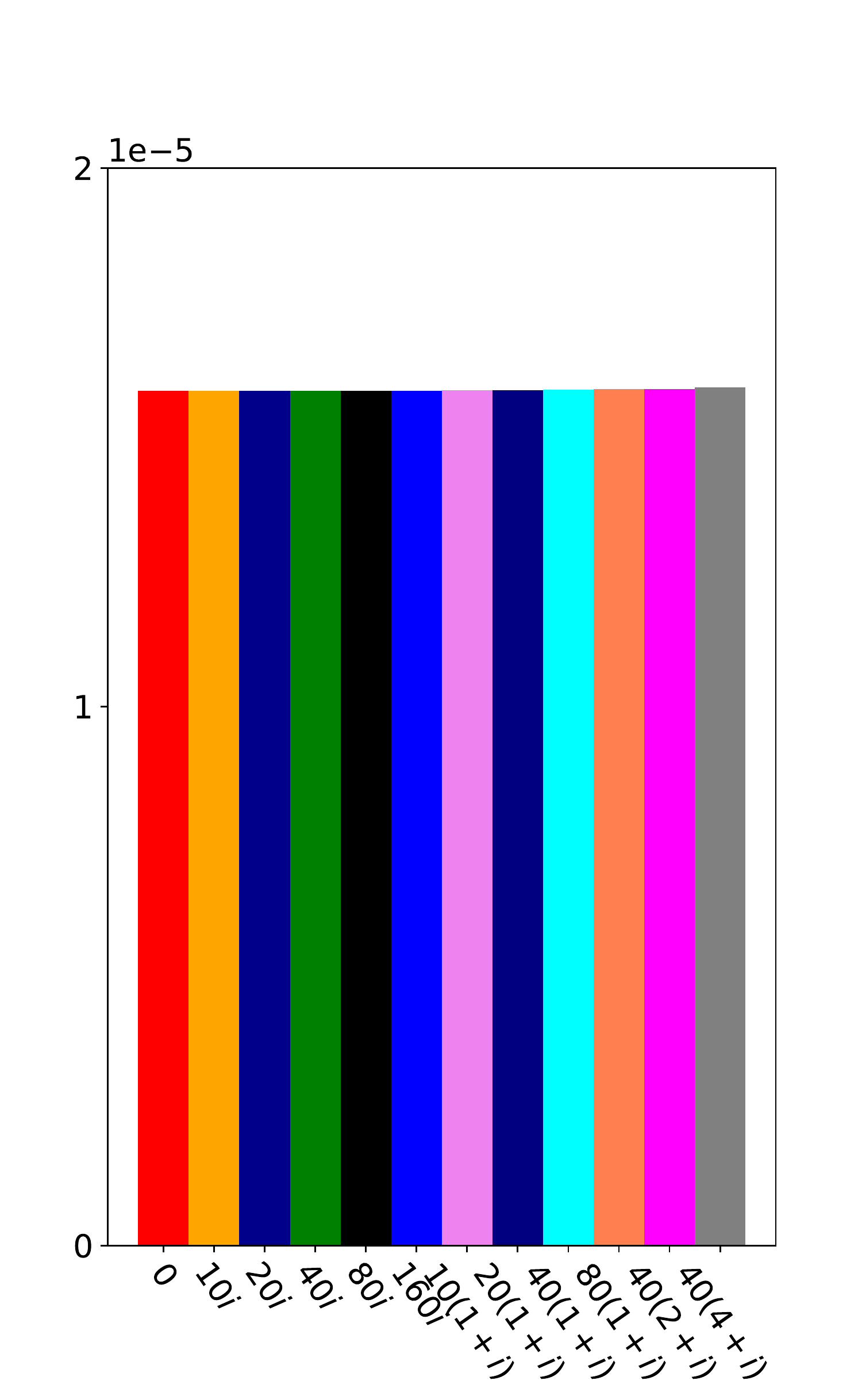} 
 \caption{$h=2^{-7}$}
 \label{fig:errors_sigpmlni_4}
\end{subfigure}
\begin{subfigure}{.19\textwidth}
 \centering
 \includegraphics[width=\textwidth]{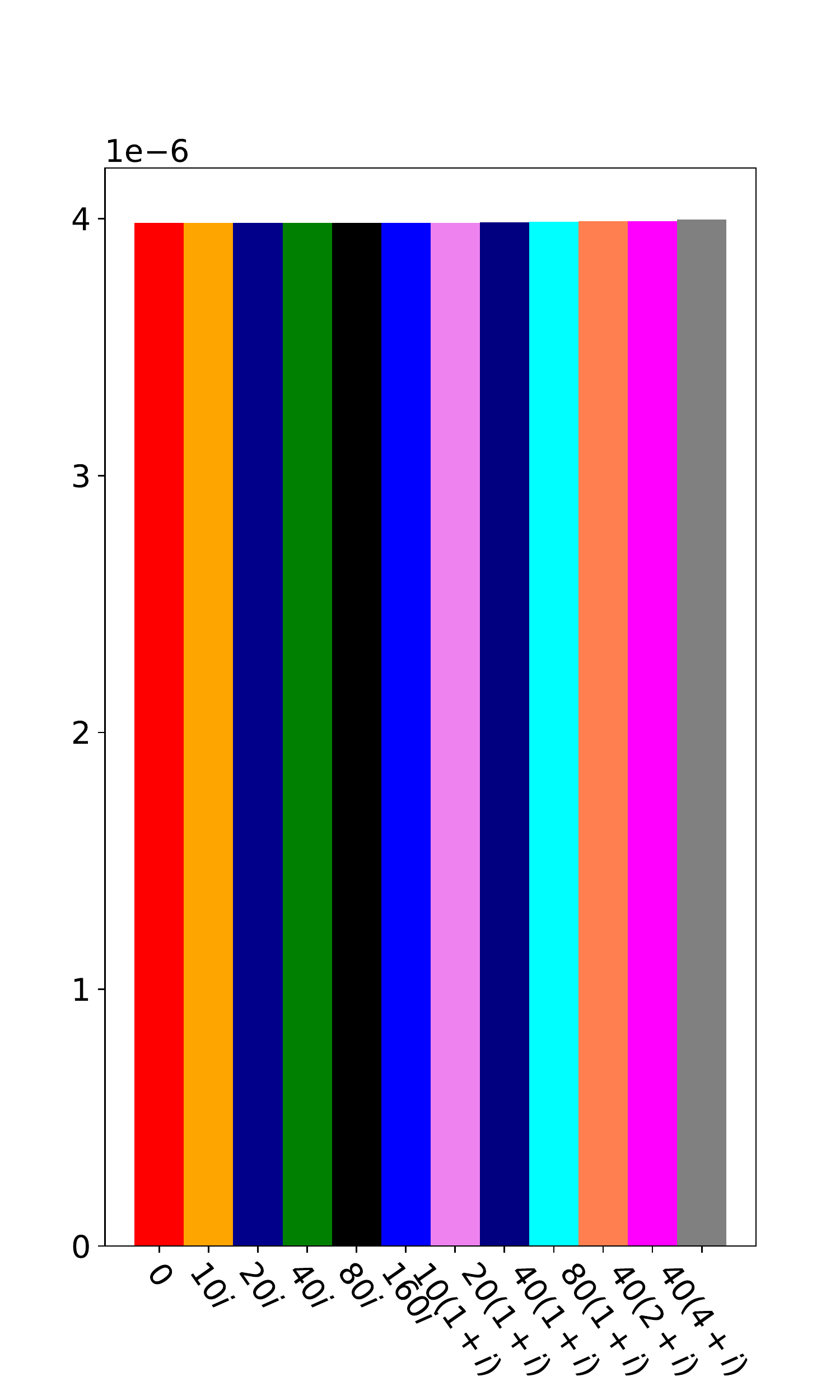} 
 \caption{$h=2^{-8}$}
 \label{fig:errors_sigpmlni_5}
\end{subfigure}
	\caption{(Example 1.1) errors for $k=2\pi$, $c_\gamma=1.1/k$ and different PML coefficients $z=0,10\i,20\i,40\i,80\i,160\i,10(1+\i),20(1+\i),40(1+\i),80(1+\i),40(2+\i),40(4+\i)$.}
\end{figure}

Table~\ref{tab:ex1p1errors} shows the errors for different $k$ and $\delta$ and the second-order convergence rate for the numerical scheme, where the PML coefficient is taken as $z=40\i$.

\begin{table}[htbp]
\centering
\begin{tabular}{|c|c|c|c|c||c|c|c|c|}
\hline
 & \multicolumn{4}{c||}{$c_\gamma=0.9/k$}  & \multicolumn{4}{c|}{$c_\gamma=1.1/k$}\\
 \hline
\diagbox{$h$}{$k$} & $2\pi$ & Order & $4\pi$ & Order & $2\pi$ & Order & $4\pi$ & Order\\
\hline
$2^{-5}$ & 2.05e-02 & -- &8.14e-02 & -- & 2.52e-04 & -- &3.43e-04 & -- \\
\hline
$2^{-6}$ &  5.29e-03 &1.96 &2.02e-02 & 2.01 & 6.33e-05 &1.99 &8.87e-05 & 1.99\\
\hline
$2^{-7}$ & 1.35e-03 &1.97 &4.92e-03 & 2.04 & 1.58e-05 &2.00& 2.23e-05 & 2.00\\
\hline
$2^{-8}$ & 3.54e-04 &1.94 &1.22e-03 & 2.01 & 3.96e-06 &2.00& 5.60e-06& 2.00\\
\hline
\end{tabular}
\caption{(Example 1.1) $L^2$-errors of numerical solutions with the PML coefficient $z=40(1+\i)$.} \label{tab:ex1p1errors}
\end{table}

\textbf{Example 1.2.} We now consider a fractional Helmholtz equation
\begin{align*}
(-\Delta)^s u(x) - k^2u(x) =& f(x),\quad (0<s<1),
\end{align*}
where
\begin{align*}
(-\Delta)^s =& C(s)\ \mathrm{p.v.}\int_\R \frac{u(x)-u(y)}{|x-y|^{1+2s}}\dy,\quad C(s)=\frac{2^{2s}s\Gamma(s+\frac{1}{2})}{\pi^\frac12\Gamma(1-s)}.
\end{align*}
We use this example to investigate the performance of the PML for the nonintegrable kernel.

The source function is taken as the Gaussian function \eqref{GSF}, and the parameters are taken by $l=d_{pml}=10$, $h=2^{-10}$ and $z = 10(15+\i)$ in the simulations. Figure~\ref{fig:ex3solutions} plots the reference solutions and numerical solutions for $s=1/2$ and $k=2\pi/10,32\pi/10$, respectively. It can be observed that the real part to $z$ makes the wave oscillate faster in the PML region, which possibly comes at a price by exacerbating the numerical reflections. Table~\ref{tab:ex1p3errors} lists the errors and convergence orders for $s=1/4,1/2$ and different $k$.

\begin{figure}[htbp]
\centering
\begin{subfigure}{.48\textwidth}
 \centering
 \includegraphics[width=0.49\textwidth]{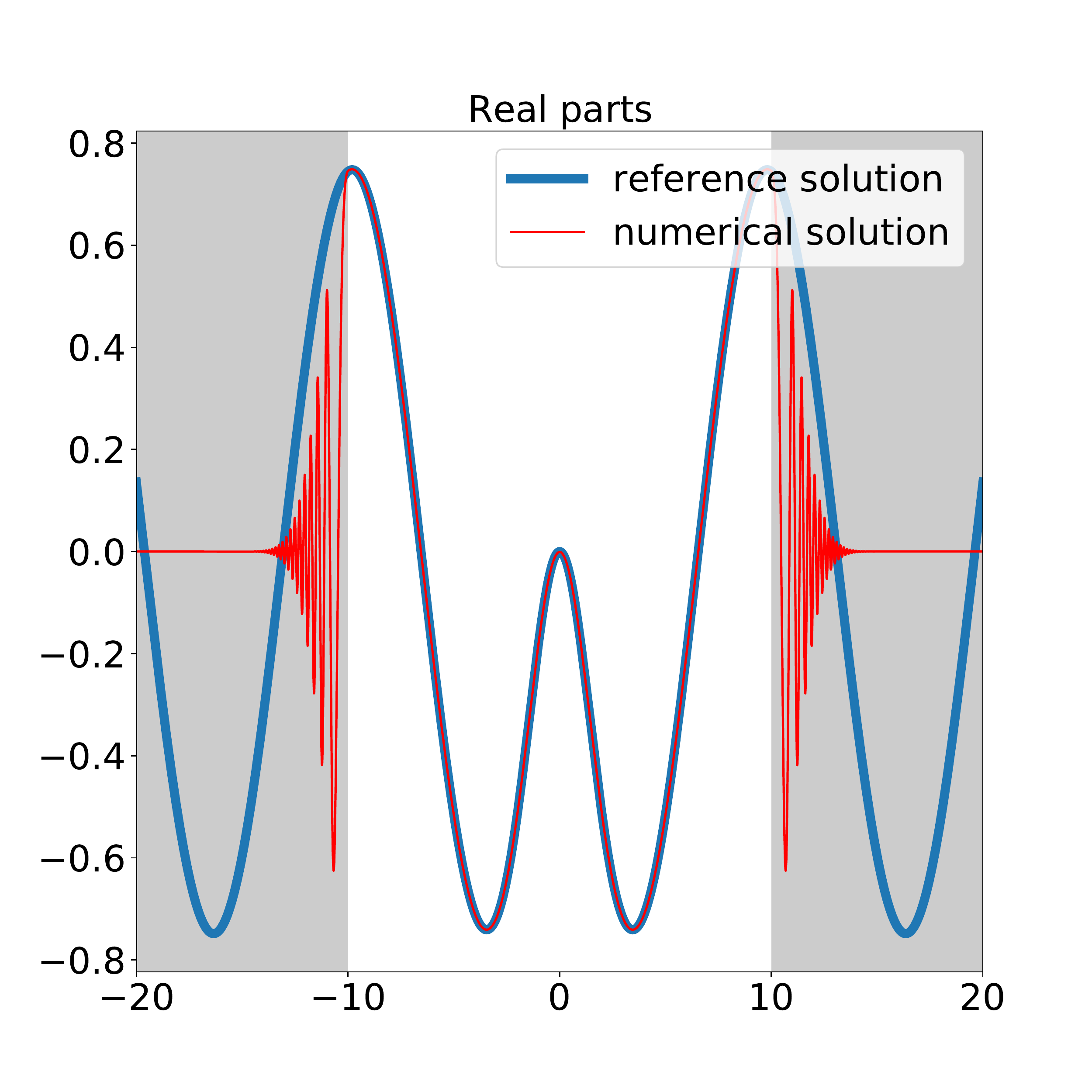}
 \includegraphics[width=0.49\textwidth]{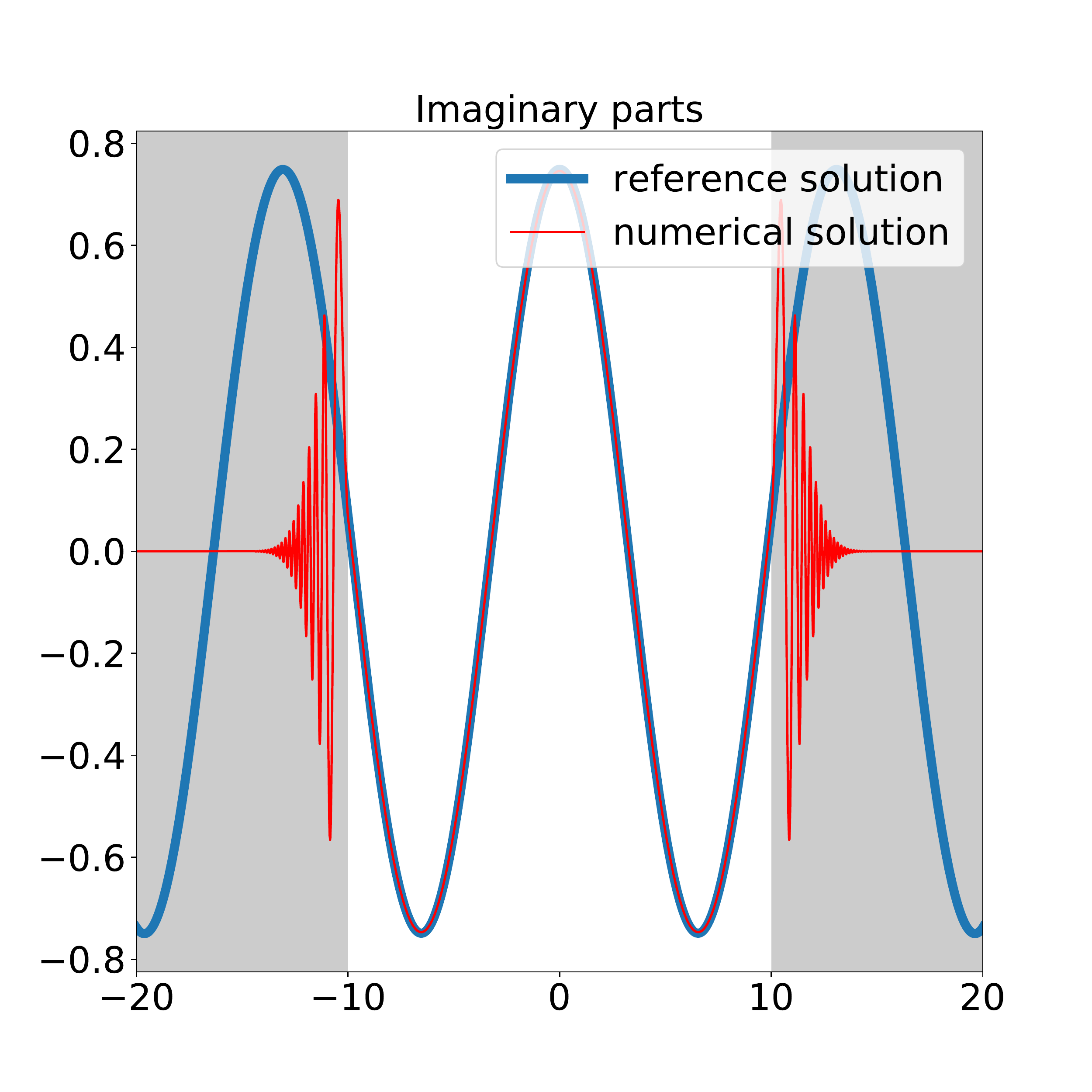}
 \caption{$s=1/2,k=2\pi/10$}
 \label{fig:ex3solutionsa}
\end{subfigure}
\begin{subfigure}{.48\textwidth}
 \centering
 \includegraphics[width=0.49\textwidth]{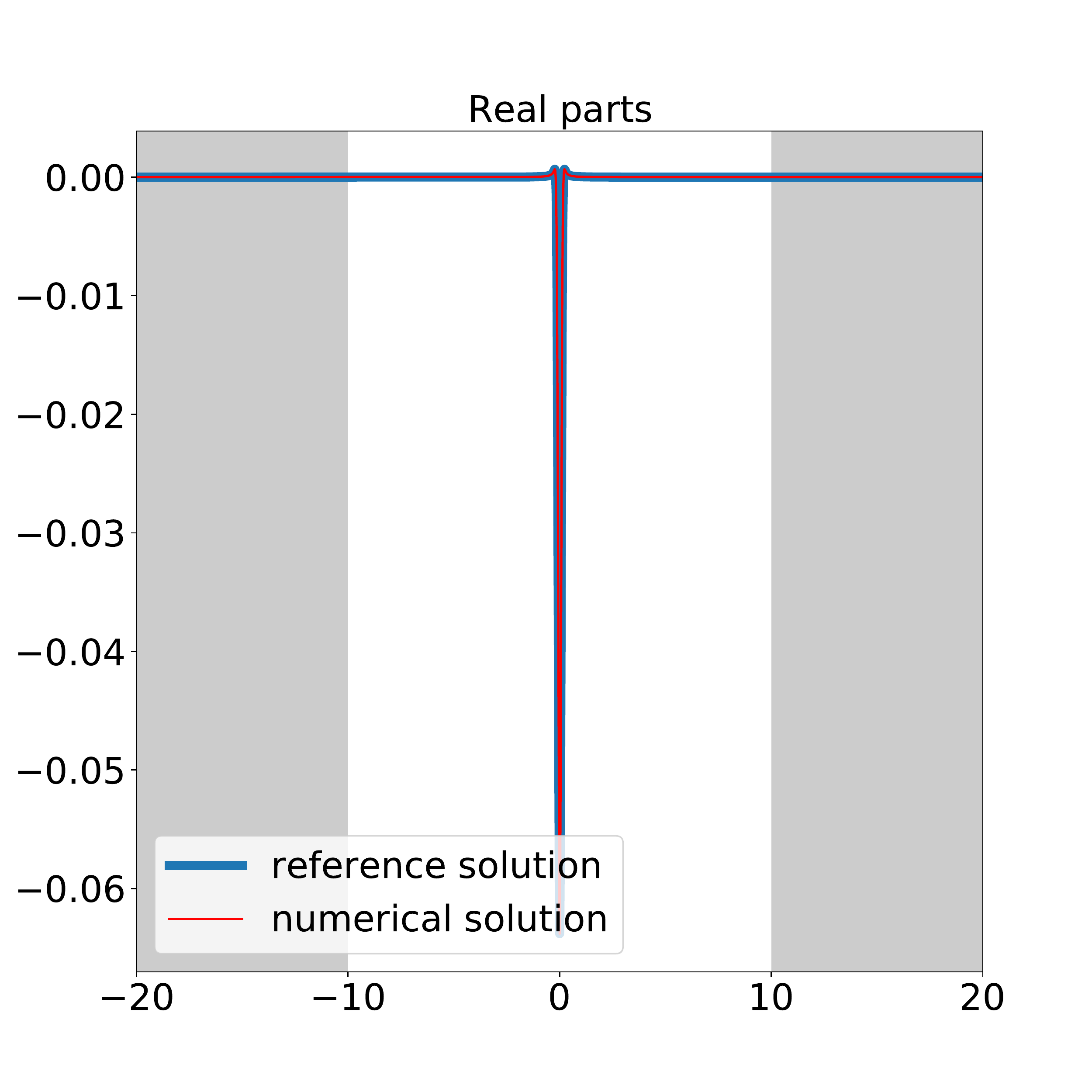}
 \includegraphics[width=0.49\textwidth]{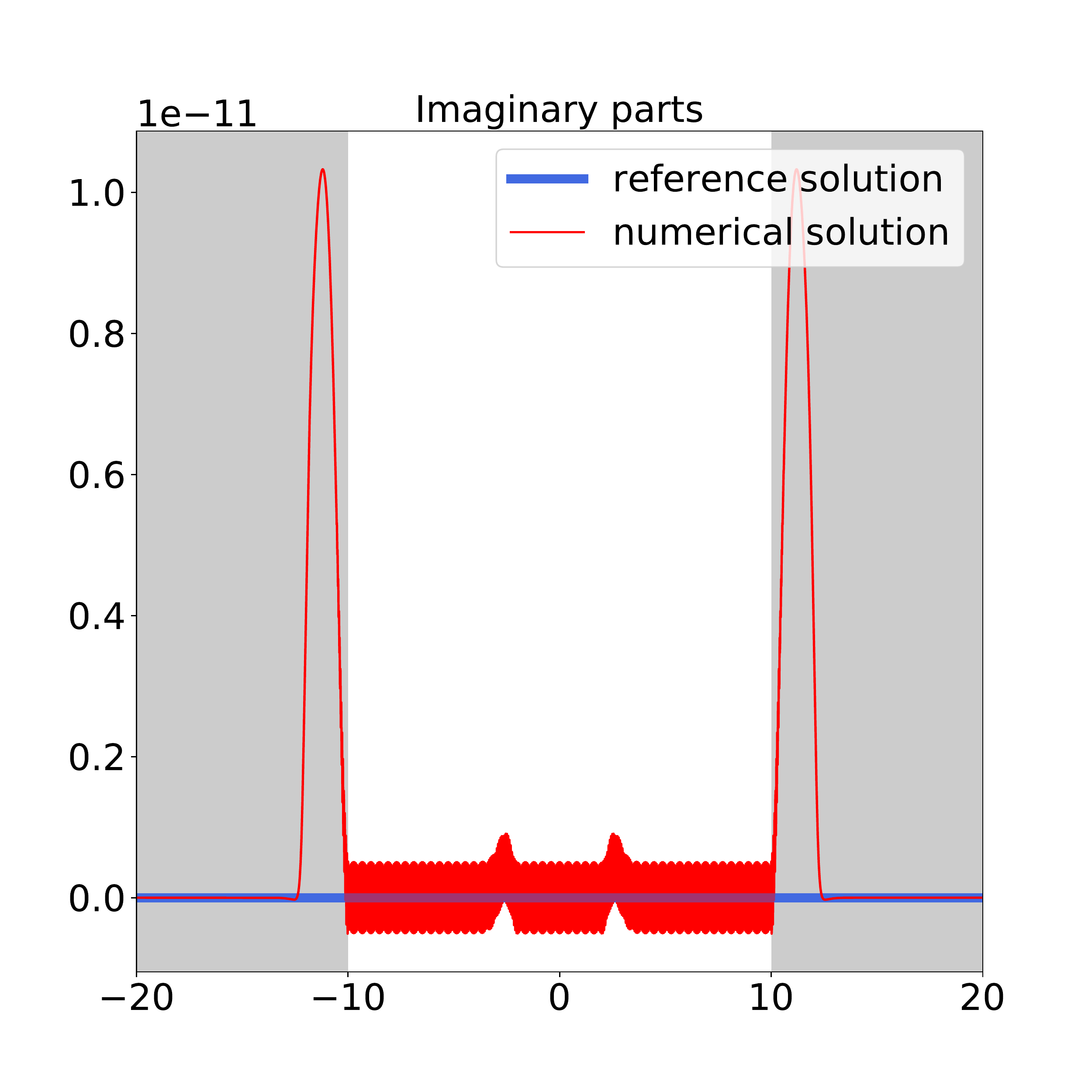}
 \caption{$s=1/2,k=32\pi/10$}
 \label{fig:ex3solutionsb}
\end{subfigure}
	\caption{Example 1.2: the reference solution and numerical solution $\hat{\tilde u}_h$ for $s=1/2$ and $k=2\pi/10,32\pi/10$, respectively. $\hat{\tilde u}_h$ is obtained on the mesh with $h=2^{-10}$. The PML region for numerical solutions are shaded in light grey.} \label{fig:ex3solutions}
\end{figure}

\begin{table}[htbp]
\centering
\begin{tabular}{|c|c|c|c|c||c|c|c|c|}
\hline
 $s$ & \multicolumn{4}{c||}{$1/4$} & \multicolumn{4}{c|}{$1/2$}  \\
 \hline
\diagbox{$h$}{$k$} & $2\pi/10$ & Order & $32\pi/10$ &Order & $2\pi/10$ & Order & $32\pi/10$ & Order \\
\hline
$2^{-5}$ &2.11e-01 & -- &2.53e-05 & -- &1.62e-01 & -- & 1.61e-04 & -- \\
\hline
$2^{-6}$ & 5.35e-03 & 5.30 & 4.51e-06 & 2.49 & 1.32e-02 & 3.62 & 4.97e-05 & 1.69 \\
\hline
$2^{-7}$ & 1.35e-03 & 1.98 & 9.71e-07& 2.22 & 3.22e-03 & 2.04 & 9.28e-06 & 2.42\\
\hline
$2^{-8}$ &2.80e-04 & 2.27 & 1.89e-07 & 2.36 & 8.62e-04 & 1.90 & 2.09e-06 & 2.15\\
\hline
\end{tabular}
\caption{(Example 1.2) $L^2$-errors for $s=1/4,1/2$ and different $k$. 
} \label{tab:ex1p3errors}
\end{table}

\textbf{Example 1.3.} The source function is taken as the Gaussian function \eqref{GSF}, the parameters are taken by $l=d_{pml}=1$, $\delta=1/4$, $z=40(10+\i)$, and the discontinuous kernel is taken by 
$$\gamma(x,y) = \frac{3}{\delta^3}\chi_{[-\delta,\delta]}(|x-y|).$$ 
Since this kernel is discontinuous at points $|x-y|=\delta$, a new difficulty is how to analytically continue the kernel into complex plane. To overcome the difficulty, we first approximate $\chi_{[-1,1]}(s)$ by a smooth function
\begin{align}
\hat\chi_{[-1,1]}(s) = \frac12\Big( \frac{e^{-\tau(|s|-1)}-1}{e^{-\tau(|s|-1)}+1} +1\Big)\ \mathrm{ with }\ \tau=-\frac{\ln tol}{\epsilon_0}. \label{eq:estsmofunc}
\end{align}
It is clear that $\hat\chi_{[-1,1]}(s)$ converges to $\chi_{[-1,1]}(s)$ as $tol$ and $\epsilon_0$ go to zero (Figure~\ref{fig:ex1p4kernel}). In practical simulations, we replace the kernel $\gamma(x,y) = \frac{3}{\delta^3}\chi_{[-\delta,\delta]}(|x-y|)$ by
\begin{align*}
\hat\gamma(x,y) = \frac{3}{\delta^3} \hat\chi_{[-1,1]}\big(\frac{|x-y|}{\delta}\big)
\end{align*}
with $tol=\epsilon_0=0.01$.

\begin{figure}[htbp]
\centering
	\includegraphics[width=0.9\textwidth]{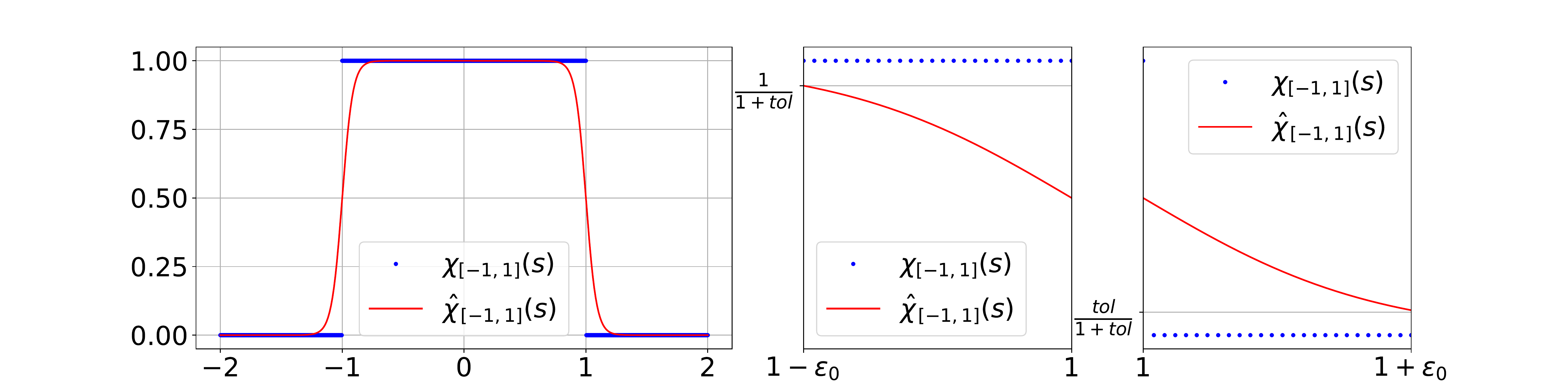}
	\caption{Example 1.4: $\chi_{[-1,1]}(s)$ and $\hat\chi_{[-1,1]}(s)$.} \label{fig:ex1p4kernel}
\end{figure}

Figure~\ref{fig:ex1p4solutions} plots the solutions for $k=2\pi,4\pi$, respectively. One can see that the solutions decay exponentially outside $\Omega$ and the real part of $z$ makes the wave oscillate faster in the PML region. Table~\ref{tab:ex1p4errors} shows the errors and convergence orders by refining mesh size $h$. 

\begin{figure}[htbp]
\centering
\begin{subfigure}{.48\textwidth}
 \centering
 \includegraphics[width=0.49\textwidth]{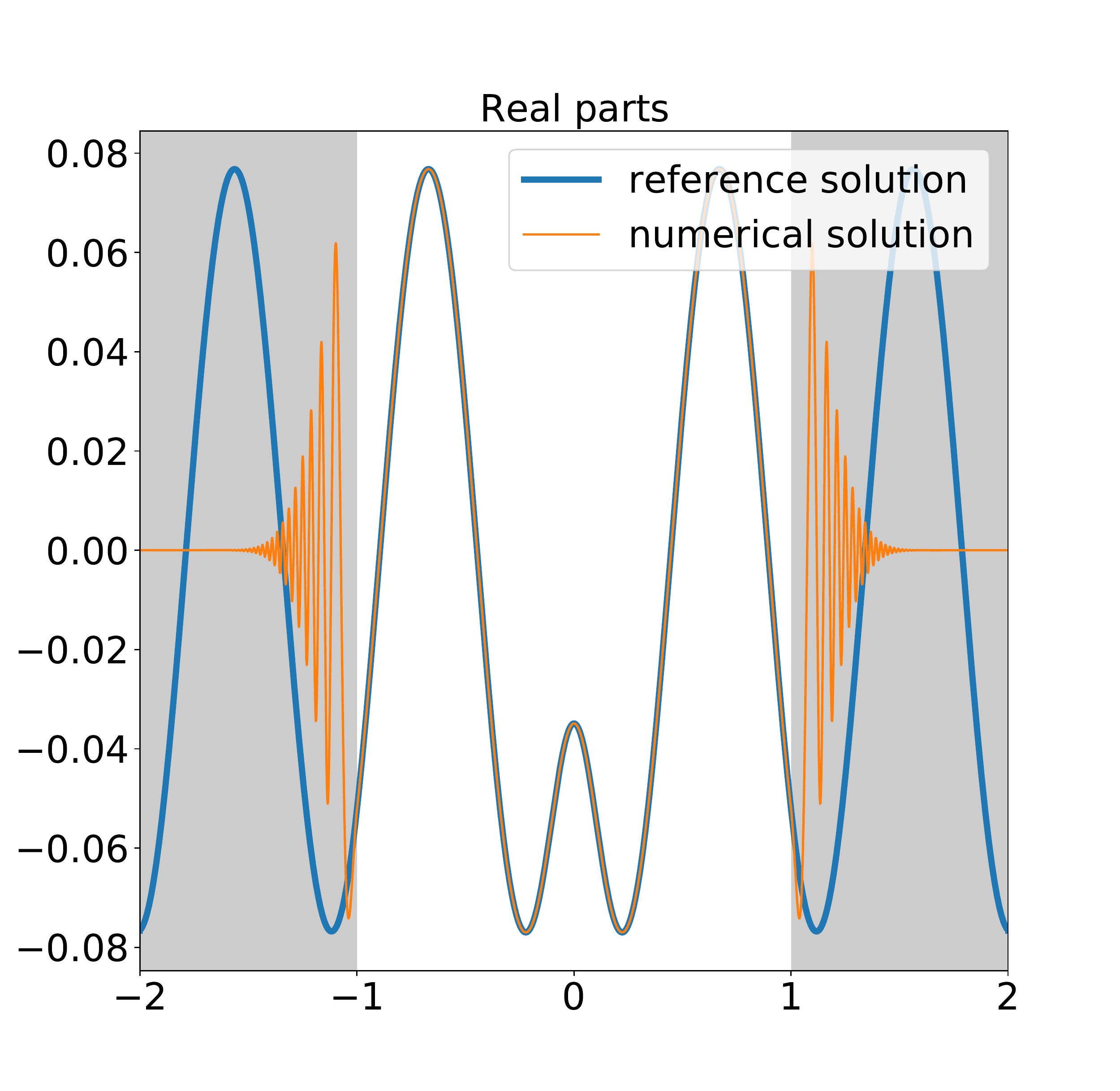}
 \includegraphics[width=0.49\textwidth]{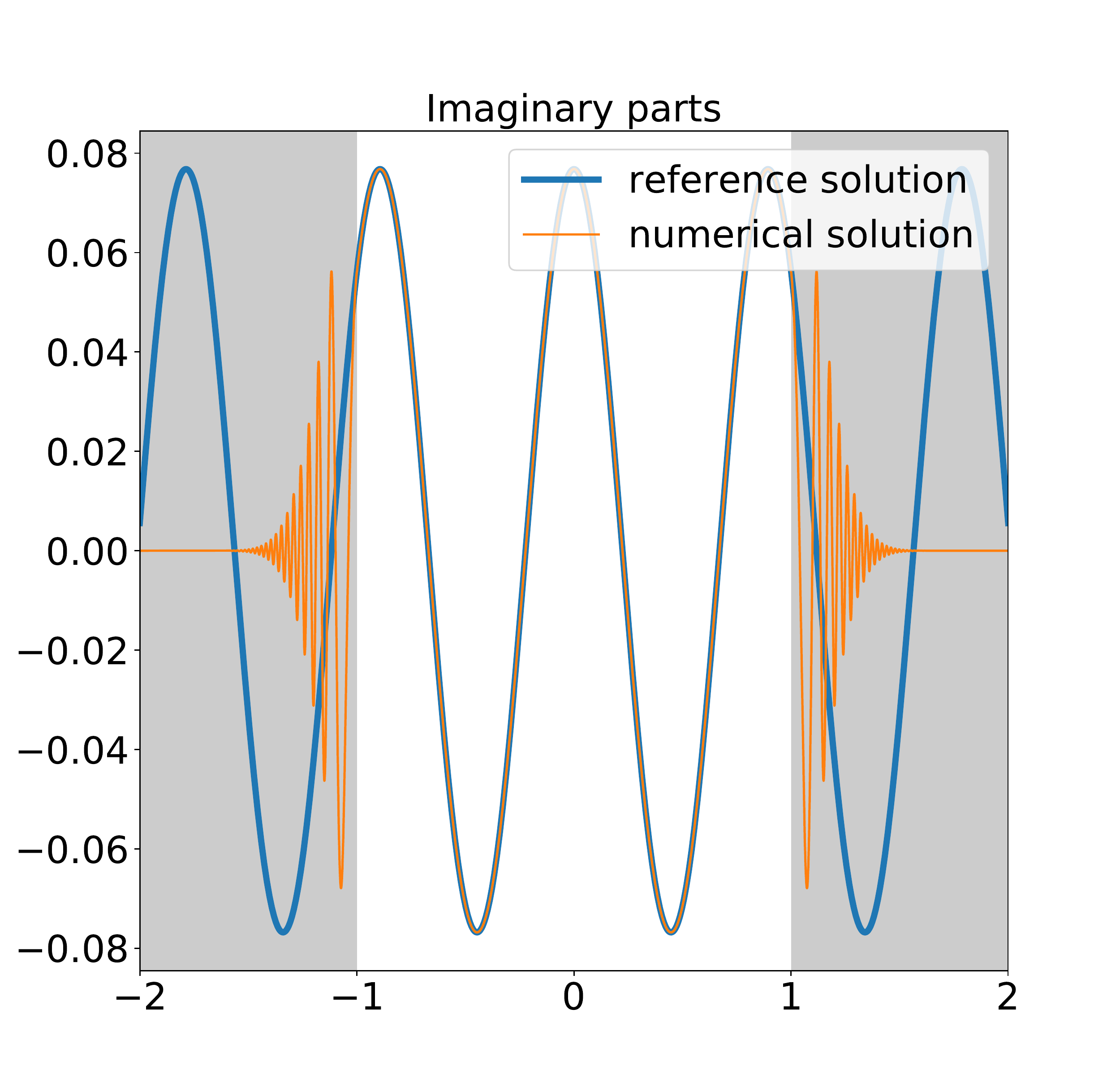}
 \caption{$k=2\pi,\delta=1/4$}
 \label{fig:ex1p4solutionsa}
\end{subfigure}
\begin{subfigure}{.48\textwidth}
 \centering
 \includegraphics[width=0.49\textwidth]{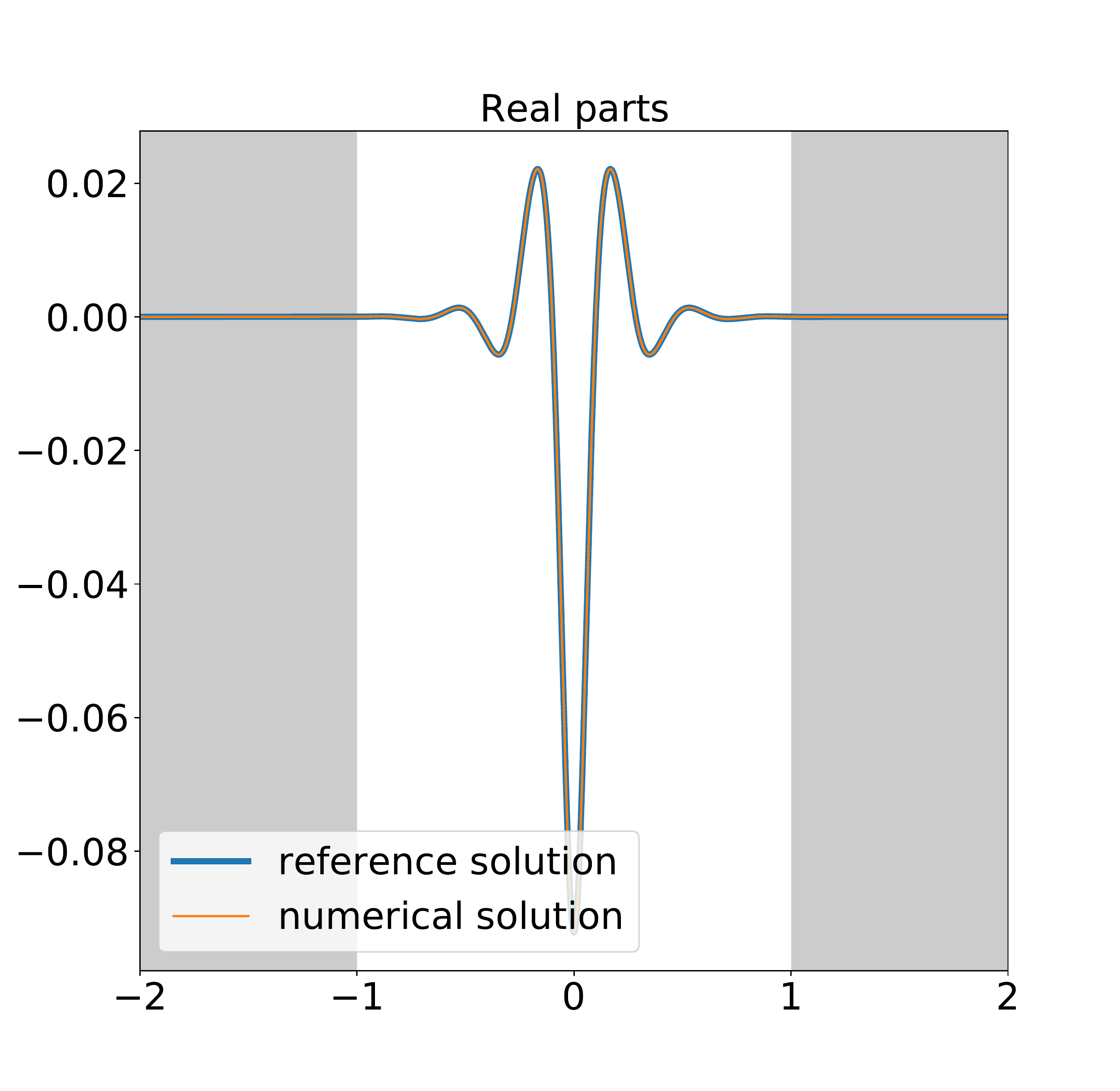}
 \includegraphics[width=0.49\textwidth]{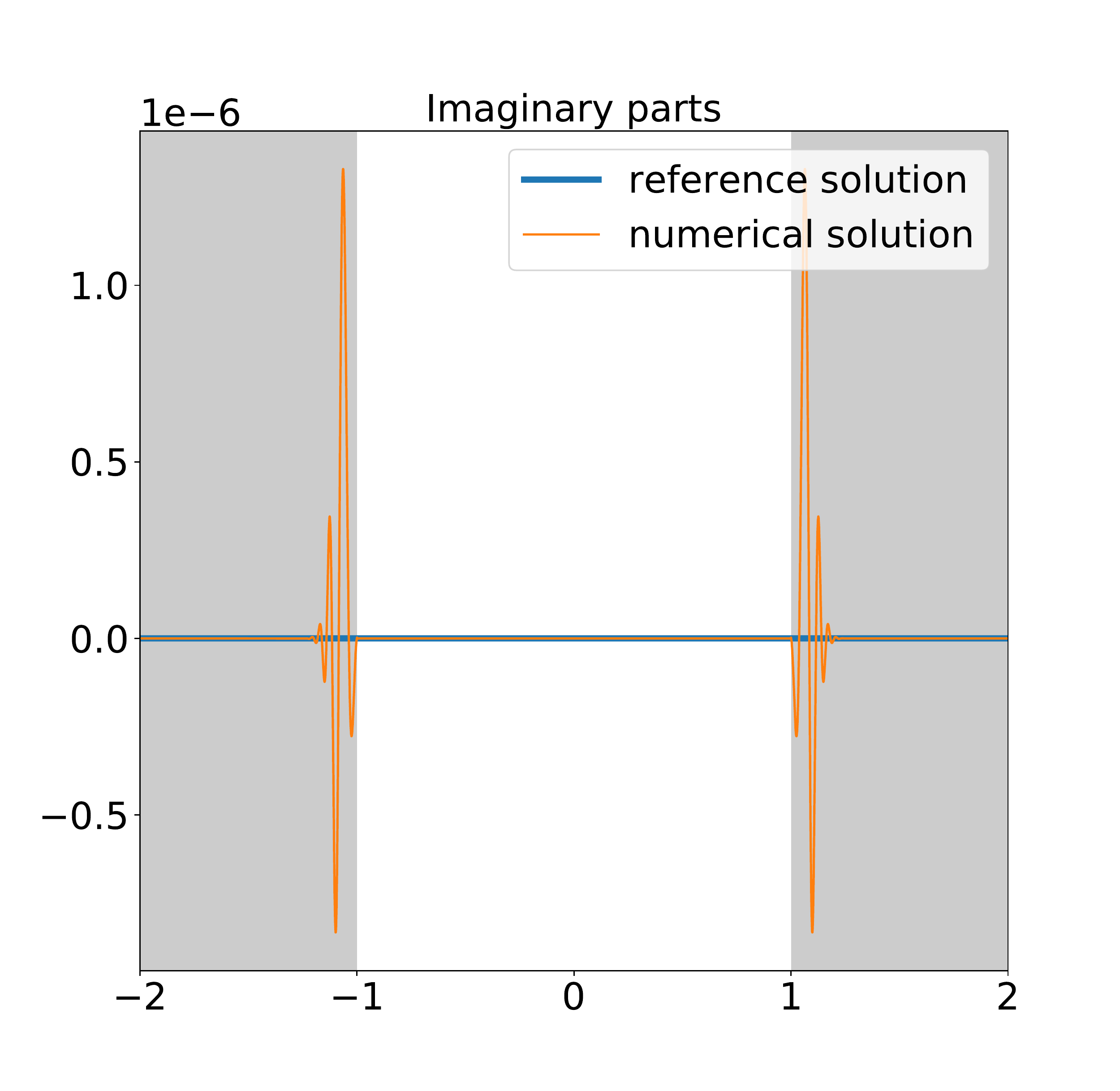}
 \caption{$k=4\pi,\delta=1/4$}
 \label{fig:ex1p4solutionsb}
\end{subfigure}
	\caption{(Example 1.3) the reference solution and numerical solution $\hat{\tilde u}_h$ for $\delta=1/4$ and $k=2\pi,4\pi$, respectively. $\hat{\tilde u}_h$ is obtained on the mesh with $D_f=2^{12}-1$. The PML region for numerical solutions are shaded in light grey.} \label{fig:ex1p4solutions}
\end{figure}

\begin{table}[htbp]
\centering
\begin{tabular}{|c|c|c|c|c|c|c|}
 \hline
\diagbox{$h$}{$k$} & $2\pi$ & Order & $4\pi$ & Order \\
\hline
$2^{-5}$ & 5.05e-02 & -- & 2.3435e-04 & -- \\
\hline
$2^{-6}$ &  6.46e-03 & 2.97 & 5.9073e-05 & 1.99\\
\hline
$2^{-7}$ & 2.55e-04 & 4.66 & 1.4621e-05 & 2.01\\
\hline
$2^{-8}$ & 7.78e-05 & 1.71 & 3.4811e-06 &2.07\\
\hline
\end{tabular}
\caption{(Example 1.3) $L^2$-errors for $k=2\pi,4\pi$ and $\delta=1/4$. 
} \label{tab:ex1p4errors}
\end{table}

\subsection{Numerical examples in two dimensions} \label{subsec:ne_2D}
Here we consider three examples. For PMLs in Cartesian coordinates, we choose $\Omega=[-l,l]^2$ and $\Omega_p$ as \eqref{eq:Omegaps}, and set
\begin{align}
\sigma_1(t) = \sigma_2(t) = \begin{cases}
0, & -l\leq t\leq l,\\
\frac{1}{d_{pml}}(|t|-l), &\quad l<|t|.
\end{cases}
\end{align}
For PMLs in polar coordinates, we choose $\Omega=\{x\in\R^2:|x|_2\leq l_r\}$ and $\Omega_p$ as \eqref{eq:Omegapc} and set
\begin{align}
\sigma_r(t) = \begin{cases}
0, & 0\leq t\leq l_r,\\
\frac{1}{d_{pml}}(t-l_r), & l_r<t.
\end{cases}
\end{align}

Figure~\ref{fig:2Dmesh} illustrates the two-dimensional meshes for PMLs in Cartesian coordinates and polar coordinates, respectively.
\begin{figure}[htbp]
\centering
	\includegraphics[width=0.4\textwidth]{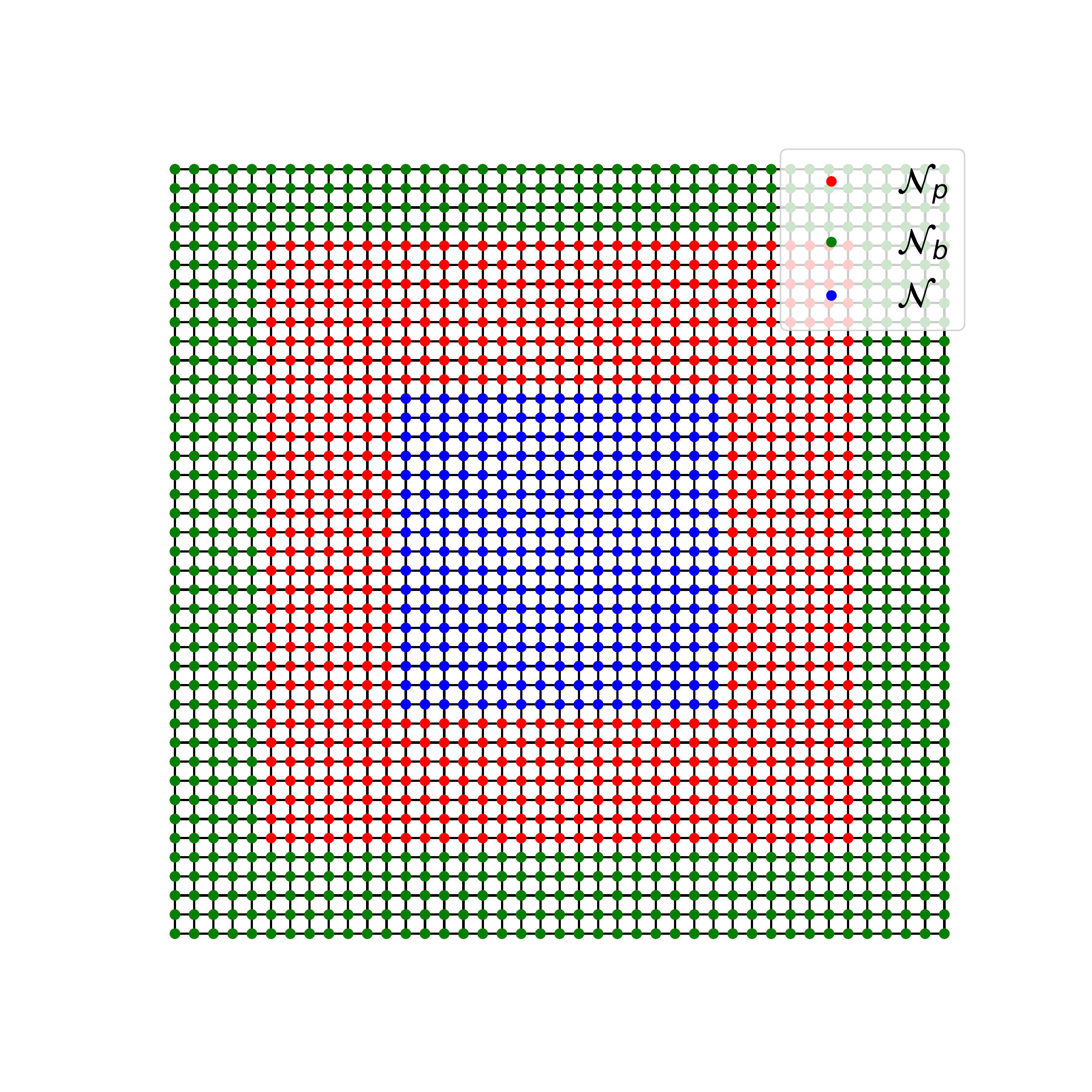}
	\includegraphics[width=0.4\textwidth]{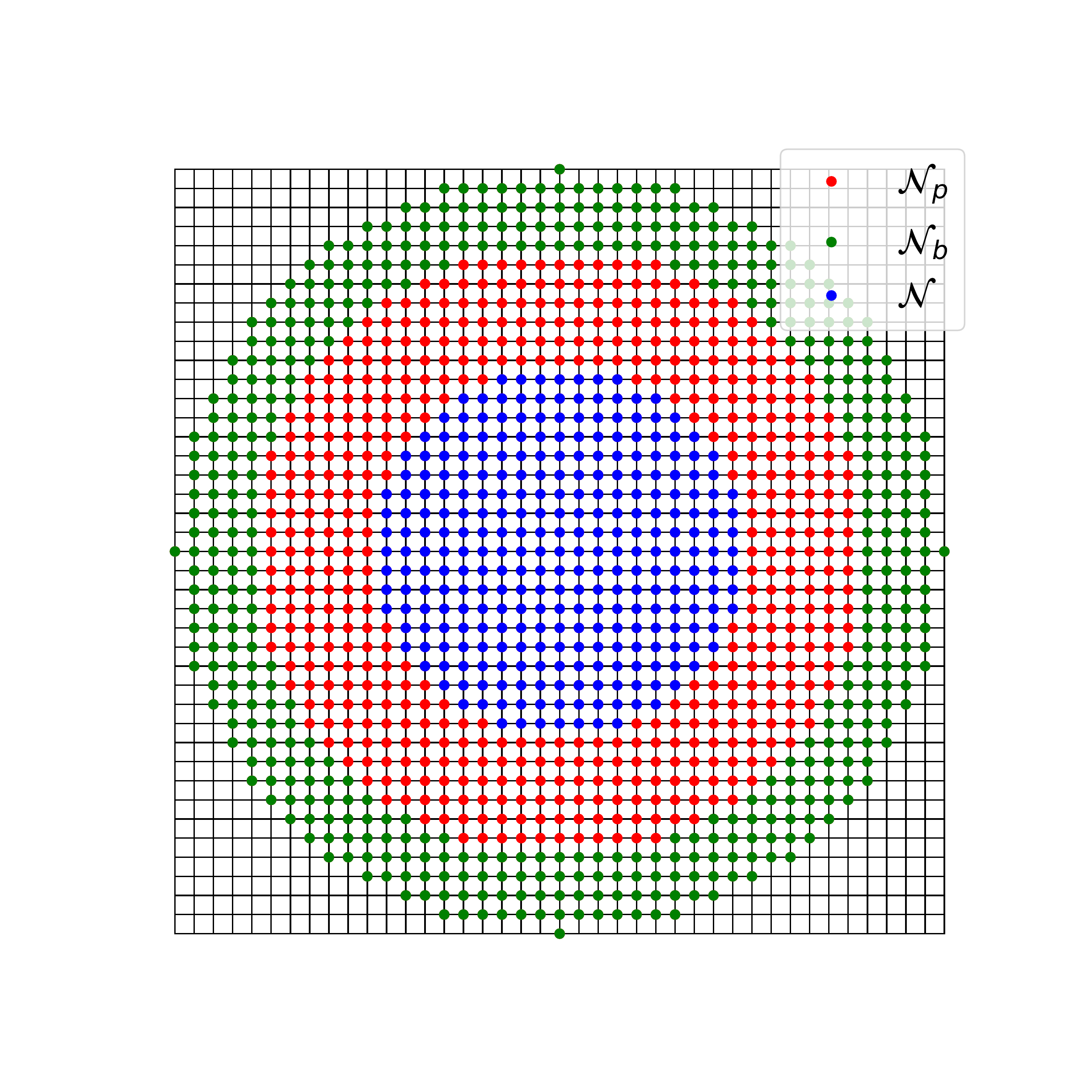}
	\caption{Two dimensional meshes for PMLs in Cartesian coordinates (left) and polar coordinates (right). $\mathcal{N}_p$, $\mathcal{N}_b$ and $\mathcal{N}$ are the sets of points with indices in $\mathcal{I}_p$, $\mathcal{I}_b$ and $\mathcal{I}$, respectively.} \label{fig:2Dmesh}
\end{figure}

\textbf{Example 2.1}: We first consider the exponential kernel $\gamma(x,y)=\frac{1}{3\pi c_\gamma^4} e^{-\frac{|x-y|_2}{c_\gamma}}$. In the simulations, we set PML coefficient $z=20\i$, $l_r=l=d_{pml}=1$, and take the source function
 \begin{align}
 f(x) = 2\sqrt{\pi} e^{-4\pi^2|x|^2}. \label{eq:ex2p1source}
 \end{align}
 
Figure~\ref{fig:ex2p1solutions} plots the numerical solutions and the projections of their contours on the `walls' of graphs for each dimension. One can see that, for $k=2\pi$ and $c_\gamma=0.1/k$, numerical solutions by PMLs in both Cartesian coordinates and polar coordinates oscillate in $\Omega$ and decay exponentially in the PML layers. For $k=20\pi$ and $c_\gamma=1/k$, the solutions themself decay outside the support of $f(x)$, but PMLs make their imaginary parts oscillate at a amplitude of 2e-23 in $\Omega$. Table~\ref{tab:ex2p1errors} shows the errors and convergence orders by refining mesh size $h$. 

\begin{figure}[htbp]
\centering
\begin{subfigure}{.48\textwidth}
 \centering
	\includegraphics[width=0.49\textwidth]{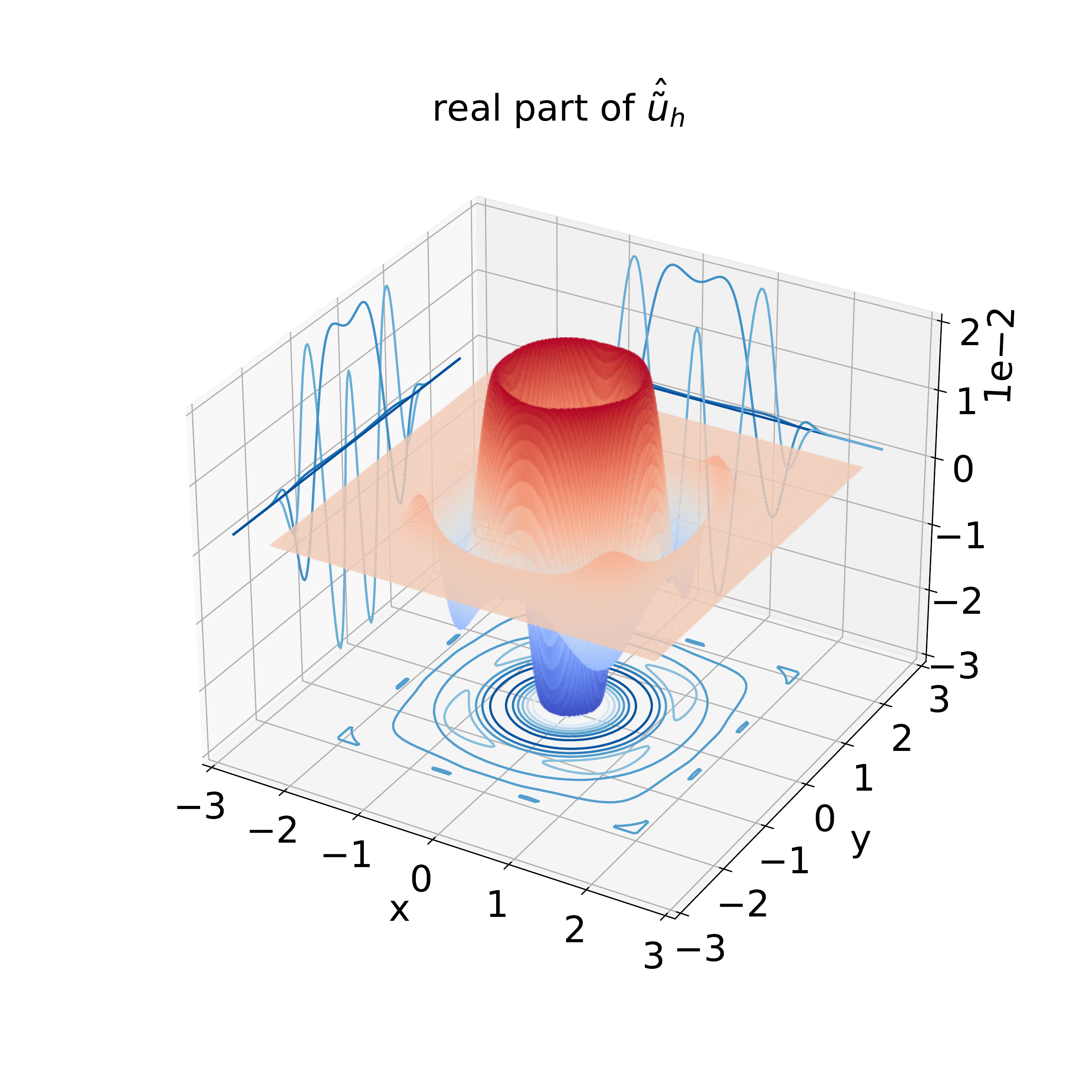}
	\includegraphics[width=0.49\textwidth]{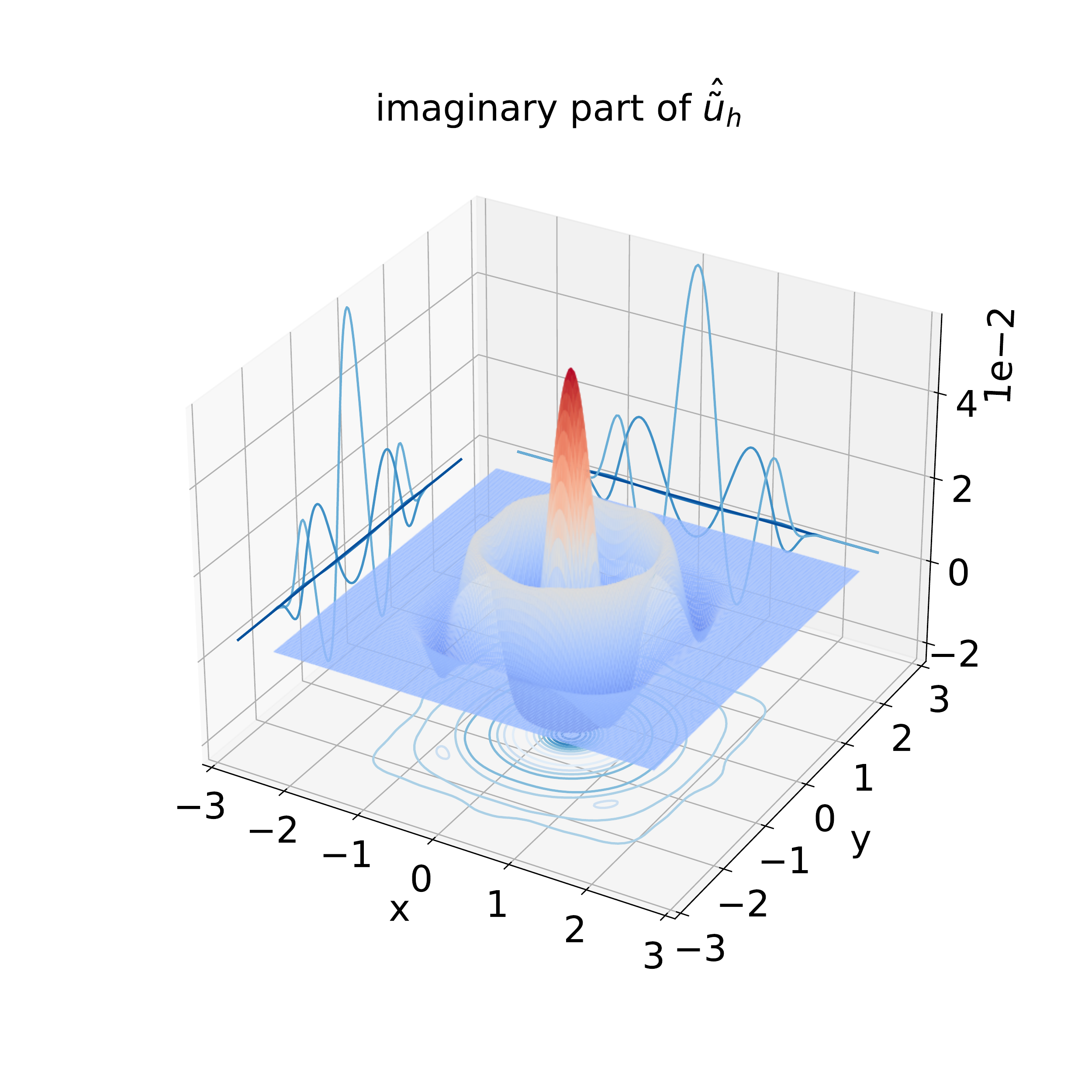}
 \caption{$k=2\pi,c_\gamma=0.1/k$ in Cartesian coordinates}
 \label{fig:ex2p1solutionsa}
\end{subfigure}
\begin{subfigure}{.48\textwidth}
 \centering
	\includegraphics[width=0.49\textwidth]{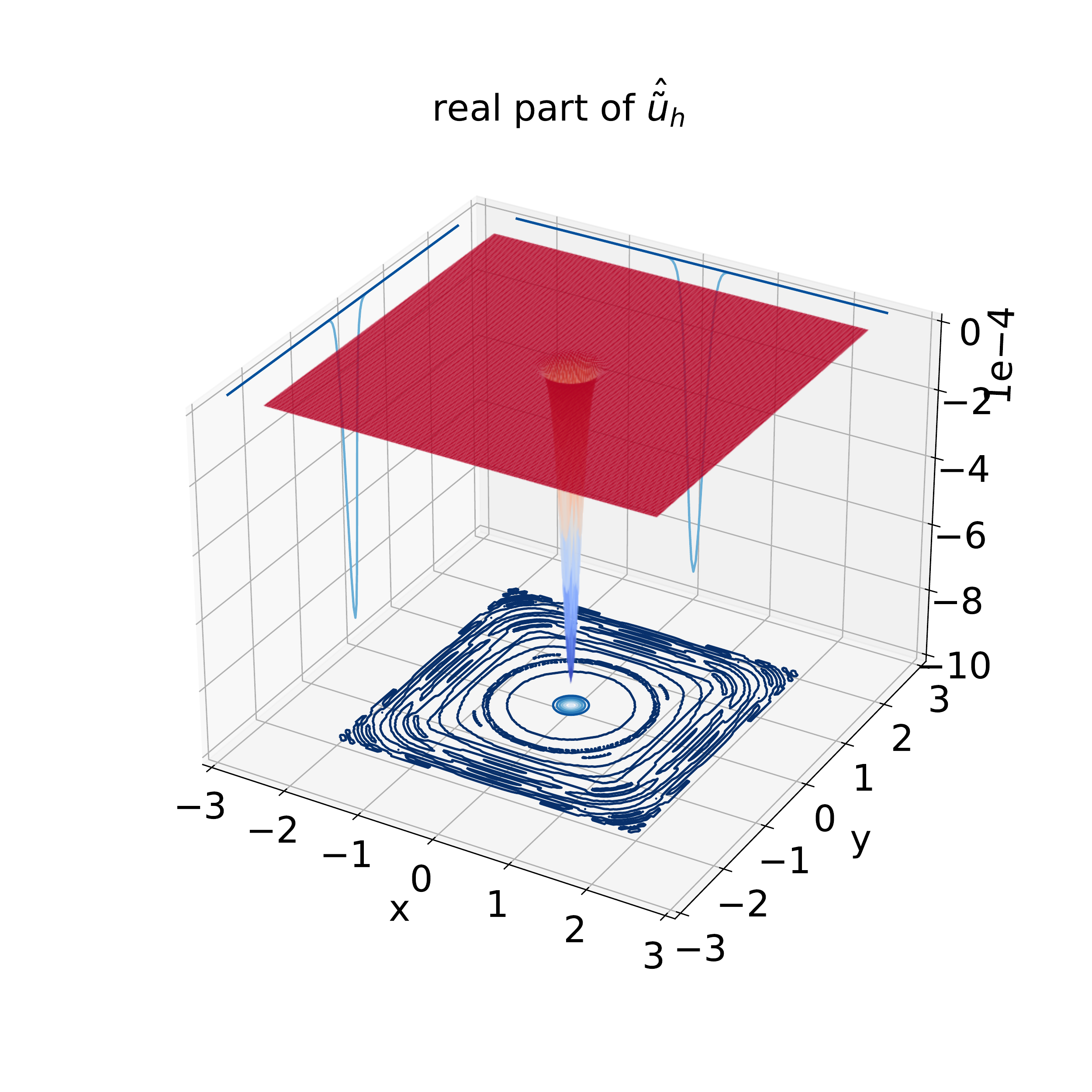}
	\includegraphics[width=0.49\textwidth]{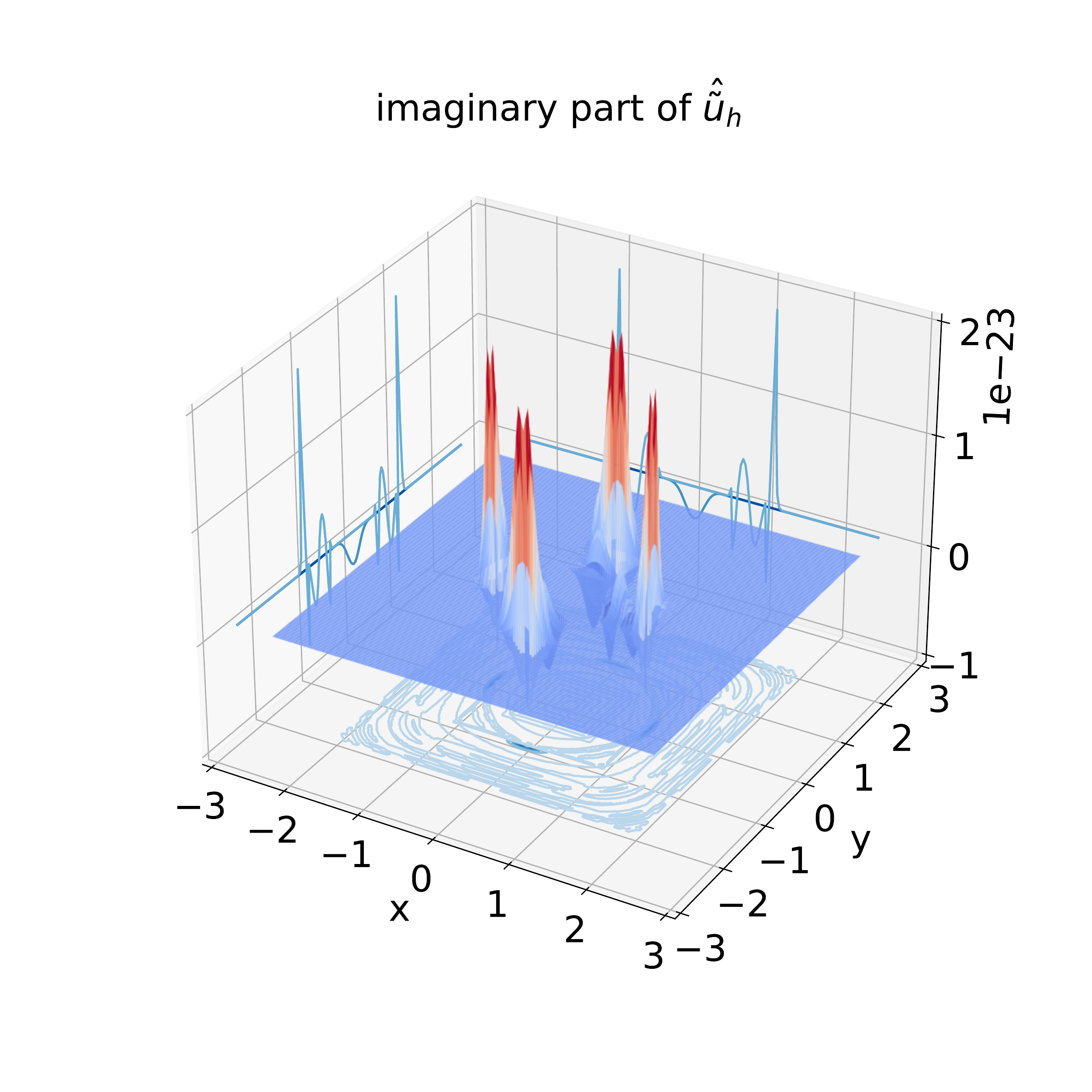}
 \caption{$k=20\pi,c_\gamma=1/k$ in Cartesian coordinates}
 \label{fig:ex2p1solutionsb}
\end{subfigure}
\begin{subfigure}{.48\textwidth}
 \centering
	\includegraphics[width=0.49\textwidth]{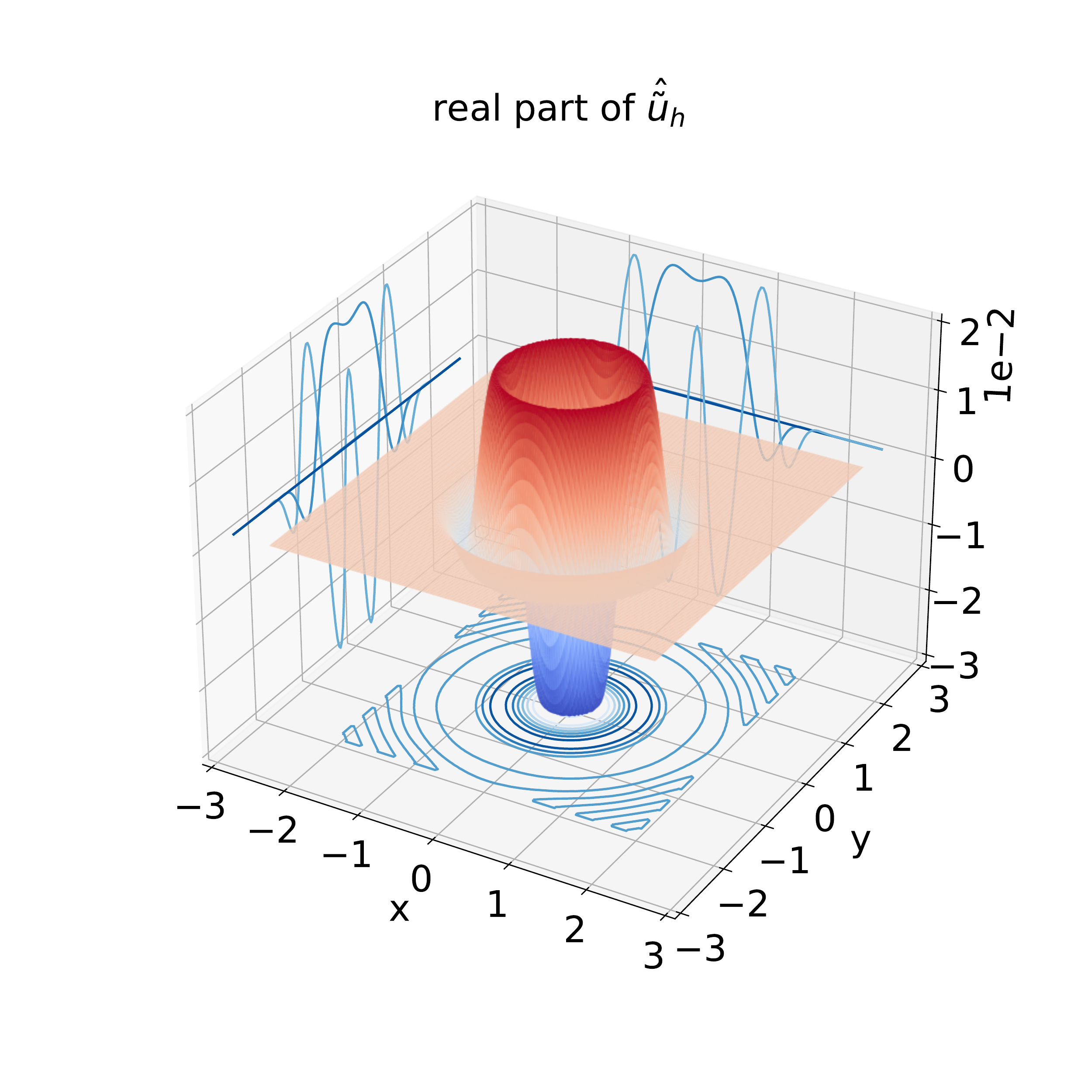}
	\includegraphics[width=0.49\textwidth]{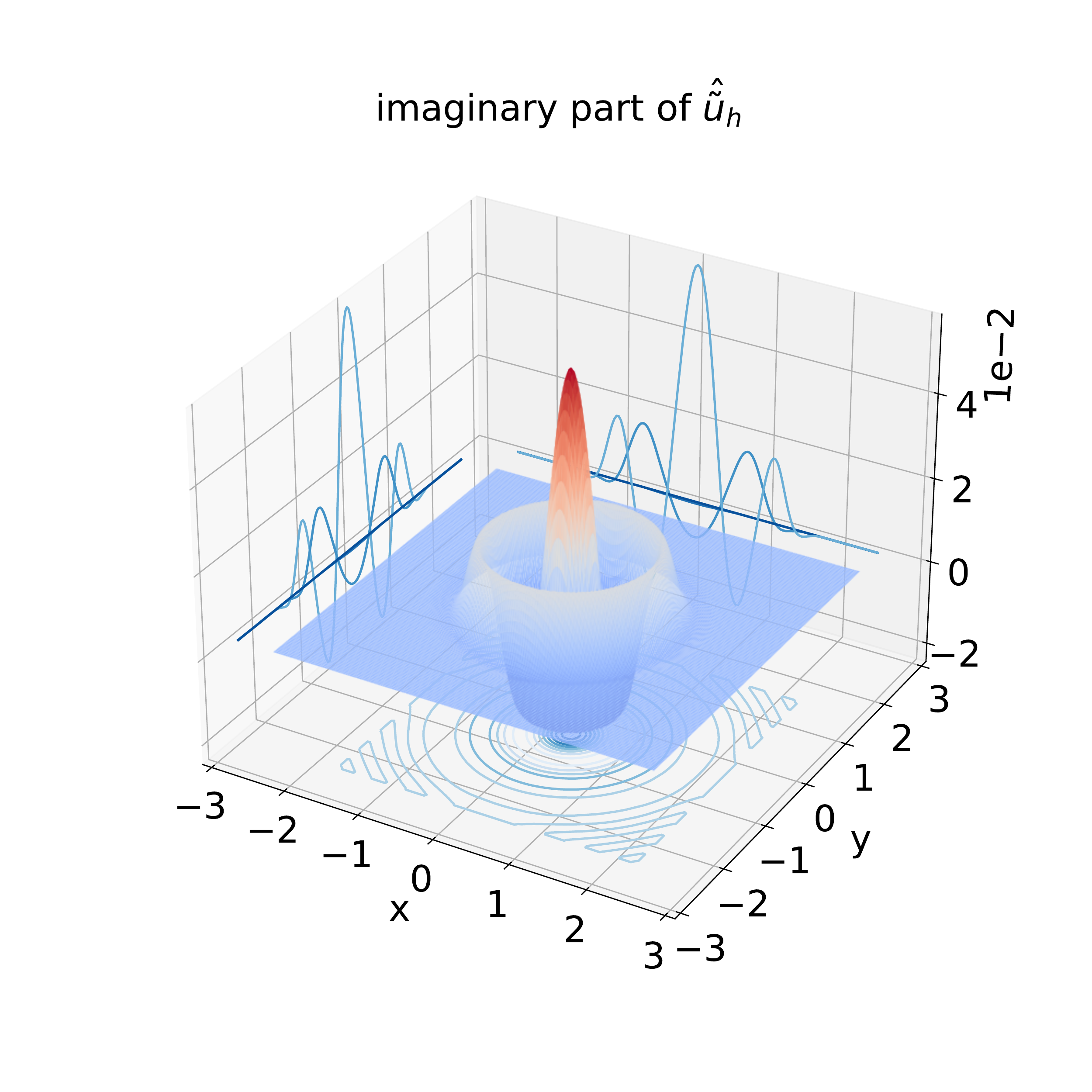}
 \caption{$k=2\pi,c_\gamma=0.1/k$ in polar coordinates}
 \label{fig:ex2p1solutionsc}
\end{subfigure}
\begin{subfigure}{.48\textwidth}
 \centering
	\includegraphics[width=0.49\textwidth]{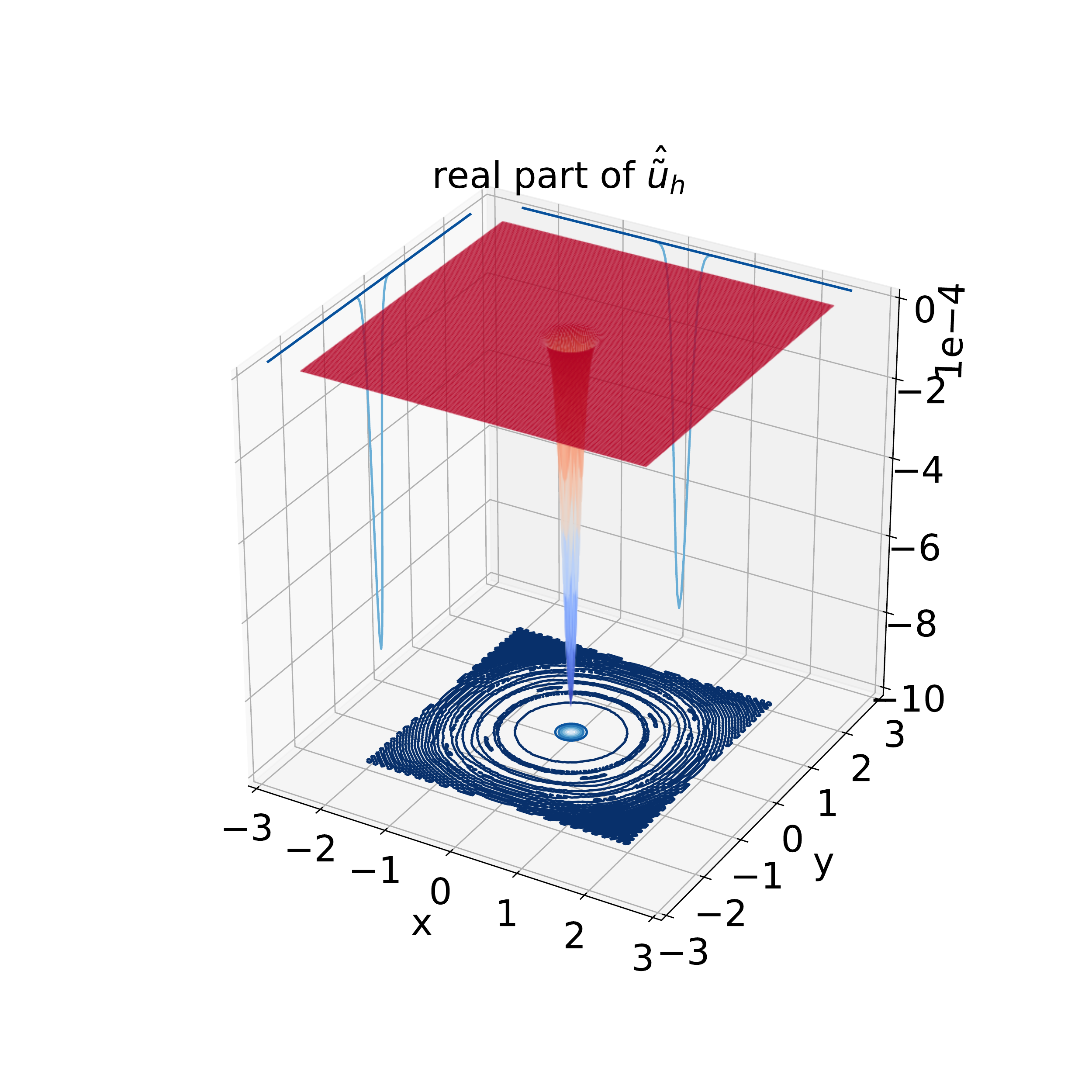}
	\includegraphics[width=0.49\textwidth]{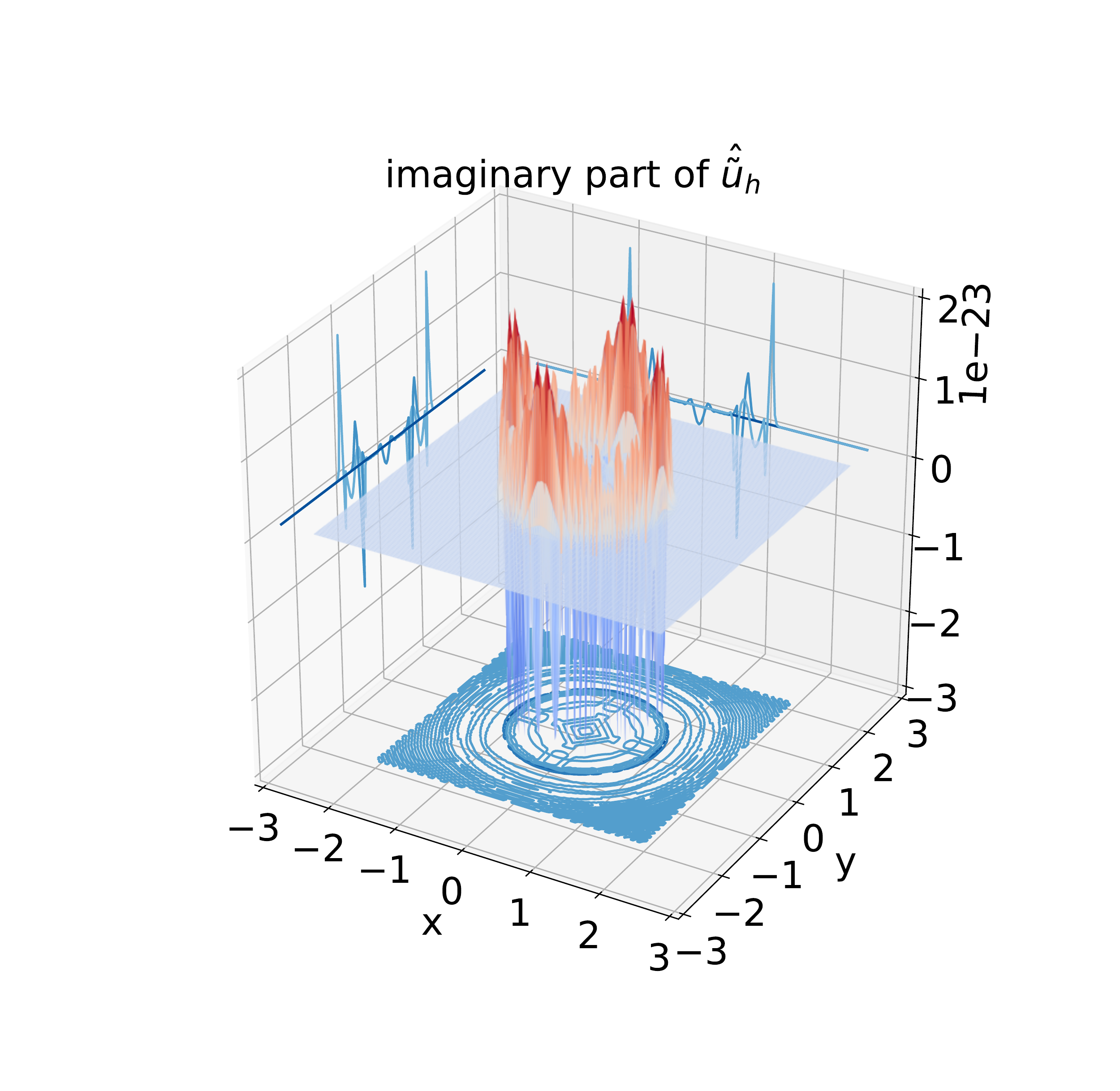}
 \caption{$k=20\pi,c_\gamma=1/k$ in polar coordinates}
 \label{fig:ex2p1solutionsc}
\end{subfigure}
	\caption{(Example 2.1) numerical solutions $\hat{\tilde u}_h$ for the exponential kernel by PMLs in Cartesian coordinates and polar coordinates.} \label{fig:ex2p1solutions}
\end{figure}

\begin{table}[htbp]
\centering
\begin{tabular}{|c|c|c|c|c||c|c|c|c|}
\hline
 & \multicolumn{4}{c||}{$k=2\pi,c_\gamma=0.1/k$}  & \multicolumn{4}{c|}{$k=20\pi,c_\gamma=1/k$}\\
 \hline
\diagbox{$h$}{PML} & CPML & Order & PPML & Order & CPML & Order & PPML & Order \\
\hline
$2^{-1}$ &  7.27e-02 & -- & 3.23e-02 & -- & 1.04e-05 & -- & 4.46e-06 & -- \\
\hline
$2^{-2}$ & 2.24e-02 & 1.70 & 1.34e-02 & 1.27 & 3.19e-06 & 1.70 & 1.49e-06  & 1.58\\
\hline
$2^{-3}$ & 2.60e-03 & 3.11 & 1.72e-03 & 2.96 & 6.25e-07 & 2.35 & 3.02e-07  & 2.30\\
\hline
$2^{-4}$ &  5.60e-04 & 2.21 & 4.72e-04 & 1.87 & 1.01e-07 & 2.63 & 4.96e-08 &2.61\\
\hline
\end{tabular}
\caption{(Example 2.1) $L^2$-errors for the exponential kernel by PMLs. CPML and PPML are short for PMLs in Cartesian coordinates and polar coordinates, respectively.} \label{tab:ex2p1errors}
\end{table}

\textbf{Example 2.2}: We here consider a piecewise constant kernel
\begin{equation}\label{kerc}
\gamma(x,y) = \frac{8}{\pi\delta^4}\chi_{[-\delta,\delta]}(|x-y|), \quad \quad \forall x,y\in\R^2.
\end{equation} 
To analytically continue this discontinuous kernel \eqref{kerc} into complex plane, we also approximate \eqref{kerc} by the smooth function \eqref{eq:estsmofunc}. In the simulations, we set the PML coefficient $z=40\i$, $l_r=l=d_{pml}=1$ and take the same source function as \eqref{eq:ex2p1source}.

Figure~\ref{fig:ex2p3solutions} plots the numerical solutions for $k=2\pi,\delta=0.5$ and $k=20\pi,\delta=0.5$ by two kinds of PMLs. Table~\ref{tab:ex2p3errors} shows the errors and convergence order for and $k$ by refining $h$. 

\begin{figure}[htbp]
\centering
\begin{subfigure}{.48\textwidth}
 \centering
	\includegraphics[width=0.49\textwidth]{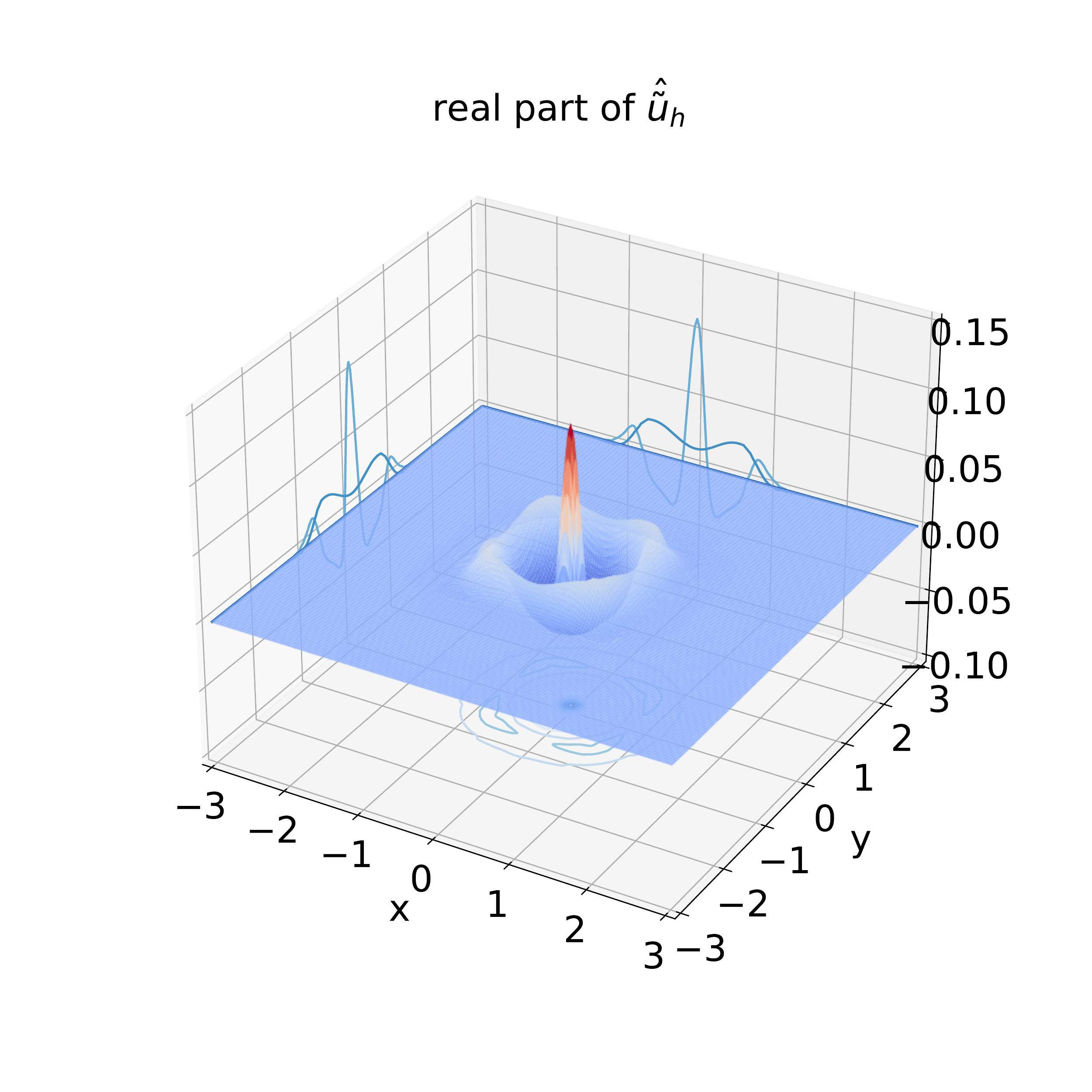}
	\includegraphics[width=0.49\textwidth]{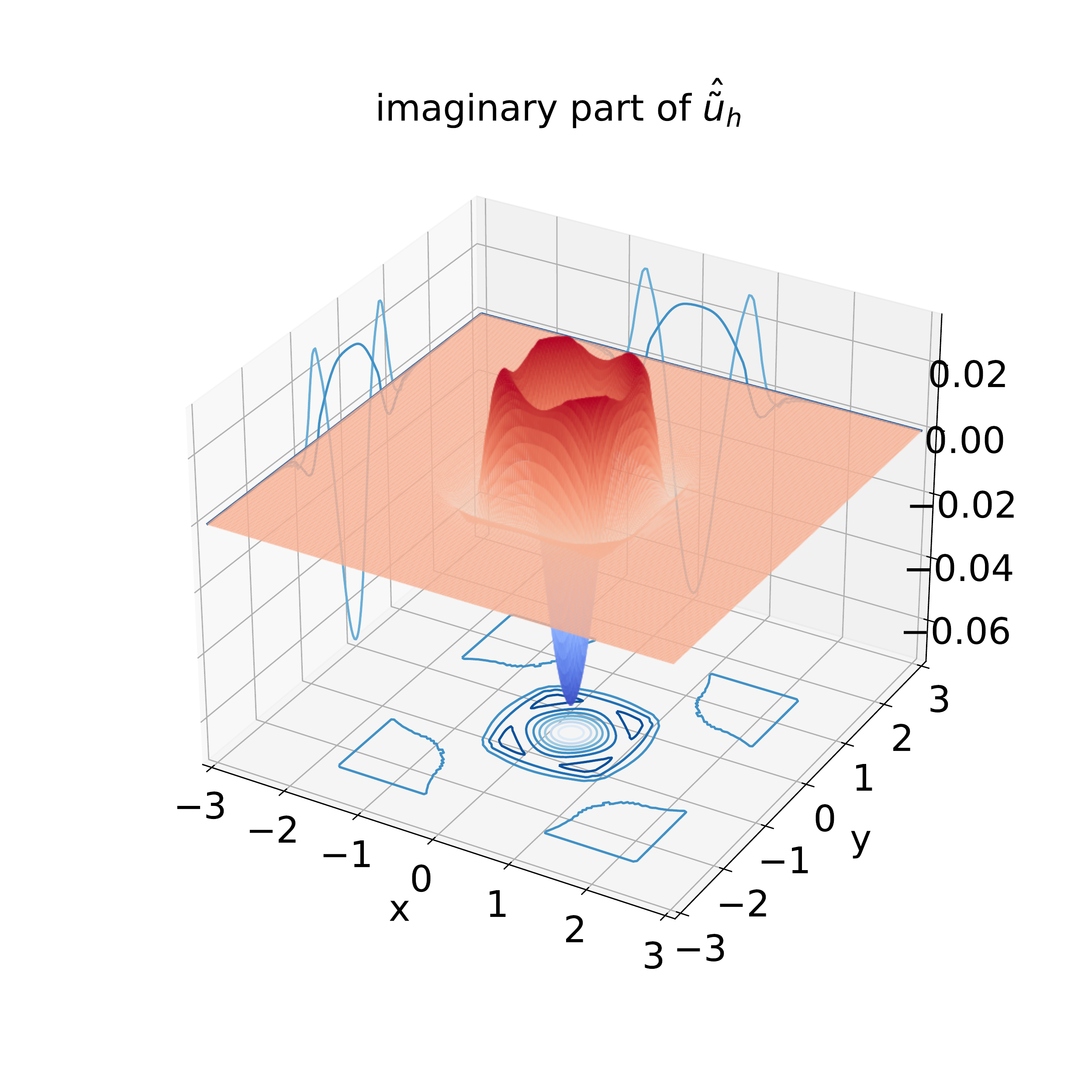}
 \caption{$k=2\pi,\delta=0.5$ by PMLs in Cartesian coordinates}
 \label{fig:ex2p2solutionsa}
\end{subfigure}
\begin{subfigure}{.48\textwidth}
 \centering
	\includegraphics[width=0.49\textwidth]{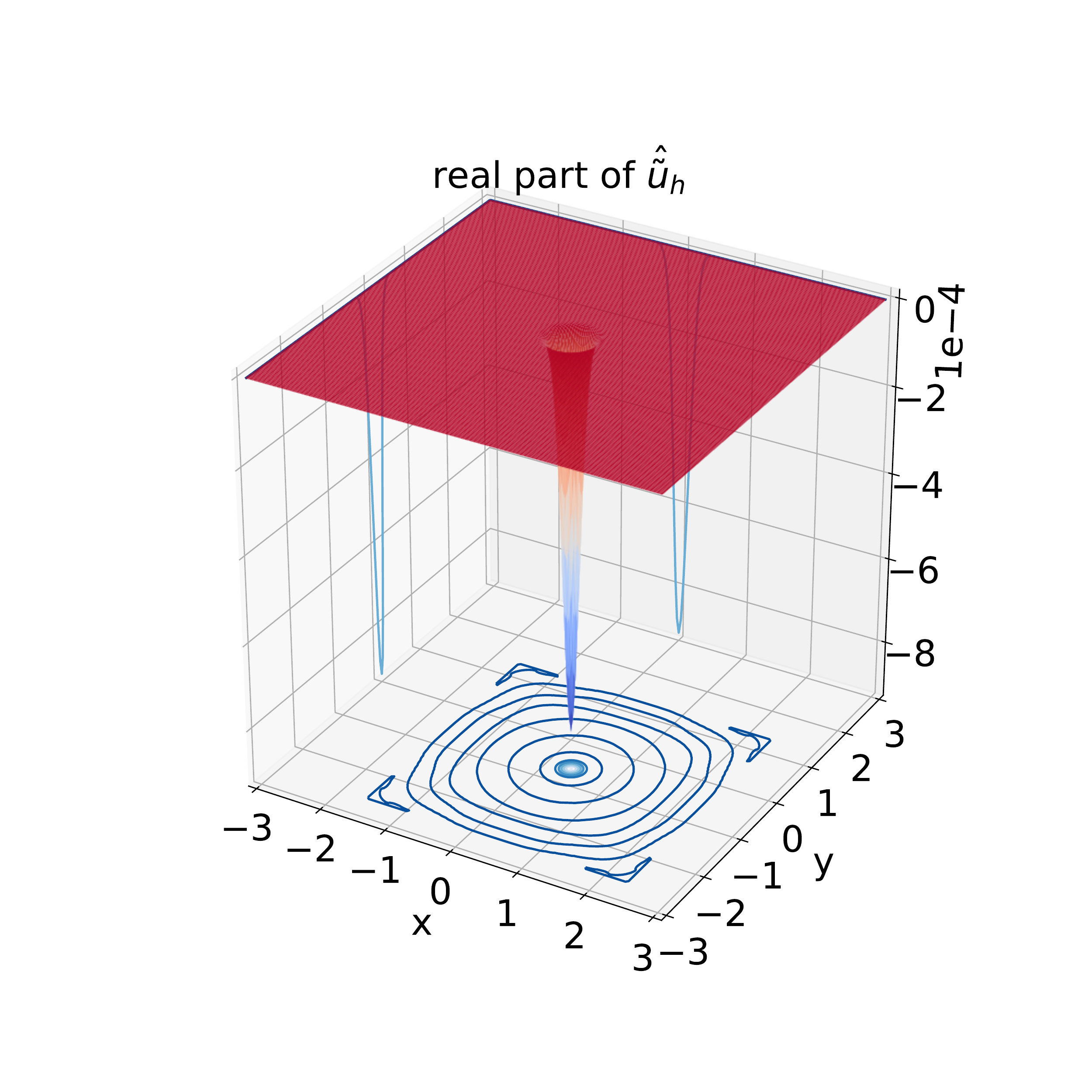}
	\includegraphics[width=0.49\textwidth]{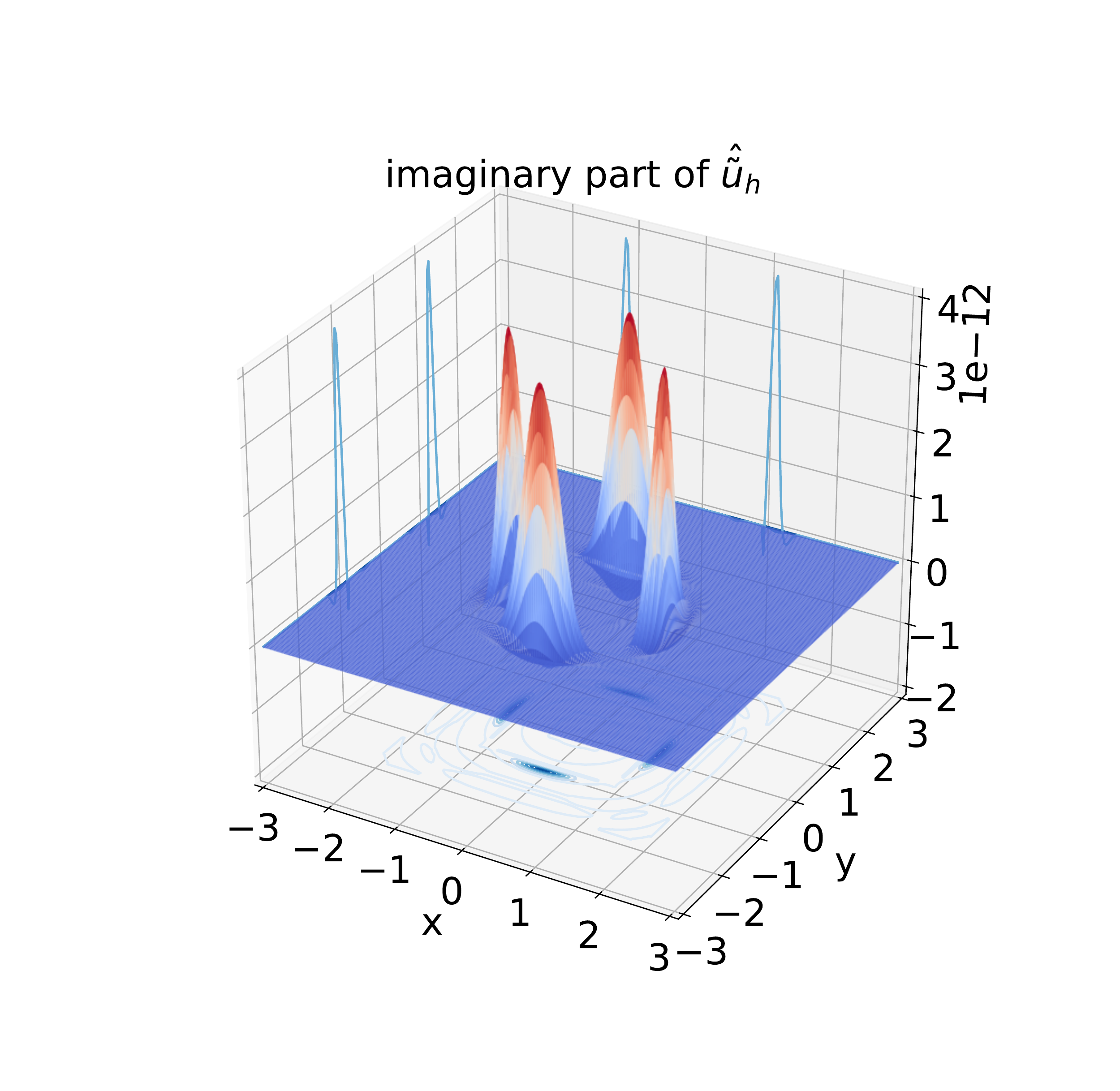}
 \caption{$k=20\pi,\delta=0.5$ by PMLs in Cartesian coordinates}
 \label{fig:ex2p2solutionsb}
\end{subfigure}
\begin{subfigure}{.48\textwidth}
 \centering
	\includegraphics[width=0.49\textwidth]{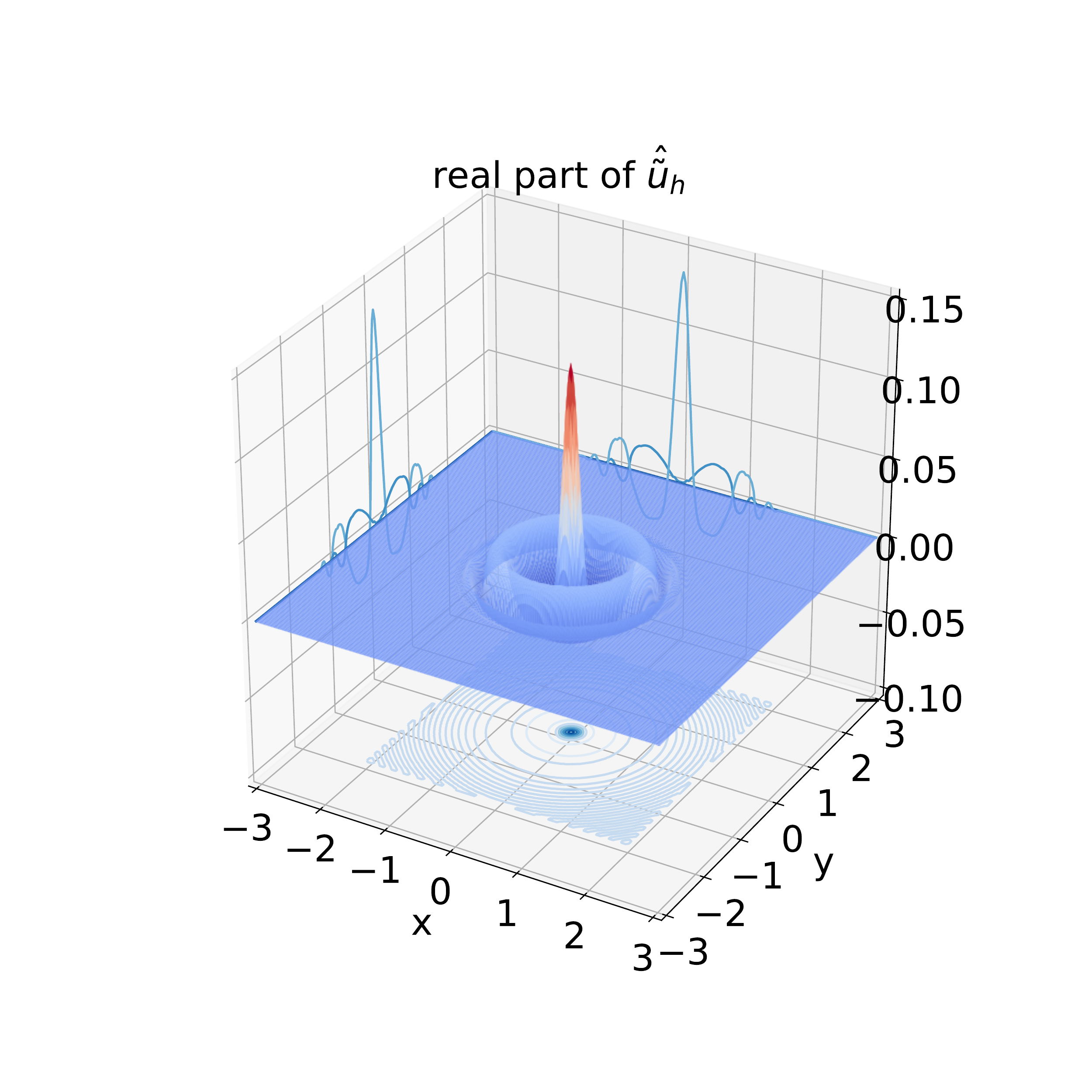}
	\includegraphics[width=0.49\textwidth]{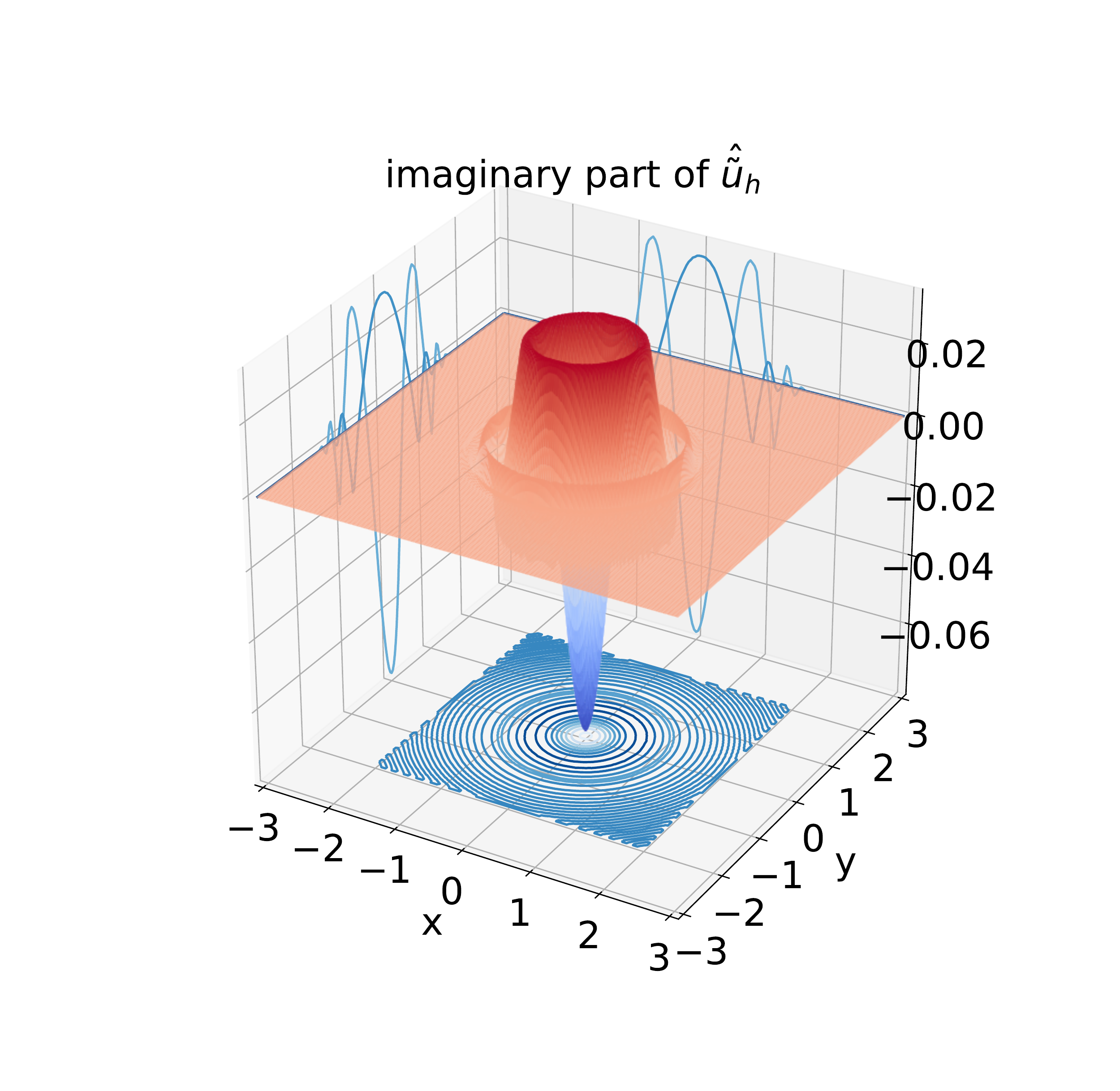}
 \caption{$k=2\pi,\delta=0.5$ by PMLs in polar coordinates}
 \label{fig:ex2p2solutionsc}
\end{subfigure}
\begin{subfigure}{.48\textwidth}
 \centering
	\includegraphics[width=0.49\textwidth]{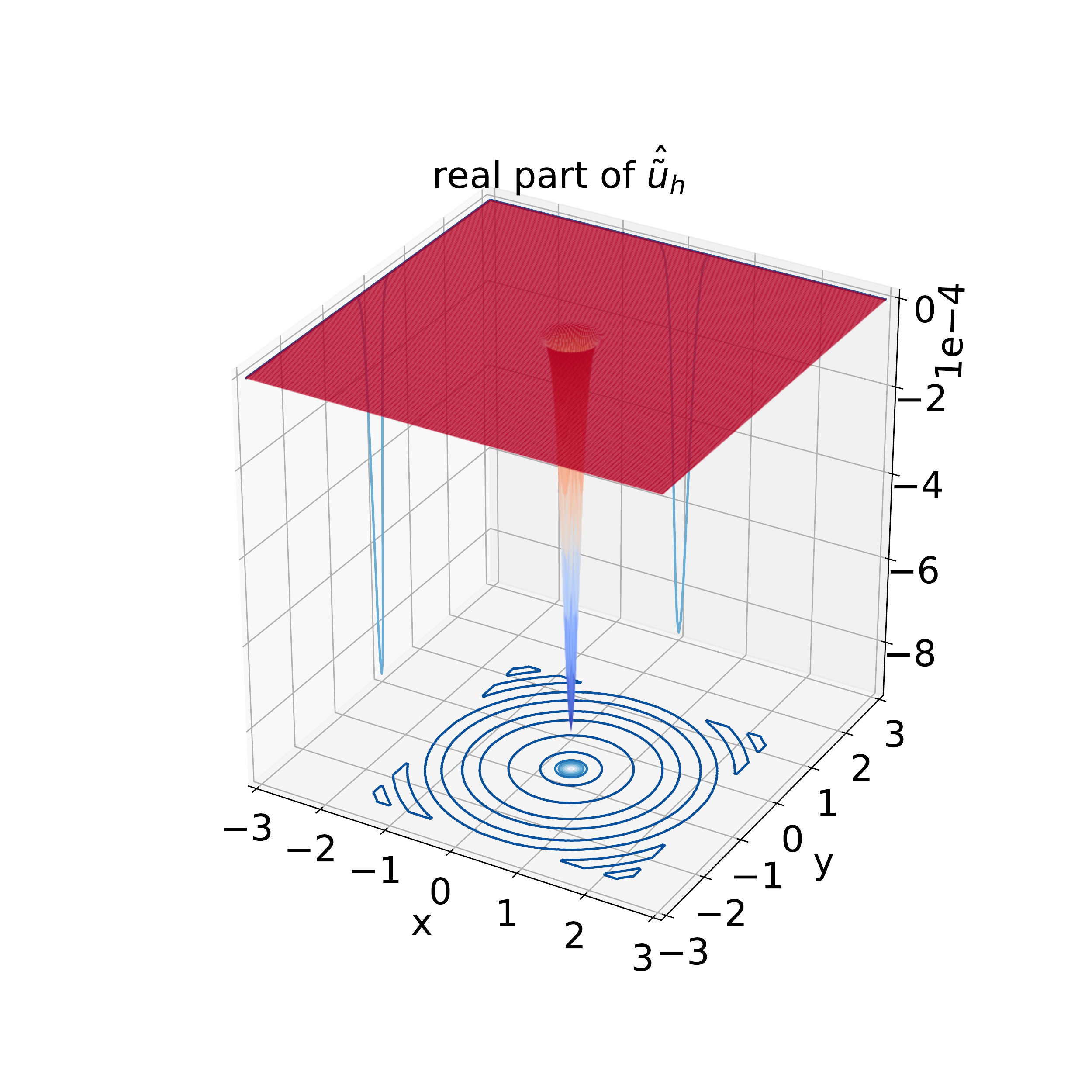}
	\includegraphics[width=0.49\textwidth]{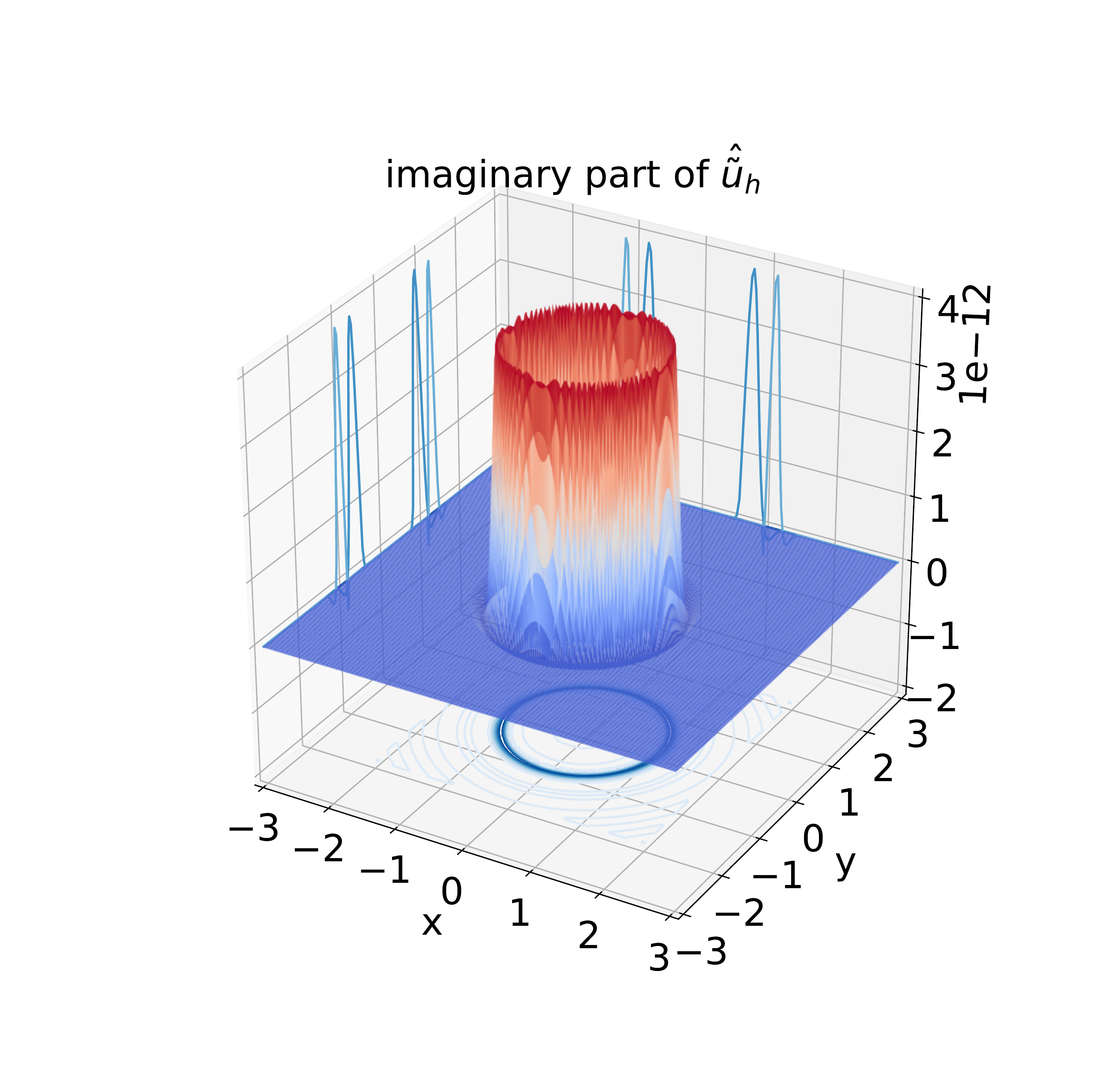}
 \caption{$k=20\pi,\delta=0.5$ by PMLs in polar coordinates}
 \label{fig:ex2p2solutionsd}
\end{subfigure}
	\caption{(Example 2.2) numerical solutions $\hat{\tilde u}_h$ for the piecewise constant kernel by PMLs in Cartesian coordinates and polar coordinates.} \label{fig:ex2p3solutions}
\end{figure}


\begin{table}[htbp]
\centering
\begin{tabular}{|c|c|c|c|c||c|c|c|c|}
\hline
 & \multicolumn{4}{c||}{$k=2\pi,\delta=0.5$}  & \multicolumn{4}{c|}{$k=20\pi,\delta=0.5$}\\
 \hline
\diagbox{$h$}{PML} & CPML & Order & PPML & Order & CPML & Order & PPML & Order \\
\hline
$2^{-1}$ &  1.67e-01 & -- & 7.54e-02 & -- &2.37e-06 & -- &1.02e-06 & -- \\
\hline
$2^{-2}$ & 1.78e-02 & 3.23 &  9.82e-03 & 2.94 &2.11e-07  & 3.49 &9.82e-08  & 3.37\\
\hline
$2^{-3}$ & 4.16e-03 &   2.10 & 1.42e-03 & 2.79 & 3.67e-08 & 2.52&1.78e-08 &2.47\\
\hline
$2^{-4}$ &  1.09e-03 & 1.92 & 3.26e-04 & 2.12 &7.13e-09 & 2.36 &3.51e-09 &2.34\\
\hline
\end{tabular}
\caption{(Example 2.2) $L^2$-errors of solutions for the piecewise constant kernel for PMLs. CPML and PPML are short for PMLs in Cartesian coordinates and polar coordinates, respectively.} \label{tab:ex2p3errors}
\end{table}

\textbf{Example 2.3}: We finally consider the fractional Helmholtz equation in 2D
\begin{align}
(-\Delta)^s u(x) -k^2 u(x) = f(x),\quad x\in \R^2,\ (0<s<1)
\end{align}
where
\begin{align*}
(-\Delta)^s u(x) = C(2,s)\ \mathrm{p.v.}\ \int_{\R^2} \frac{u(x)-u(y)}{|x-y|^{2+2s}}\dy,\quad C(2,s)=\frac{2^{2s}s\Gamma(s+1)}{\pi\Gamma(1-s)}.
\end{align*}
We set $z=10\i$, $l_r=l=d_{pml}=10$ and take the source function 
$$f(x) = \frac{\sqrt{\pi}}{5} e^{-\frac{\pi^2}{25}|x|^2}.$$

Figure~\ref{fig:ex2p4solutions} shows the corresponding solutions for $s=1/2$ and $k=2\pi/10, 32\pi/10$, and Table~\ref{tab:ex2p4errors} shows the errors and convergence orders by refining $h$. The behaviors of solutions are quite similar to those in the above examples.

\begin{figure}[htbp]
\centering
\begin{subfigure}{.48\textwidth}
 \centering
	\includegraphics[width=0.49\textwidth]{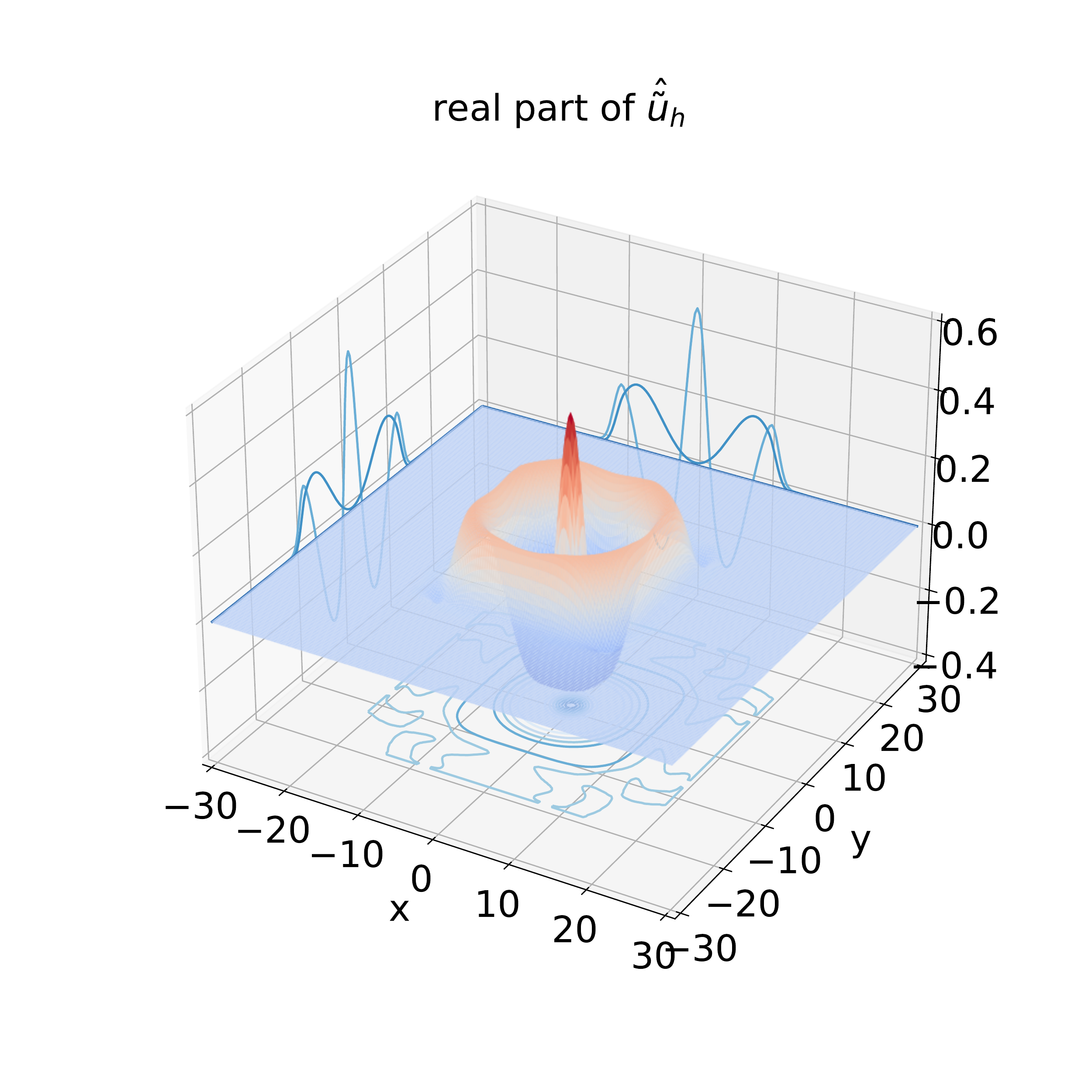}
	\includegraphics[width=0.49\textwidth]{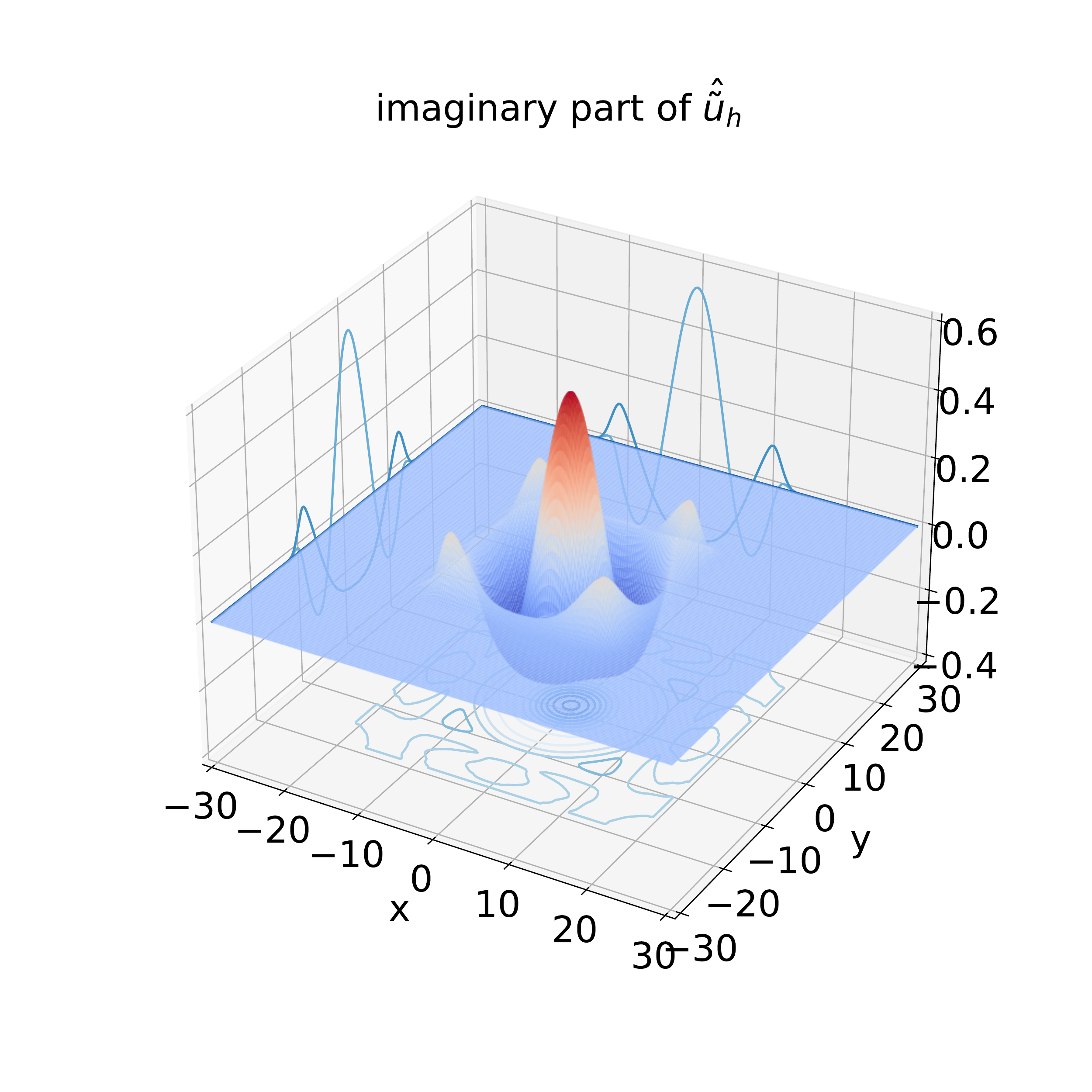}
 \caption{$k=2\pi/10,s=1/2$ in Cartesian coordinates}
 \label{fig:ex2p2solutionsa}
\end{subfigure}
\begin{subfigure}{.48\textwidth}
 \centering
	\includegraphics[width=0.49\textwidth]{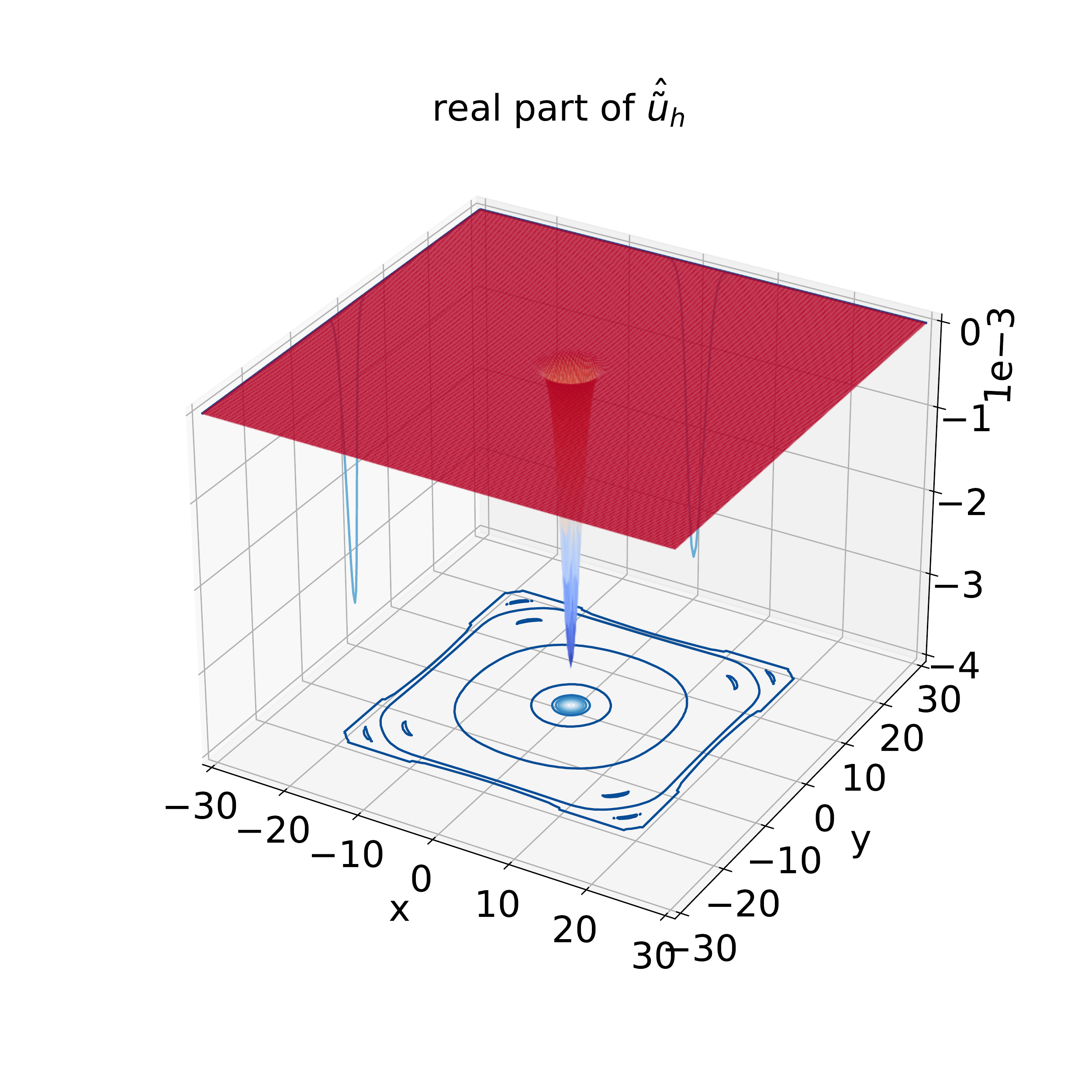}
	\includegraphics[width=0.49\textwidth]{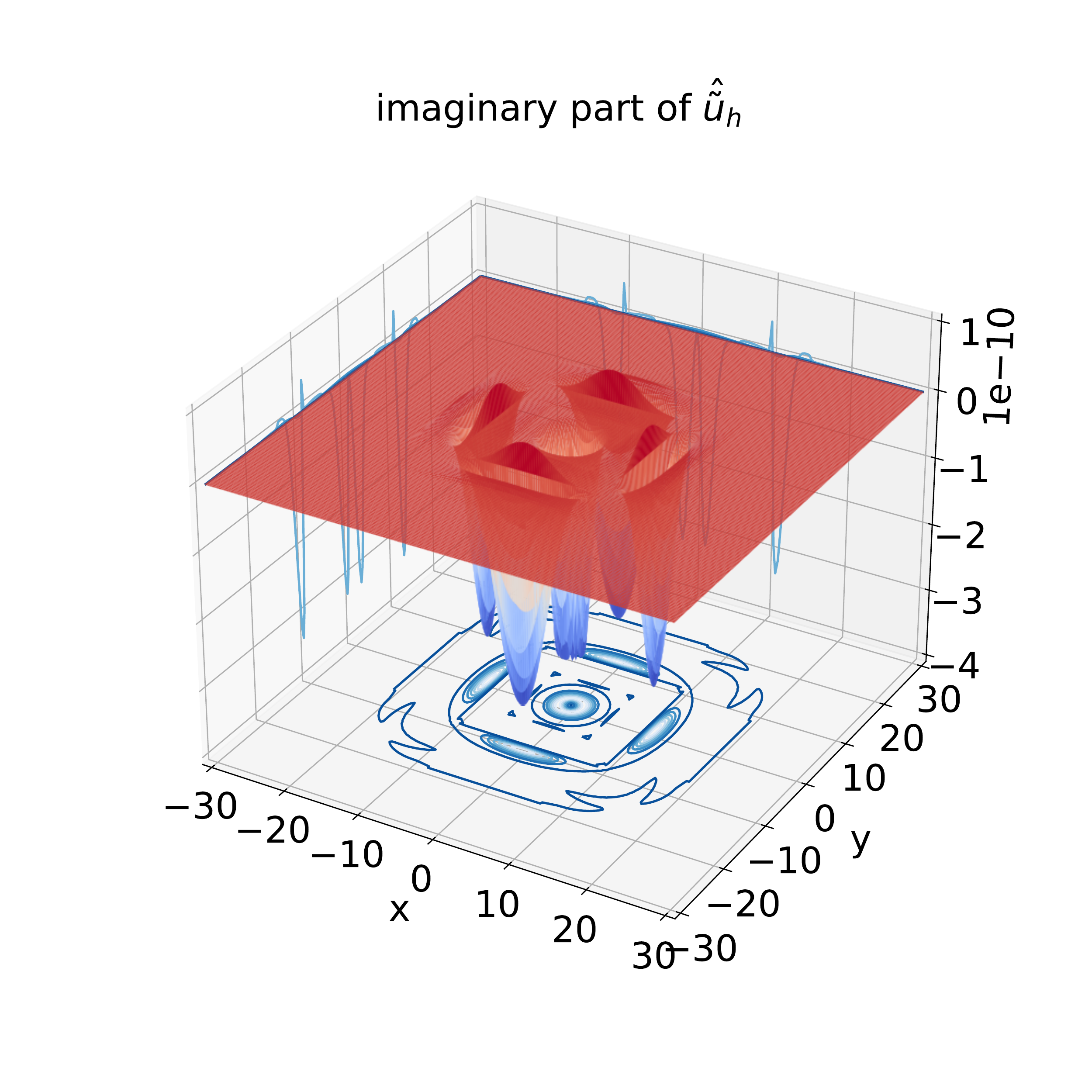}
 \caption{$k=32\pi/10,s=1/2$ in Cartesian coordinates}
 \label{fig:ex2p2solutionsb}
\end{subfigure}
\begin{subfigure}{.48\textwidth}
 \centering
	\includegraphics[width=0.49\textwidth]{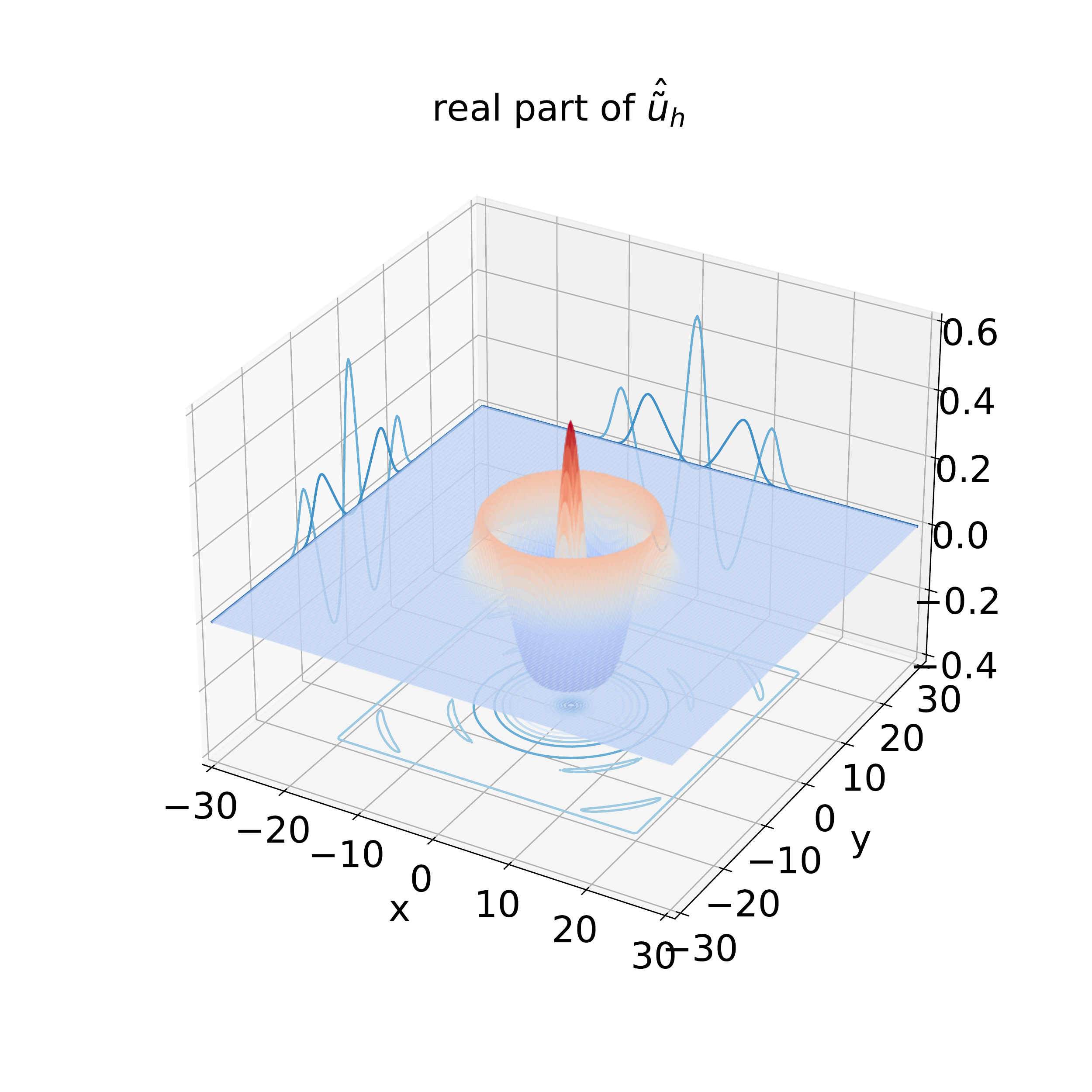}
	\includegraphics[width=0.49\textwidth]{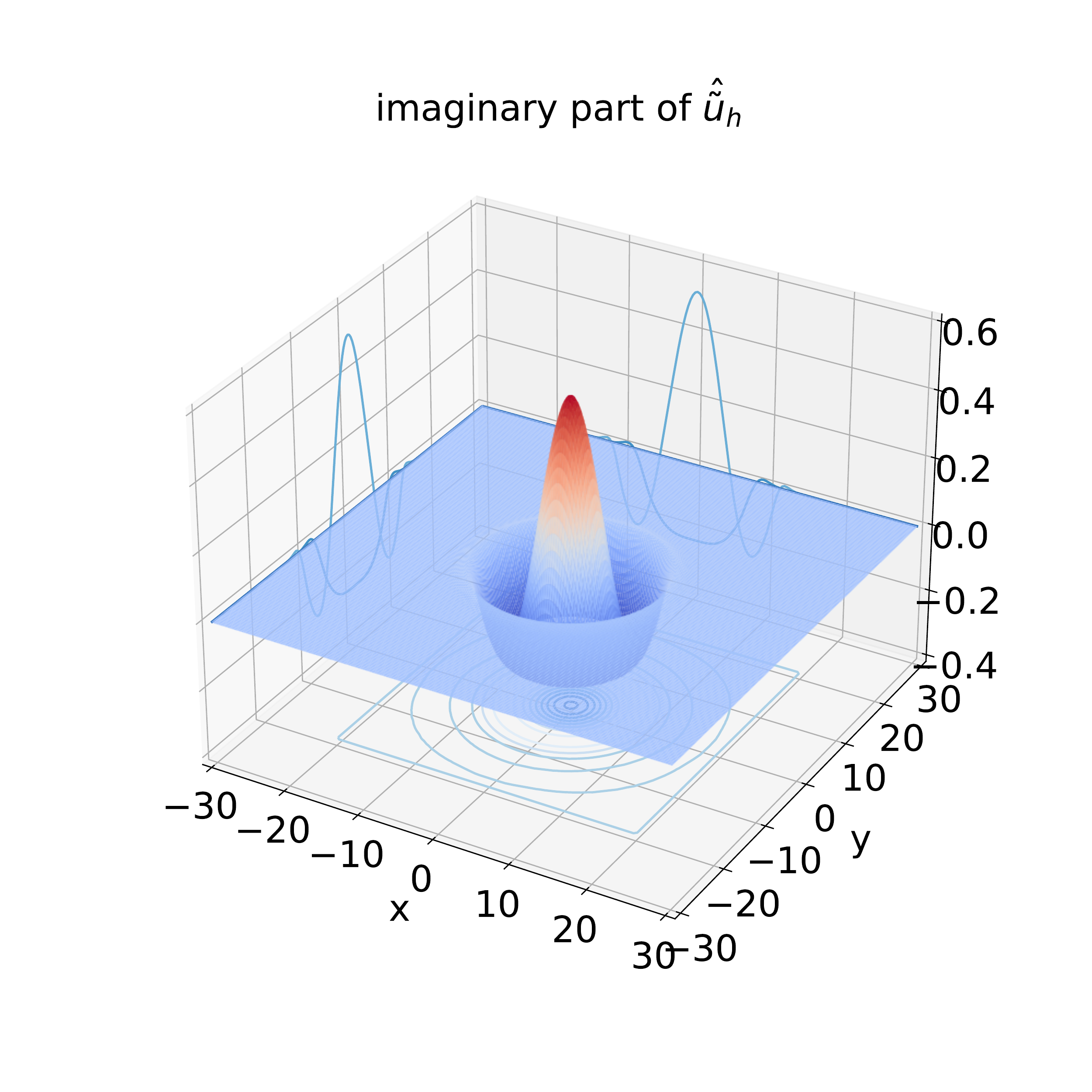}
 \caption{$k=2\pi/10,s=1/2$ in polar coordinates}
 \label{fig:ex2p2solutionsc}
\end{subfigure}
\begin{subfigure}{.48\textwidth}
 \centering
	\includegraphics[width=0.49\textwidth]{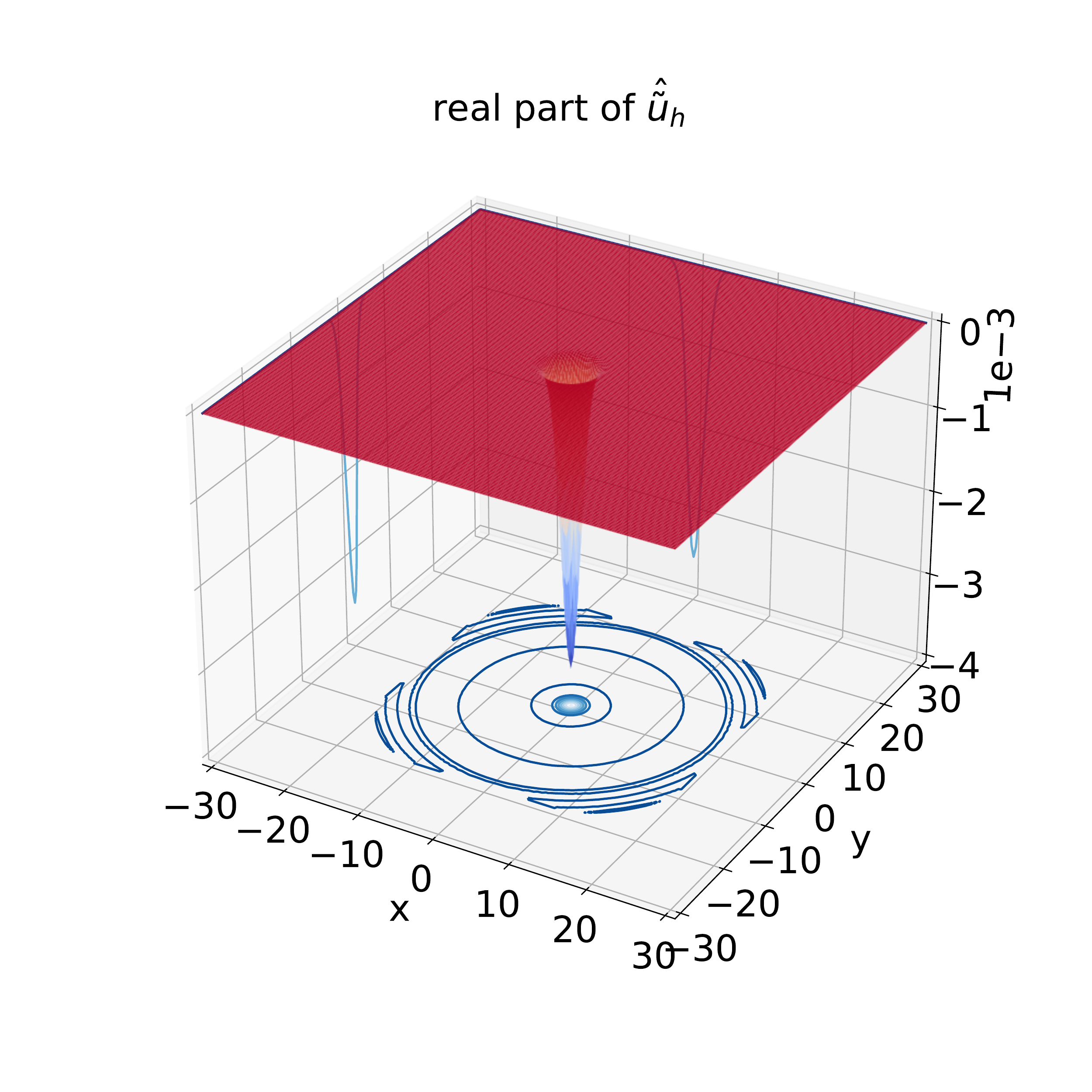}
	\includegraphics[width=0.49\textwidth]{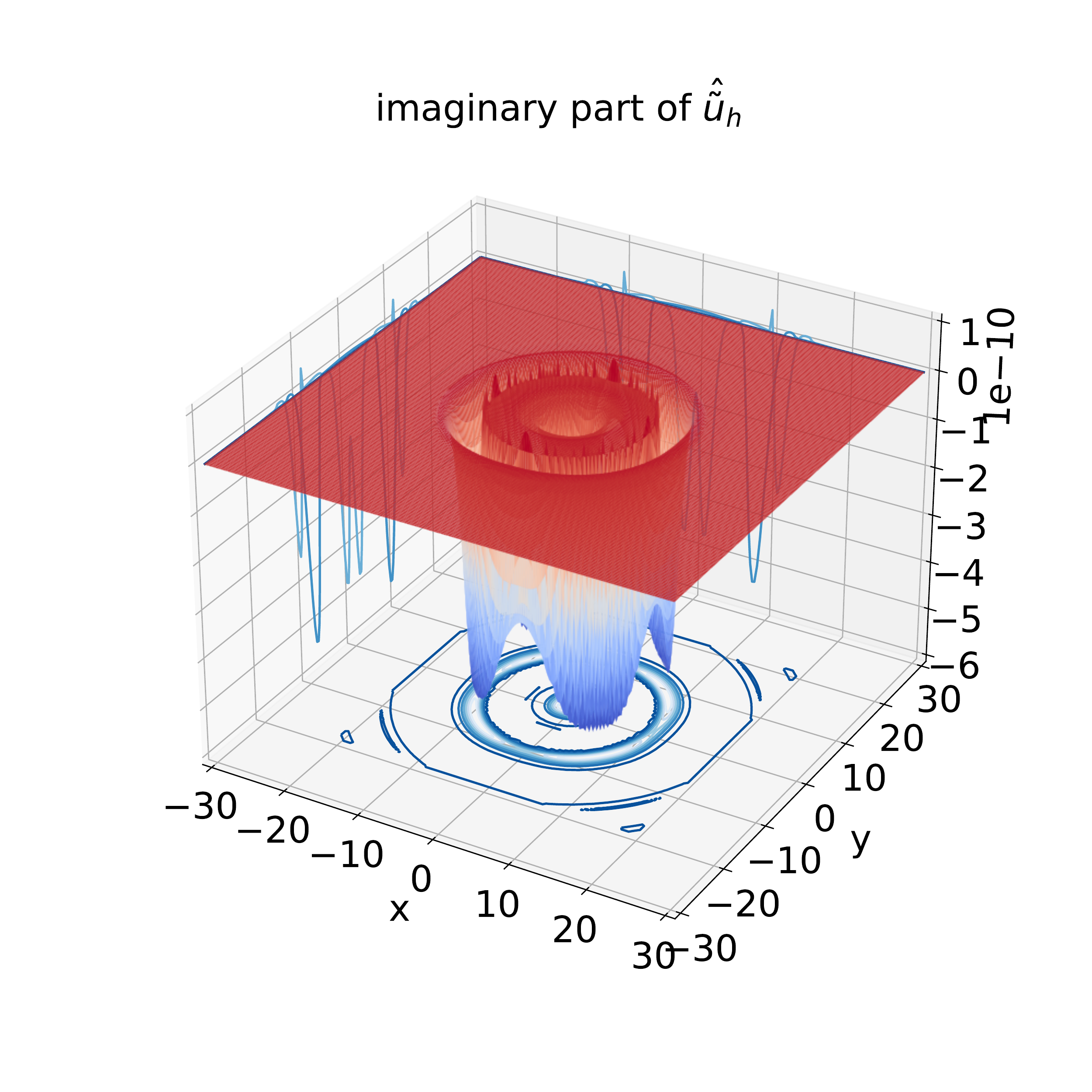}
 \caption{$k=32\pi/10,s=1/2$ in polar coordinates}
 \label{fig:ex2p2solutionsd}
\end{subfigure}
	\caption{(Example 2.3) numerical solutions $\hat{\tilde u}_h$ of fractional Helmholtz equations by PMLs in Cartesian coordinates and polar coordinates.} \label{fig:ex2p4solutions}
\end{figure}


\begin{table}[htbp]
\centering
\begin{tabular}{|c|c|c|c|c||c|c|c|c|}
\hline
 & \multicolumn{4}{c||}{$k=2\pi/10,s=1/2$}  & \multicolumn{4}{c|}{$k=32\pi/10,s=1/2$}\\
 \hline
\diagbox{$h(*10)$}{PML} & CPML & Order & PPML & Order & CPML & Order & PPML & Order \\
\hline
$2^{-1}$ &   2.63 & -- & 1.24 & -- & 8.71e-06 & -- & 3.73e-06 &--\\
\hline
$2^{-2}$ &  3.84e-01 & 2.77 & 1.04e-01 & 3.58 & 2.08e-06  & 2.06 & 9.71e-07  & 1.94\\
\hline
$2^{-3}$ & 6.48e-02 & 2.56 & 2.45e-02 & 2.08 & 3.22e-07& 2.69 & 1.56e-07 & 2.63\\
\hline
$2^{-4}$ &   1.57e-02 & 2.03 & 6.89e-03 & 1.83 & 5.82e-08 & 2.46 & 2.86e-08 & 2.44\\
\hline
\end{tabular}
\caption{(Example 2.3) $L^2$-errors for solutions of fractional Helmholtz equations by PMLs. CPML and PPML are short for PMLs in Cartesian coordinates and polar coordinates, respectively.} \label{tab:ex2p4errors}
\end{table}

\section{Conclusion}
The PML for the nonlocal Helmholtz equation on unbounded domains is studied. A 1D PML was introduced which has a more flexible PML coefficient comparing with our previous work \cite{DuZhangNonlocal1}. This PML can be directly used to solve the nonlocal equation for all \emph{wavenumber} $k$, whose efficiency was presented in the theoretical results and numerical tests. Meanwhile, we constructed the general form of the 2D nonlocal PML by introducing two common PMLs: one is obtained to stretch Cartesian coordinates and the other is to stretch the radial coordinate in polar coordinates. We also analyzed the behaviors of nonlocal solutions for radial kernel in 1D and showed PML absorbs incoming waves. Finally, AC discretization schemes were used to numerically solve the truncated PML problems.

There are a number of interesting issues to be studied further. First, the far field boundary condition for the nonlocal Helmholtz equation is fundamental but remains open. Second, the theoretical analysis of nonlocal solutions for general kernels in 1D and 2D are still to be studied. Third, numerical tests show that the numerical solution suffers the pollution effects and numerical reflections, which needs expensive computational cost for the huge linear algebraic system resulting from its numerical discretization. Therefore, more investigations on fast solvers for the discrete scheme are necessary and interesting. 
In addition, using the idea in this paper to construct ABCs for nonlocal wave equations will be our future work.

\bibliographystyle{siam}
\bibliography{referrence}

\end{document}